\documentclass{article}

\usepackage{amsmath,amstext,amssymb,amsfonts,dsfont}
\usepackage{color,graphicx}
\usepackage{tabularx}
\usepackage{verbatim}
\usepackage{pgf,tikz}
\usepackage{mathrsfs}
\usepackage{algorithm2e}
\usepackage{scalerel,amssymb}
\usepackage{algorithmic}
\usetikzlibrary{arrows}
\usetikzlibrary{decorations.pathreplacing,angles,quotes}

\usepackage{amsthm}
\usepackage{ifdraft}

\usepackage{nicefrac}

\newtheorem{theorem}{Theorem}[section]
\newtheorem*{theorem*}{Theorem}

\newtheorem*{claim*}{Claim}

\newtheorem{proposition}[theorem]{Proposition}
\newtheorem*{proposition*}{Proposition}

\newtheorem*{lemma*}{Lemma}

\newtheorem*{fact*}{Fact}

\newtheorem*{hypothesis*}{Hypothesis}

\newtheorem{definition}[theorem]{Definition}

\newtheorem{remark}[theorem]{Remark}
\newtheorem{assumption}[theorem]{Assumption}

\usepackage{prettyref}
\newcommand{\savehyperref}[2]{\texorpdfstring{\hyperref[#1]{#2}}{#2}}

\newrefformat{eq}{\savehyperref{#1}{\textup{(\ref*{#1})}}}
\newrefformat{lem}{\savehyperref{#1}{Lemma~\ref*{#1}}}
\newrefformat{lemma}{\savehyperref{#1}{Lemma~\ref*{#1}}}
\newrefformat{def}{\savehyperref{#1}{Definition~\ref*{#1}}}
\newrefformat{thm}{\savehyperref{#1}{Theorem~\ref*{#1}}}
\newrefformat{cor}{\savehyperref{#1}{Corollary~\ref*{#1}}}
\newrefformat{cha}{\savehyperref{#1}{Chapter~\ref*{#1}}}
\newrefformat{sec}{\savehyperref{#1}{Section~\ref*{#1}}}
\newrefformat{section}{\savehyperref{#1}{Section~\ref*{#1}}}
\newrefformat{app}{\savehyperref{#1}{Appendix~\ref*{#1}}}
\newrefformat{tab}{\savehyperref{#1}{Table~\ref*{#1}}}
\newrefformat{fig}{\savehyperref{#1}{Figure~\ref*{#1}}}
\newrefformat{hyp}{\savehyperref{#1}{Hypothesis~\ref*{#1}}}
\newrefformat{alg}{\savehyperref{#1}{Algorithm~\ref*{#1}}}
\newrefformat{sdp}{\savehyperref{#1}{SDP~\ref*{#1}}}
\newrefformat{rem}{\savehyperref{#1}{Remark~\ref*{#1}}}
\newrefformat{item}{\savehyperref{#1}{Item~\ref*{#1}}}
\newrefformat{step}{\savehyperref{#1}{step~\ref*{#1}}}
\newrefformat{conj}{\savehyperref{#1}{Conjecture~\ref*{#1}}}
\newrefformat{fact}{\savehyperref{#1}{Fact~\ref*{#1}}}
\newrefformat{prop}{\savehyperref{#1}{Proposition~\ref*{#1}}}
\newrefformat{claim}{\savehyperref{#1}{Claim~\ref*{#1}}}
\newrefformat{relax}{\savehyperref{#1}{Relaxation~\ref*{#1}}}
\newrefformat{red}{\savehyperref{#1}{Reduction~\ref*{#1}}}
\newrefformat{part}{\savehyperref{#1}{Part~\ref*{#1}}}
\newrefformat{prob}{\savehyperref{#1}{Problem~\ref*{#1}}}
\newrefformat{ass}{\savehyperref{#1}{Assumption~\ref*{#1}}}
\newrefformat{cond}{\savehyperref{#1}{Condition~\ref*{#1}}}

\newcommand{\Sref}[1]{\hyperref[#1]{\S\ref*{#1}}}

\usepackage[varg]{txfonts}    

\renewcommand{\mathbb}{\varmathbb} 

\renewcommand{\leq}{\leqslant}

\renewcommand{\geq}{\geqslant}

\usepackage{bm}

\usepackage{xspace}

\usepackage[pdftex,pagebackref,colorlinks,linkcolor=blue,filecolor = blue, citecolor = blue, urlcolor  = blue]{hyperref}

\usepackage{boxedminipage}

\newcommand{\abs}[1]{\left\lvert#1\right\rvert}

\newcommand{\floor}[1]{\lfloor #1 \rfloor}

\newcommand{\norm}[1]{\left\lVert#1\right\rVert}

\newcommand{\R}{\mathbb R}

\definecolor{DSgray}{cmyk}{0,0,0,0.7}

\usepackage{arxiv}

\usepackage[utf8]{inputenc}
\usepackage{amsmath}
\usepackage{amssymb}
\usepackage{amsthm}
\usepackage{bbm}
\usepackage{bm}
\usepackage{color}

\newcommand{\II}{\mathbbm{1}}
\newcommand{\EE}{\mathbb{E}}
\newcommand{\RR}{\mathbb{R}}

\newcommand{\Pro}{\mathbb{P}}

\newcommand{\LL}{\mathcal{L}}

\newcommand{\CC}{\mathcal{C}}

\newcommand{\FF}{\mathcal{F}}

\newcommand{\HH}{\mathcal{H}}

\newcommand{\MM}{\mathcal{M}}

\newcommand{\Scal}{\mathcal{S}}
\newcommand{\NN}{\mathcal{N}}

\newcommand{\UU}{\mathcal{U}}

\renewcommand{\b}[1]{\left\{ #1\right\}}
\renewcommand{\c}[1]{\left[ #1\right]}
\newcommand{\p}[1]{\left( #1\right)}
\newcommand{\iid}{\overset{i.i.d}{\sim}}

\newcommand{\inner}[1]{\left\langle#1\right\rangle}

\newtheorem{theo}{Theorem}[section]

\newtheorem{lem}[theo]{Lemma}

	\title{Nonparametric regression on random geometric graphs sampled from submanifolds}
	
\author{Paul Rosa \\
 Department of Statistics\\
 University of Oxford\\
 \texttt{paul.rosa@stats.ox.ac.uk} \\
  \And
 Judith Rousseau\\
 Department of Statistics
 University of Oxford and \\
  CEREMADE, CNRS, Universit\'e Paris-Dauphine, PSL University}

\begin{document}

	\maketitle

	\begin{abstract}
	We consider the nonparametric regression problem when the covariates are located on an unknown compact submanifold of a Euclidean space. Under defining a random geometric graph structure over the covariates we analyse the asymptotic frequentist behaviour of the posterior distribution arising from Bayesian priors designed through random basis expansion in the graph Laplacian eigenbasis. Under H\"older smoothness assumption on the regression function and the density of the covariates over the submanifold, we prove that the posterior contraction rates of such methods are minimax optimal (up to logarithmic factors) for any positive smoothness index.
	\end{abstract}
	
\section{Introduction}\label{section:intro}

When the data is high dimensional, it is common practice to assume a lower dimensional structure. This idea dates back at least to Pearson's principal component analysis \cite{pearson_liii_1901} and the Johnson-Lindenstrauss lemma \cite{johnson_extensions_1986,ailon_fast_2009} where a linear projection onto a lower dimensional subspace of variables is performed while still retaining much of the statistical information. However, a linear constraint on the possible intrinsic shape of the data can be restrictive, and nonlinear dimensionality reduction has received much attention during the past two decades. Examples include kernel PCA \cite{scholkopf_nonlinear_1998}, Isomap \cite{tenenbaum_global_2000}, locally linear embeddings \cite{roweis_nonlinear_2000}, Laplacian eigenmaps and diffusion maps \cite{coifman_diffusion_2006}. These techniques typically take into account intrinsic geometric properties of the data. 

 Inference of intrinsic geometric properties of data is often believed to be a powerful way to gain insight in complex high dimensional statistical problems. Examples include its support \cite{genovese_minimax_2012,aamari_nonasymptotic_2019} under the Hausdorff metric, differential quantities \cite{aamari_nonasymptotic_2019,aamari_optimal_2023}, the geodesic distance \cite{aamari_optimal_2023,bernstein_graph_nodate}, the intrinsic dimension \cite{kim_minimax_2019,denti_generalized_2022} or even the Laplace-Beltrami operator eigenpairs \cite{garcia_trillos_error_2020,calder_improved_2020,dunson_spectral_2021,wormell_spectral_2021}. In the same way, inferring the intrinsic topological structure in the data has also received a lot of interest since the modern development of topological data analysis as field within data science, with in particular the study of the persistence diagrams and modules \cite{chazal_introduction_2021,noauthor_structure_nodate,divol_understanding_2021,divol_density_2019,loiseaux_framework_2023,loiseaux_stable_2024}.

The Laplacian eigenmaps, introduced initially by \cite{belkin_laplacian_2001,belkin_laplacian_2003} is a popular non linear method for dimensionality reduction, together with a tool for inference in high dimension setups. It has also been used successfully in the context of spectral clustering \cite{liu_spectral_2014,von_luxburg_tutorial_2007}, supervised and semisupervised learning \cite{green_minimax_2021,shi_adaptive_2023,sanz-alonso_unlabeled_2022} : in this paper we consider the use of Laplacian eigenmaps in the context of  the reknown nonparametric regression problem where 
\begin{equation}\label{reg}
y_i = f(x_i) + \varepsilon_i, \quad \varepsilon_i\stackrel{iid}{\sim}\mathcal N(0, \sigma^2) , \quad i\leq n;
\end{equation}
with covariates $x_i \in \mathbb R^D$ and $D$ is possibly large. Then, a way to formalize a low dimensional structure to the problem is to assume that the covariates $X_i$'s belong to some low dimensional submanifold  $\mathcal M$ of $\mathbb R^D$, also known as the manifold hypothesis. This is the setup we are considering.  We also consider both the supervised learning setup where the observations consist of $(y_i, x_i)_{i=1}^n$, under model \eqref{reg}; and the semi-supervised setup where in addition another sample $(x_i)_{i=n+1}^N$ of covariates are observed without their labels $y_i$.

Laplacian eigenmaps approaches are based on the construction of  a graph $G = (V, E)$ whose vertices are the covariates $(x_i)_{i=1}^N$ and where an edge between two vertices $x_i $ and $x_j$ exists if and only if $\|x_i-x_j\|\leq h$ for some predefined threshold $h$. From that the Laplacian of the graph $L : L^2\p{V} \to L^2 \p{V}$, is constructed together with its spectral decomposition 
$L = \sum_{j=1}^N \lambda_j \inner{u_j|\cdot}_{L^2\p{V}} u_j, 0\leq \lambda_1\leq  \cdots \leq \lambda_N$ for some appropriate Euclidian space structure on $V$ to be defined in Section \ref{section:model_and_notations}. Dimensionality reduction is then obtained by considering the projection operator
$\sum_{j=1}^J \inner{u_j|\cdot}_{L^2\p{V}} u_j$ for some truncation level $J$. This dimension reduction can be used in different inferential contexts

In the context of the regression problem \eqref{reg}, the vector $f_N = (f(x_1) , \cdots, f(x_N))^T$ is modelled as an element of $\text{span}( u_1, \cdots, u_J)$. 
This construction has been studied with variants in the definition of the graph Laplacian $L$,   for instance by \cite{green_minimax_2021,shi_adaptive_2023}, where the PCR-LE estimator $\hat f_N$, i.e. it  is the minimizer of $\|Y - f_N\|^2$ over $\Sigma^J := \text{span}( u_1, \cdots, u_J)$ (in the fully supervised case $N=n$) while   \cite{sanz-alonso_unlabeled_2022,fichera_implicit_2024,dunson_graph_2022} are using a Bayesian estimation procedure based on Gaussian processes defined via $L$.\\
In this paper we are interested in understanding the capacity  of  graph Laplacian regression methods of capturing both the unknown low dimensional submanifold  support $\mathcal M$ of the design, together with the unknown smoothness $\beta$ of the regression functions. \\

\textbf{Related works :} \\

The asymptotic behaviour of least square estimators has been studied in \cite{green_minimax_2021,shi_adaptive_2023}  where they  obtain minimax  convergence rates when the support $\mathcal M$ of the covariates is an open set of $\mathbb R^D$ and under a Sobolev regularity condition. Extension to the case where $\mathcal M$ is a $d$ dimensional submanifold of $\mathbb R^D$, with $d < D$, is also considered in \cite{green_minimax_2021} however only for regression functions which have Sobolev smoothness restricted to the set $\{1,2,3\}$. \\
Bayesian estimators of $f$ based on the graph Laplacian have been considered in the literature as well for instance in  \cite{sanz-alonso_unlabeled_2022,fichera_implicit_2024,dunson_graph_2022}, however the asymptotic properties in terms of posterior contraction rates around the true regression function have only been derived in some restricted cases. In \cite{dunson_graph_2022}, it is shown that in the case $N=n$ and when the true regression vector $f_{0,N} = \p{f_0(x_i)}_{i=1}^N$ belongs to $\Sigma^J$ for $J \lesssim \p{\frac{n}{\ln n}}^{\frac{d}{2q+2d}}, q \geq \frac{d}{2}$ then the posterior contraction rate (with respect to the empirical $L^2$ norm) is upper bounded by a multiple of $\p{\frac{\ln n}{n}}^{\frac{d}{2q+2d}}$, under some kind of separability assumption between the eigenvalues of the Laplace-Beltrami operator of $\MM$. Assuming that $f_{0,N} \in \Sigma^J$ avoids studying the approximation of $f_{0,N}$ by an element in $\Sigma^J$, which is a nontrivial problem where one only has access to smoothness assumptions on the function $f_0 : \MM \to \RR$. A step in this direction has been made in \cite{sanz-alonso_unlabeled_2022}, where it is established that if the number of unlabeled covariates $\b{x_{n+1},\ldots,x_N}$ is large enough, i.e at least $N \geq n^{2d}$, then posterior contraction at a minimax optimal rate (up to logarithmic factors) of the true regression vector $f_{0,N}$ is possible, provided that its smoothness $\beta$ (more precisely : $f_0$ belongs to the Besov space $B_{\infty \infty}^\beta \p{\MM}$, see \cite{coulhon_heat_2012}) satisfies $\beta > d-\frac{1}{2}$ using a graph Mat\'ern process. Notably, as is pointed out in the paper, their assumptions rule out the fully supervised setting and every other intermediate cases. Furthermore, adaptation to the smoothness $\beta$ is not discussed. In both papers \cite{dunson_graph_2022,sanz-alonso_unlabeled_2022}, the authors control the contraction of the posterior distribution on $f_N$ by controlling the convergence of the graph Laplacian (or the associated heat kernel) to the  Laplacian-Beltrami operator of $\mathcal M$. As we will see, it is not fundamentally needed to study the nonparametric regression problem; this is also the point of view of \cite{green_minimax_2021,shi_adaptive_2023}.\\
Thus, the question the approximation of $f_{0,N}$ by elements of $\Sigma^J$ for $J = o(N)$, when $f_0$ is a smooth function on an unknown manifold $\MM$ remains unanswered. In this paper we therefore aim to bridge this gap in the literature.\\

\textbf{Our contributions :} \\
In this paper, our main aim is to study posterior contraction rates for priors based on the graph Laplacian for the recovery of the regression vector $f_{0,N}$. To do that we first show that any $\beta$ H\"older function on $\mathcal M,~\beta > 0$ can be well approximated (in some sense to be made precise later in the paper) by a function in $\Sigma^J$  for $J\geq n^{d/(2\beta+d)}$ (up to $\log n $ terms) , where $d$ is the intrinsic dimension of the submanifold $\mathcal M$. To do that we use a novel approximation argument by first constructing an approximation of $f_{0,N}$ by a vector of the form $e^{-tL/h^2}f_t$ for some well chosen $t>0$ and $f_t \in \RR^V$. Thanks to the exponential factor we can then control well the difference between $e^{-tL/h^2}f_t$ and its projection on $\Sigma^J$. This approximation argument is valid for any $\beta > 0$. This result has an interest in its own right; in particular we apply it to the frequentist PCR-LE estimator.

We also propose general classes of priors on $f_N$ supported on $\Sigma^J$ for which we characterize the posterior contraction rates, both for the empirical loss $\| f-f_0\|_n^2= \sum_{i=1}^n (f(x_i)-f_0(x_i))^2/n$ and for the global loss $\| f-f_0\|_N^2= \sum_{i=1}^N (f(x_i)-f_0(x_i))^2/N$, as functions of the truncation level $J$ and the connectivity parameter $h$.

Finally by considering hyperpriors on $h$ and $J$ we propose a novel class of priors that achieve minimax adaptive posterior contraction rates (up to $\ln n$ terms), both under the $\| \cdot \|_n$ and the $\|\cdot \|_N$ (pseudo-)distances.

Hence, contrarywise to \cite{sanz-alonso_unlabeled_2022} we do not need a very large number of unlabelled covariates which were used to ensure an accurate discrete to continuum approximation of the graph Laplacian eigenpairs and compared to \cite{green_minimax_2021,shi_adaptive_2023,sanz-alonso_unlabeled_2022} we significantly weaken the smoothness assumptions on $f_0$ (our assumptions actually match those considered in the continuous setting, see e.g \cite{vaart_information_2011,rosa_posterior_2023}). Also note that we do not require separability assumptions on the eigenvalues of the Laplace-Beltrami operator on $\MM$, as opposed to \cite{dunson_graph_2022}.
	
	\textbf{Outline :} We start by describing our geometrical and statistical model as well as our notations in Section \ref{section:model_and_notations}. We then describe our priors and the associated posterior contraction results in Section \ref{section:main_results}; with non adaptive results in Section \ref{subsubsection:rates_non_adaptive_priors} and adaptive results in Section \ref{subsubsection:rates_adaptive_priors}. Convergence rates for the PCR-LE estimator are provided in Section \ref{subsubsection:rates_non_adaptive_priors}. The approximation theory for $f_{0,N}$ by $\Sigma^J$ is provided in Section \ref{subsection:approximation_results}. Main proofs are provided in Section \ref{sec:main_proofs}. Finally, the appendix contains additional proofs.

	\section{Model and notations}\label{section:model_and_notations}
	
	In this paper we consider the semi-supervised nonlinear regression model where we observe labelled data $(y_i, x_i)_{i\leq n}$ together with unlabelled data $(x_i)_{i=n+1}^N$ (in particular $N = N_n \geq n$ depends on $n$). Note that $N=n$ corresponds to the supervised case. In this work we make the following growth assumption on $N$ : there exists $b>0$ such that $N_n \leq n^b$ for all $n$. Note that in particular we always have the inequality $\ln n \leq \ln N \leq b \ln n$. The nonparametric regression model is then 
	\begin{equation}\label{model:nonparametric_regression}
	\begin{split}
			y_i &= f_0(x_i) + \varepsilon_i,\quad  \varepsilon_i \iid \NN \p{0,\sigma^2}, \quad i = 1, \ldots, n , \\
			x_i &\iid \mu_0 , \quad  i = 1,\ldots,N
	\end{split}
		\end{equation}	
	where $\sigma > 0$, $ f_0 \in \CC^\beta \p{\MM}$ with $\MM$ an unknown compact connected submanifold of regularity $\alpha \geq \p{\beta+3} \vee 6$ of $\mathbb R^D$ (see appendix \ref{sec:regularity_M_holder_spaces} for a precise definition of $\alpha$ and the space $\CC^\beta\p{\MM}$) and $\mu_0$ is a probability distribution on $\MM$ (the connectedness assumption could actually be dropped by working on each connected component separately). We assume that $\mu_0$ is absolutely continuous with respect to the volume measure on $\MM$, with density $ p_0$ and $f_0, \MM, p_0$ are unknown quantities. Without loss of generality we consider the case where $\sigma$  is known, since the case of unknown $\sigma$ can be easily derived from the results here, using the approach of  \cite{naulet_aspects_2018}. Moreover, it would also be possible to handle non Gaussian errors as done in Ghosal \& Van der Vaart \cite{ghosal_fundamentals_2017}, section 8.5.2.

	We consider an unweighted random geometric graph $G = G^{(h)} = \p{V,E^{(h)}}$ based on $x_{1:N} = (x_i)_{i=1}^N$  generated as follows: the vertex set is $V = \b{x_1,\ldots,x_N}$ and the edges are defined by $E_{ij} = E_{ij}^{(h)} = 1$ if and only if  $\norm{x_i - x_j} < h$, for some $ h>0$, where $\norm{\cdot}$ denotes the Euclidean distance in $\R^D$. Hereafter we denote $x_i \sim x_j \iff E_{ij}=1$.

	We define the \emph{normalized graph Laplacian}
	\[
	L = L^{(h)} = I - D^{-1}A
	\]
	where $D = D^{(h)} = diag(\mu)$ is the \emph{degrees matrix} $D_{xy} = \delta_x^y \mu_x, \mu_x = \mu_x^{(h)} = \# \b{y : \norm{x-y} < h}$ and $A = A^{(h)}$ is the \emph{adjacency matrix} $A_{xy} = A_{xy}^{(h)} = \mu_{xy} = \II{}_{x \sim y}$. We also define the \emph{normalized degree measure} $\nu = \nu^{(h)} = \frac{\mu}{\mu \p{V}}$. The graph Laplacian $L$ is a nonnegative self-adjoint operator from $L^2 \p{\nu}$ to itself, where the inner product is given by
	\[
	\forall f,g \in L^2\p{\nu}, \inner{f|g}_{L^2\p{\nu}} = \sum_{y \in V} f(y) g(y) \nu_y
	\]
	and $L$ is given explicitely by the formula
	\[
	\forall f \in \RR^V, \p{Lf}(x) = \frac{1}{\mu_x} \sum_{y \sim x} \p{f(x) - f(y)}
	\]
	We will often use the following identity
	\[
	\forall f,g \in L^2\p{\nu}, \inner{f|L g}_{L^2\p{\nu}} = \frac{1}{2 \mu\p{V}} \sum_{x \sim y} \p{f(x)-f(y)}\p{g(x)-g(y)}
	\]
	The negative graph Laplacian $-L$ satisfies 
	\begin{enumerate}
		\item $\forall x \in V, -L_{xx} \geq  0$
		\item $\forall x\neq y \in V, -L_{xy} = \frac{\mu_{xy}}{\mu_x} \geq 0$
		\item $\forall x \in V, \sum_{y \in V} -L_{xy} = \sum_{y \in V} \frac{\mu_{xy}}{\mu_x} - 1 = 0$
	\end{enumerate}
	As a consequence (see e.g Norris \cite{norris_markov_1997}), for each $t \geq 0, e^{-tL}$ is a row stochastic matrix. It is the transition matrix in time $t$ of the continuous time Markov chain $\p{\tilde{W}_t}_{t \geq 0}$ going from vertex $x$ to vertex $y$ with rate $\frac{\mu_{xy}}{\mu_x}$. This chain has a constant unit speed and a jump matrix given by $D^{-1}A$ : equivalently $\tilde{W}_t = Y_{N_t}$ where $N$ is a Poisson process with unit rate and $\p{Y_n}_{n \geq 0}$ is an independent Markov chain with transition matrix $D^{-1}A$. We will denote by $\Pro_x^{(h)}$ the probability distribution of the corresponding Markov chains (in either discrete or continuous times) starting from $x \in V$. For convenience in what follows we will define $\LL = h^{-2} L$ and $W_t = \tilde{W}_{t/h^2}$.
	
	It is shown for instance in \cite{garcia_trillos_error_2020,calder_improved_2020} that under appropriate conditions the graph Laplacian $\LL$ converges in some sense to the true Laplace-Beltrami operator $\Delta$ of the limiting manifold $\MM$ (up to a proportionality constant). Since $W$ is a continuous time Markov chain with transition matrix $e^{-t\LL}$, this shows that $W$ can actually be seen in a way as a numerical approximation of a Brownian motion on $\MM$, i.e an $\MM-$valued continuous time Markov process with infinitesimal generator $-\Delta$.
	
	For each $x \in V, t \geq 0$ we define $p_t \p{x,\cdot} = p_t^{(h)} \p{x,\cdot}$ as the density of $e^{-t\LL}\p{x,\cdot}$ with respect to $\nu$
	\[
	\forall y \in V, p_t \p{x,y} = \frac{e^{-t \LL}\p{x,y}}{\nu_y}
	\]
	and we call the matrix $p_t$ the \emph{heat kernel of $\LL$ with respect to $\nu$}. If $\p{\lambda_j,u_j}_{j=1}^N = \p{\lambda_j^{(h)},u_j^{(h)}}_{j=1}^N$ is an orthonormal eigendecomposition of $\LL$ in $L^2 \p{\nu}$, i.e.
	\[
	\forall j,l \in \b{1,\ldots,N}, \LL u_j = \lambda_j u_j, \inner{u_j|u_l}_{\nu} = \delta_j^l
	\]
	(the existence of which is guaranteed by the finite dimensional spectral theorem) then we have
	\[
	\forall t \geq 0, e^{-t \LL} = \sum_{j=1}^N e^{-t\lambda_j} \inner{u_j | \cdot}_\nu u_j = \frac{1}{\mu \p{V}} \sum_{j=1}^N e^{-t\lambda_j} u_j u_j^T 
	\]
	Hence
	\[
	\forall t \geq 0, x,y \in V, p_t \p{x,y} = \sum_{j=1}^N e^{-t\lambda_j} u_j(x) u_j(y)
	\]
	In particular $p_t$ is a symmetric matrix.\\
	
	In addition to the $L^2\p{\nu}$ structure we will also use the uniform norm
	\[
	\forall f \in \RR^V, \|f\|_{L^\infty \p{\nu}} = \max_{1 \leq i \leq N} \abs{f(x_i)}
	\]
	
	As well as the associated operator norm
	\[
	\| A \|_{L^\infty \p{\nu}} = \sup_{\|f\|_{L^\infty\p{\nu}} \leq 1} \|Af\|_{L^\infty \p{\nu}}
	\]
	defined for any linear endomorphism $A : L^\infty \p{\nu} \to L^\infty \p{\nu}$.

	
	\textbf{Notations :} We denote by $\Pro_0$ the frequentist probability distribution of $\mathbb{X}^n := \p{x_{1:N}, y_{1:n}}$, i.e 
	\[
	\Pro_0(dx_{1:N} , dy_{1:n}) = \prod_{i=1}^N p_0(x_i) \mu(dx_i) \prod_{i=1}^n \NN\p{y_i | f_0(x_i),\sigma^2} dy_i
	\]
	as well as $P_f = \bigotimes_{i=1}^n \NN \p{f_i,\sigma^2}$ for any $f \in \RR^n$, with the abuse of notation $P_{f_0} = P_{\p{f_0(x_i)}_{i=1}^n}$. We will also denote the probability distribution of the infinite sequences $\p{x_i}_{i \geq 1}, \p{y_i}_{i \geq 1}$ by $P_0^\infty$. We denote the geodesic metric on $\MM$ by $\rho : \MM \times \MM \to \RR_+$. For any $x \in \RR^D, r > 0$ we define $B_{\RR^D}(x,r) = \b{y \in \RR^D : \norm{x-y} < r}$ the Euclidean ball and $B_\rho(x,r) = \b{y \in \MM : \rho(x,y)<r}$ the geodesic ball. For any measurable subset $A$ of $\RR^d, d \geq 1$, $vol(A)$ will denote its $d-$dimensional Lebesgue measure.
	
	For any open subset $\UU$ of $\RR^d, d,k \geq 1$ and $f : \UU \to \RR$ we define its $\CC^k\p{\UU}$ norm by $\norm{f}_{\CC^k \p{\UU}} = \max_{\alpha \in \mathbb{N}^d, \abs{\alpha} \leq k} \sup_{x \in \UU} \abs{\frac{\partial^{\abs{\alpha}}f}{\partial x^\alpha}(x)}$. We also define $\CC^\infty \p{\UU} = \cap_{k \geq 0} \CC^k \p{\UU}$. For non integer $\beta \in (k,k+1), k \in \mathbb{N}$ we define the H\"older class $\CC^\beta \p{\UU}$ class as the space of all functions $f \in \CC^k\p{\UU}$ satisfying
	\[
	\norm{f}_{\CC^\beta \p{\UU}} := \norm{f}_{\CC^k\p{\UU}} + \max_{\alpha \in \mathbb{N}^d, \abs{\alpha} = k} \sup_{x \neq y \in \UU} \frac{\abs{\frac{\partial^{\abs{\alpha}} f}{\partial x^\alpha} - \frac{\partial^{\abs{\alpha}} f}{\partial x^\alpha}}}{\norm{x-y}^{\beta-k}} < +\infty
	\]
	If $V$ is a finite dimensional vector space and $\UU \subset \RR^d, d \geq 1$, we say that a mapping $f : \UU \to V$ is of class $\CC^\beta\p{\UU}$ if $T \circ f : \UU \to \RR \in \CC^\beta \p{\UU}$ for any $T \in V^*$ (linear form on $V$). Equivalently, each coordinate function of $f$ (in any choice of basis) is of class $\CC^\beta$.
	
	It is immediate to check that any $f \in \CC^\beta \p{\UU}, k < \beta \leq k+1$ has a Taylor development of the form $f(y) = \sum_{l=0}^k \frac{d^lf(x).(y-x)^l}{l!} + R_k(x,y), \abs{R_k(x,y)} \leq C_f \norm{x-y}^\beta$ with $C_f = \|f\|_{\CC^\beta \p{\UU}}/k!$. 
	
We write $\norm{f-f_0}_n^2 = \frac{1}{n} \sum_{i=1}^n \p{f(x_i)-f_0(x_i)}^2$ and 	$\norm{f-f_0}_N^2 = \frac{1}{N} \sum_{i=1}^N \p{f(x_i)-f_0(x_i)}^2$.
	
In the following the symbols $\gtrsim,\lesssim$ refer to inequalities up to constants that depend only on $\MM,p_0$ and $f_0$. We will also write $a \asymp b$ if both $a \lesssim b$ and $b \lesssim a$ hold. We will also often write $C(\p{\lambda}_{\lambda \in \Lambda})$ or $C_{\p{\lambda}_{\lambda \in \Lambda}}$ generically to denote a positive constant depending only on the parameters $\lambda \in \Lambda$ whose value can change from line to line.
	
\section{Main results}\label{section:main_results}
In the nonparametric regression model \eqref{model:nonparametric_regression}, the goal is to estimate the regression function $f_0$ evaluated on the covariates $x_i, i = 1,\ldots,N$. Hence from a Bayesian point of view the goal is to design a prior distribution on an element $f : V \to \RR$ leading to good frequentist guarantees a posteriori. More precisely, following \cite{ghosal_convergence_2000,ghosal_fundamentals_2017} our goal is to derive \textit{posterior contraction rates} for our suggested Bayesian methodologies : we wish to identify a sequence $\varepsilon_n$ of positive real numbers satisfying
\[
\Pi \c{\norm{f-f_0}_n > \varepsilon_n | \mathbb{X}^n} \xrightarrow[n \to \infty]{P_0^\infty} 0
\]
Deriving such rates typically requires to prove preliminary approximation results which is the object of section \ref{subsection:approximation_results}.
	
\subsection{Priors}\label{subsection:priors}
	
As done in \cite{sanz-alonso_unlabeled_2022,fichera_implicit_2024} we design our priors by random basis expansion in the graph Laplacian eigenbasis. Our construction depends on $2$ degrees of freedom : the number of basis functions allowed in the sum, and the graph connectivity parameter $h$. This leads to a first family of  priors :
	
\noindent
\textbf{Prior 1 } Fixed $J, h$ case:\\
		We consider the prior $\Pi \c{\cdot | J,h}$ defined by the probability law of the random vector
		\begin{equation} \label{prior1} 
		f = \sum_{j=1}^J Z_j u_j, \quad  Z_j \iid \Psi
		\end{equation}
		for some $h>0, J \in \b{1,\ldots,N}$ and $\Psi$ a probability distribution having a positive and continuous density $\psi$ on $\R$ satisfying $\int_{\abs{x}>z}\psi(x)dx \leq e^{-b_1 z^{b_2} }$ , for all $z\geq z_0$ for some $b_1,b_2,z_0 >0$. In particular, $\Psi$ can be Gaussian, leading to the construction of a Gaussian vector $f$ on $V$, as done in \cite{dunson_graph_2022,sanz-alonso_unlabeled_2022}.

	As we shall see in section \ref{subsubsection:rates_non_adaptive_priors}, this prior leads to explicit posterior contraction rates which happen to be minimax optimal (up to logarithmic factors) under appropriate choice of the hyperparameters $J = J_n,h = h_n$ that depends on the unknown regularity $\beta$ of $f_0$ together with the dimension $d$ of $\MM$. In this sense the proposed prior is \textit{non adaptive}.

 It is then of interest to choose the hyperparameters $J,h$ in a data driven way. From the Bayesian point of view this can be done  by putting an additional prior layer over $J$ and $h$. For practical reasons, since for each value of $h$ a new Laplacian $\mathcal L$ needs to be computed, together with its eigenpairs $(u_j, \lambda_j)$ we only consider discrete priors on $h$, i.e. we enfore $h \in \mathcal H$ where $\mathcal H$ is a discrete subset of $\mathbb R_+$.

\noindent 
\textbf{Prior 2} Adaptive prior:\\
		We consider the prior $\Pi$ defined by the probability law of the random vector
		\begin{equation}\label{prior2}
		\begin{split}
		J \sim \pi_J,& \quad 	h|J \sim \pi_h(\cdot|J),  \quad f | J,h \sim \Pi \c{\cdot | J,h}, \quad	
		\text{denote		}  \mathcal H_J = \text{supp}(\pi_h(\cdot|J) ) \subset \mathcal H, 
		\end{split}
		\end{equation}
		 where $\pi_h$ is a probability distribution supported on $\mathcal H_J$ (that may depends on $J$) and $\pi_J$ is a probability distribution on $\mathbb N^*$ satisfying
		 \[
		 \forall j \geq j_0, e^{-a_1 jL_j} \leq \pi_J(j) \leq e^{a_2 j L_j} \text{ where } L_j = 1 \text{ or } L_j = \ln j, \text{ and } \# \mathcal H_J \leq K_1 e^{K_2 J \ln n}
		 \]
		 for some constants $a_1,a_2,K_1,K_2 > 0$.

Note that the joint prior on $(J,h)$ can be of very different forms. For instance  if $J\sim \mathcal P(\lambda)$ then the condition on $\pi_J$ holds with $L_j = \log j$ while if $\pi_J$ is a Geometric distribution, then it holds with $L_j=1$. Also we can choose for instance  the support of $\pi_h(\cdot|J)$ to be $\mathcal H_J  = \b{\frac{J^{-1/d}}{\ln^\kappa N}} $ for some $\kappa > 2$ and $\mathcal H = \cup_J \mathcal H_J$.  We will see that this leads to adaptive nearly minimax estimation rates. Alternatively and for more flexibility we can choose  
$ \mathcal H = \b{h_l := 2^l h_* : l = 0,\ldots,L}$ with $h_* = \p{\frac{\ln n}{n}}^{1/d}, \kappa > 2$ and $L \in \mathbb{N}$ such that $2^L h_* \leq 1 < 2^{L+1} h_*$, together with
\[\pi_h (h_l) \geq  b_1e^{- b_2 h_l^{-d}} \quad \forall h \in \mathcal H\]
for some constants $b_1,b_2>0$. We will detail the precise assumptions on $\Pi_J, \pi_h\p{\cdot|J}$ in theorems \ref{theorem:rate_prior2_d_n} \& \ref{theorem:rate_prior2_d_N}.

	\begin{remark}
		Both  \textbf{prior 1} and \textbf{prior 2},  depend on  the covariates $x_{1: N}$ and possibly on the sample size $n$, but we keep the notation $\Pi$ (and not $\Pi_{n,x}$) to avoid cumbersome notations.
	\end{remark}
	
	Having defined the two types of   priors we now  study  their posterior contraction properties.
	
	\subsection{Posterior contraction rates}\label{subsection:posterior_contraction_rates}
	
	We present here results on posterior contraction for \textbf{prior 1} \ and \textbf{prior 2}. We first present the non adaptive posterior contraction rate as it highlights the role of $J$ and $h$ and then we present the posterior contraction rate derived from prior 2. 
	The proof is based on the general prior mass and testing theorem (see e.g \cite{ghosal:vdv:07,ghosal_fundamentals_2017,rousseau:16}). A key aspect of the proof is a new scheme to approximate $f_0$ by functions in the span of $u_1, \cdots, u_J$. This allows in particular to extend the results of \cite{green_minimax_2021} to non linear manifolds when the regularity of $f_0$ not restricted to $\{1,2,3\}$. The approximation result is interesting in its own right and it allows also to derive minimax convergence rates for the frequentist estimator of \cite{green_minimax_2021}; hence it is presented in Section \ref{subsection:approximation_results} and the convergence rate of the frequentist estimator is provided in Section \ref{subsubsection:rates_non_adaptive_priors}.  Moreover, the proof is very different from \cite{green_minimax_2021} and is valid even when $N > n$.
	
\subsubsection{Non adaptive rates: Prior 1 and PCR-LE estimator}\label{subsubsection:rates_non_adaptive_priors}

Throughout this section we assume that the chosen connectivity parameter $h_n$ is not too small :
\begin{assumption}\label{assumption:h}
		The connectivity parameter $h_n$ satisfies $h_n \to 0$ and $\frac{Nh_n^d}{\ln N} \to \infty$. 
\end{assumption}

Our first main result identifies a posterior contraction rate associated with prior \ref{prior1} if $h$ satisfies assumption \ref{assumption:h} and $J$ is not too small. In particular, we recover minimax optimal rates up to logarithmic factors for appropriate choices of parameters :
	
\begin{theo}\label{theorem:rate_prior1_d_n}
		Let $\p{J_n,h_n}$ be a sequence of truncation and connectivity parameters satisfying $J_n \in \b{1,\ldots,N}, J_n \geq \ln^\kappa N, \kappa > d$ and $h_n$ satisfying assumption \ref{assumption:h}. Consider  a prior of $f$ belonging to the  \textbf{Prior 1}  family. Then there exists $C > 0$ such that
		\[
		\Pi \c{\norm{f-f_0}_n >C \varepsilon_n | \mathbb{X}^n, J_n, h_n} \xrightarrow[n \to \infty]{P_0^\infty} 0
		\]
		where
		\begin{equation}\label{epsilonn}
		\varepsilon_n = 		\varepsilon_n(J_n, h_n) = 
\sqrt{\frac{J_n \ln N}{n}} + {\ln N}^{\lceil \beta/2 \rceil} \max\left( 1 , \frac{ \ln N}{ J_n^{2/d} h_n^2}\right)^{\lceil \beta/2 \rceil} \p{ h_n^\beta + \II_{\beta > 1 } \p{\frac{\ln N}{Nh_n^d}}^{1/2} h_n }
		\end{equation}
		In particular, for any $\tau > d/2$ and
		\[
		h_n = n^{-\frac{1}{2\beta+d}} \p{\ln n}^{-\frac{1-\tau-2\p{1+2\tau/d} \lceil \beta/2 \rceil}{2\beta+d}}, J_n = \frac{h_n^{-d}}{\ln^\tau N}
		\]
		then 
		\[
		\varepsilon_n \asymp \p{\ln n}^{\frac{\p{2\tau+d}\lceil \beta/2\rceil + (1-\tau)\beta}{2\beta+d}} n^{-\frac{\beta}{2\beta+d}}
		\]
\end{theo}
The proof is given in section \ref{appendix:proof_posterior_contraction_rates:non_adaptive_priors}. Theorem \ref{theorem:rate_prior1_d_n}  highlights  the trade - off between the complexity of the prior (term $\sqrt{ J \ln N/n}$) and the approximation error bounded by $\p{\frac{J_n^{-2/d}\ln N}{h_n^2}}^{\lceil \beta/2\rceil} \c{h_n^{\beta} + \II_{\beta>1} h_n \left( \frac{ \ln N }{ Nh_n^d }\right)^{1/2} }$. The first term of the approximation error has a bias flavour while the second, which appears only when $\beta>1$ comes, roughly speaking,  from stochastic deviations of $\mathcal L  f _0- \EE(\mathcal L f_0) $, see Lemma \ref{lemma:concentration}. More discussion is given after Theorem \ref{thm:approximation}.

Notice that theorem \ref{theorem:rate_prior1_d_n} only guarantees a posterior contraction with respect to the empirical $L^2$ norm $\norm{f-f_0}_n$. Thus, it is not clear whether or not the posterior distribution correctly extrapolates outside of the design $\b{x_1,\ldots,x_n}$ in order to estimate the values $\b{f_0(x_{n+1}),\ldots,f_0(x_N)}$. This is the object of theorem \ref{theorem:rate_prior1_d_N} below, which requires stronger assumptions on $J_n,h_n$ and $\beta$.
	
\begin{theo}\label{theorem:rate_prior1_d_N}
	Assume that $\beta > d/2$ and let $h_n$ satisfying $h_n \geq \p{\ln N}^s n^{-\frac{1}{2\beta+d}}$ for some $s \in \RR$ as well as $J=J_n \in \b{1,\ldots,N}, \frac{h_n^{-d}}{\ln^{\tau_1} N} \leq J_n \leq \frac{h_n^{-d}}{\ln^{\tau_2} N}, \tau_1 \geq \tau_2 > 2d$. Then there exists $C>0$ such that 
	\[
	\Pi \c{\norm{f-f_0}_N > C\varepsilon_n | \mathbb{X}^n, J_n, h_n} \xrightarrow[n \to \infty]{P_0^\infty} 0
	\]
	where
	\[
	\varepsilon_n = \sqrt{\frac{J_n \ln N}{n}} + \p{\frac{J_n^{-2/d} \ln N}{h_n^2}}^{\lceil \beta/2 \rceil} \p{ h_n^\beta + \II_{\beta > 1} \p{\frac{\ln N}{Nh_n^d}}^{1/2} h_n }
	\]
		
	In particular, choosing $J_n = \frac{h_n^{-d}}{\ln^\tau N}, \tau > 2d$ and $h_n = n^{-\frac{1}{2\beta+d}} \p{\ln N}^{-\frac{1-\tau-2\lceil \beta/2 \rceil \p{1+2\tau/d}}{2\beta+d}}$ leads to a posterior contraction rate of order 
	\[
	\varepsilon_n \asymp \p{\ln N}^{\frac{2\lceil \beta/2\rceil \p{2\tau+d} + 2\beta\p{1-\tau}}{2\beta+d}} n^{-\frac{\beta}{2\beta+d}}
	\]
\end{theo}

The proof is provided in Section \ref{appendix:d_N} and relies on Theorem \ref{theorem:rate_prior1_d_n} together with a concentration inequality to upper bound $\norm{f-f_0}_N$ by (a multiple of) $\norm{f-f_0}_n$.

The main limitation of \textbf{prior 1}, with deterministic $J,h$ is that the choice of $J,h$ crucially impacts the asymptotic behaviour of the posterior distribution. In particular to obtain optimal behaviour (with respect to $\norm{\cdot}_n$ or $\norm{\cdot}_N$) one needs to choose them as a function of $\beta$ and $d$. While there exist simple and consistent estimators of $d$, see for instance \cite{berenfeld_density_2021}, it is very difficult to estimate $\beta$.

It is interesting to compare our results with those of \cite{sanz-alonso_unlabeled_2022}. In this work a Mat\'ern type Gaussian process of the form $f = \sum_{j=1}^{J_N} \p{1 + \lambda_j}^{-\frac{\beta+d/2}{2}} Z_j u_j, Z_j \iid \NN\p{0,1}$ (also for some deterministic values of $J_n,h_n$) is used in place of prior \ref{prior1}. 
In \cite{sanz-alonso_unlabeled_2022} the author obtain the minimax rate of convergence with respect to the empirical $L^2$ norm $\norm{\cdot}_n$ by choosing $J,N,h$ accordingly but under a more restrictive assumption on $N,\beta$, implying in particular
$N >> n^{2d}$ and $\beta > d-1/2$. Also, their proof technique is very different from ours as it is based on the convergence of the (discrete) Mat\'ern Gaussian process prior  to its continuous counterpart (associated to the Laplace- Beltrami operator on $\MM$).

In the fully supervised $N=n$ case, our result can also be compared to the (frequentist) PCR-LE estimator of \cite{green_minimax_2021}. To be completely clear the estimator proposed in \cite{green_minimax_2021} relies on the unnormalized graph Laplacian rather than the normalized one, but we believe our proof could be applied for the unnormalized graph Laplacian as well. The PCR-LE estimator is defined as 
\[ \hat f = \text{argmin}_{ f \in \text{span}(u_1, \cdots, u_J)} \|Y - f \|_2, \quad Y = (y_1, \cdots, y_n), \quad f = (f(x_1), \cdots, f(x_n))\]
In  \cite{green_minimax_2021}, the authors prove that when $f_0$  belongs to a $\beta-$ Sobolev class  with $\beta \in \b{1,2,3}$ (where the Sobolev space is defined using the Laplacian operator weighted according to the covariates density), then $\hat f$ converges to $f_0$ at the rate $n^{-\beta/(2\beta+d)}$, if $J$ and $h$ are chosen accordingly. In comparison we obtain a near minimax rate for any $\beta-$ H\"older regularity assumption for $f_0$, $\beta > 0$ (for the $\norm{\cdot}_n$ loss) or $\beta > d/2$ (for the $\norm{\cdot}_N$ loss). It should be noted that the restriction $\beta > d/2$ also appears in the continuous case when extending outside of the labelled points, see e.g \cite{vaart_information_2011,rosa_posterior_2023}

As mentioned earlier, a key component of the proof of Theorems \ref{theorem:rate_prior1_d_n} \& \ref{theorem:rate_prior1_d_N} is to show that $f_0$  can be well approximated by elements in the span of $(u_1, \cdots, u_J)$, see Theorem \ref{thm:approximation} . A consequence of Theorem \ref{thm:approximation} is thus that the PCR-LE estimator \cite{green_minimax_2021} achieves near minimax convergence rate over H\"older classes of any order $\beta$, thus extending the results of \cite{green_minimax_2021}.

	
		\begin{theo}{Performance of PCR-LE for the random walk graph Laplacian}\label{corollary:pcr_le}\\
		Assume that $N = n$ and that the sequence $\p{J_n,h_n}$ is such that $J_n \in \b{1,\ldots,n}, J_n \geq \ln^\kappa N, \kappa > d$ and that $h_n > 0$ satisfies assumption \ref{assumption:h}. Define
		\[
		\hat{f} = \sum_{j=1}^{J_n} \inner{u_j | Y}_{L^2\p{\nu}} u_j
		\]
		Then for any $q \geq 0$ there exists $C>0$ such that
		\[
		\EE_0 \c{ \norm{\hat{f}-f_0}_n^q } \leq c \varepsilon_n^q
		\]
		where
		\[
		\varepsilon_n = \sqrt{\frac{J_n \ln N}{n}} + \p{\frac{J_n^{-2/d} \ln N}{h_n^2}}^{\lceil \beta/2 \rceil} \p{ h_n^\beta + \II_{\beta > 1} \p{\frac{\ln N}{Nh_n^d}}^{1/2} h_n }
		\]
		In particular for the choice $J_n = \frac{h_n^{-d}}{\ln^{\tau} n}, \tau > d/2, h_n = n^{-\frac{1}{2\beta+d}} \p{\ln n}^{-\frac{\tau + 2 \lceil \beta/2 \rceil \p{1+2\tau/d}}{2\beta+d}}$ for any $q>0$ there exists $C>0$ such that
		\[
		\EE_0 \c{ \norm{\hat{f}-f_0}_n^q } \leq C \varepsilon_n^q
		\]
	\end{theo}

	The proof of Theorem \ref{corollary:pcr_le} is given in Section \ref{pr:PCR_LE}. 
	Corollary \ref{corollary:pcr_le} hence shows that in the case $N=n$ the PCR-LE estimator \cite{green_minimax_2021} can achieve optimal rates for particular choice of hyperparameters (up to logarithmic factors) for any smoothness index $\beta > 0$ if one is willing to work on the H\"older spaces scale, and not just for $\beta-$Sobolev functions, $\beta \in \b{1,2,3}$ : the limitations of the manifold adaptivity was therefore only an artefact of the proof, as the non-negligible discrepancy between geodesic and Euclidean distances can be circumvented by considering a different approximation technique. Extending \ref{corollary:pcr_le} to an adaptive estimator by Lepski's method as done in \cite{shi_adaptive_2023} in the Euclidean domain could also be of interest, which then would provide a frequentist alternative to our adaptive Bayesian prior \ref{prior2}.
	
	In the case $N>n$, our theorems \ref{thm:approximation} \& \ref{thm:approximation2} can also potentially be used to design frequentist estimators in alternative to our priors \ref{prior1} \& \ref{prior2}. Indeed, as an example Theorem \ref{thm:approximation} can be reformulated as a high probability upper bound on $\norm{f_0 - Q_t^{(k)} f_0}_{L^\infty\p{\nu}}$	where $Q_t^{(k)}$ is the kernel operator
	\[
	\forall f : V \to \RR, \p{Q_t^{(k)} f}(x) = \sum_{y \in V} \underbrace{\sum_{l=0}^k \frac{\p{t\LL}^l e^{-t\LL} (x,y)}{l!}}_{=: Q_t^{(k)}(x,y) \nu_y} f(y)
	\]
	This can be generalized by
	\[
	\chi_t^{(k)}(x,y) = \sum_{l=0}^k \frac{\p{-t\LL}^l \chi^{(l)} \p{t\LL} (x,y)}{l!} = \sum_{l=0}^k \frac{1}{l!} \sum_{j=1}^N \p{-t\lambda_j}^l \chi^{(l)}\p{t\lambda_j} u_j(x) u_j(y)
	\]
	for an arbitrary $\chi \in \CC^k\p{\RR_+,\RR}$ with $\chi(0) = 1$ which motivates the definition of kernel regression estimators of the form $\hat{f} = \sum_{i=1}^n \chi_t^{(k)}(\cdot, x_i) y_i \nu_{x_i}$ : this is always tractable even when $N>n$ and analyzing the asymptotic properties of such estimators (even on other graph models) is therefore an interesting research direction.
	
	\subsubsection{Results for the  adaptive priors}\label{subsubsection:rates_adaptive_priors}
	
	Having proved Theorems \ref{theorem:rate_prior1_d_n} \& \ref{theorem:rate_prior1_d_N} it is straightforward to derive adaptive posterior contraction rates under the \textbf{prior 2}. This results in a Bayesian method able to achieve minimax optimal contraction rates (up to logarithmic factors) with a data driven choice for $J,h$.
		
	\begin{theo}\label{theorem:rate_prior2_d_n}
		Consider the prior $\Pi$ on $f$  belonging to the class \textbf{Prior 2} as defined  by \eqref{prior2} and assume that for some $K_1,K_2 > 0$ and  some $j_0,a_1, a_2 >0$
	\begin{equation} \label{cond:prior:J}
		\forall j \geq j_0, e^{-a_1 j L_j} \leq   \pi_J(j) \leq e^{-a_2 j L_j}, \quad \text{where} \quad  L_j =1 \quad \text{or} \quad L_j = \ln j,  \quad \text{and } \, \# \HH_J \leq K_1 e^{K_2 J \ln n}
	\end{equation}
	Assume in addition that for some $\tau > d/2$, $J_0,h_0>0$  and
	\[
	J_n = J_0\p{\ln N}^{\frac{2d\p{1+2\tau/d}\lceil \beta/2 \rceil}{2\beta+d}} n^{\frac{d}{2\beta+d}}
	\]
	$h_n \in \HH_{J_n} \cap \left[\frac{h_0J_n^{-1/d}}{2\ln^{\tau/d}N}, \frac{h_0J_n^{-1/d}}{\ln^{\tau/d}N}\right]$ and $c>0$ such that 
	\begin{equation}\label{cond:priorh}
		\pi_h(h_n|J_n) \geq e^{-c h_n^{-d}}, 
		\end{equation}
	Then there exists $C>0$ such that
		\[
		\Pi \c{\norm{f-f_0}_n > C\varepsilon_n | \mathbb{X}^n} \xrightarrow[n \to \infty]{P_0^\infty} 0, \quad \text{ with} \quad 
		\varepsilon_n = n^{-\frac{\beta}{2\beta+d}} \p{\ln N}^{\frac{\p{2\tau+d}\lceil \beta/2\rceil - \beta \p{\tau + 2\beta/d}}{2\beta+d}}.
		\]
		
	\end{theo}
	The proof of Theorem \ref{theorem:rate_prior2_d_n} is given in Section \ref{appendix:proof_posterior_contraction_rates:adaptive_priors}. Notice that even if some particular values of $J_n,h_n$ are part of a condition that \textbf{prior 2}  \eqref{prior2},  must satisfy theoretically, the precise values of $J_n,h_n$ are not part of the definition of the prior itself.
	
	As discussed above in section \ref{prior2}, there are many priors on $J,h$ which satisfy \eqref{cond:priorh}. Consider two cases of particular interest
	\begin{itemize}
		\item If $J$ follows a Poisson or a Geometric prior and, given $J$, $h = \frac{J^{-1/d}}{\ln^{\tau/d}N}, \tau > d/2$ then condition \eqref{cond:priorh} is fulfilled (with $\HH_J = \b{\frac{J^{-1/d}}{\ln^{\tau/d}n}}$).
		\item A more flexible alternative is to choose a discrete prior on $\mathcal H = \{ 2^l h_* : l = 0, \cdots, L\}$ for $h_* = \p{\frac{\ln n}{n}}^{1/D}$ and $L$ such that $2^L h_* \leq 1 < 2^{L+1} h_*$, together with a probability mass function of the order of $ e^{- \lambda h^{-1} } h^a$ for  arbitrary $\lambda > 0$ and $a\geq 0$. For instance, such a prior can be constructed by first considering $\tilde{h} \sim IG( a, \lambda)$, an inverse Gamma random variable  and then defining $h = \sum_{l=0}^L 2^l h_* \II_{I_l}(\tilde{h}), I_0 = [0,h_*], I_l = ]2^l h_*, 2^{l+1} h_*]$ for $1 \leq l < L$ and $I_L = ]2^L h_*, +\infty[$. In particular this prior is fully adaptive as it depends neither on $\beta$ nor on $d$.
	\end{itemize}
	
	Similarly to the non adaptive prior, it is possible to obtain a posterior contraction rate in terms of the loss $\norm{\cdot}_N$ from the result of Theorem \ref{theorem:rate_prior2_d_n}, under additional assumptions on $\beta$ and the prior on $J,h$.
	
	\begin{theo}\label{theorem:rate_prior2_d_N}
		Under the same conditions as in theorem \ref{theorem:rate_prior2_d_n} and if in addition we have $\HH_J \subset [h_*,\frac{J^{-1/d}}{\ln^{\tau/d} N}]$ $\Pi-$almost surely for some $\tau > 2d$, $h_*>0$ satisfying assumption \ref{assumption:h} and if $\beta > d/2$ then there exists $M>0$ such that
		\[
		\Pi \c{\norm{f-f_0}_N > M\varepsilon_n | \mathbb{X}^n} \xrightarrow[n \to \infty]{P_0^\infty} 0
		\]
		for the rate
		\[
		\varepsilon_n = n^{-\frac{\beta}{2\beta+d}} \p{\ln N}^{\frac{\p{2\tau+d}\lceil \beta/2\rceil - \beta \p{\tau + 2\beta/d}}{2\beta+d}}
		\]
	\end{theo}
	
	\begin{remark}
		Theorem \ref{theorem:rate_prior2_d_N} thus shows that we can construct (near) minimax adaptive Bayesian procedures based on graph Laplacian decompositions, where the adaptivity is related to the (typically unknown) smoothness $\beta$. Similarly to theorem \ref{theorem:rate_prior2_d_n}, many priors satisfy the assumptions of Theorem \ref{theorem:rate_prior2_d_N}, for instance a Poisson or geometric prior on $J$ can be chosen together with $h = \frac{J^{-1/d}}{\ln^{\tau/d}n}, \tau > d/2$. Alternatively, as discussed below Theorem \ref{theorem:rate_prior2_d_n}, we can choose a discrete prior on $\mathcal H_J := \{ 2^l h_* : l = 0, \cdots, L\}$ with $h_* = m_n\p{\frac{\ln n}{n}}^{1/D}$ and $L$ such that $2^L h_* \leq \frac{J^{-1/d}}{\ln^{\tau/d}N} < 2^{L+1} h_*$, together with a probability mass function of the form of $ e^{- \lambda h^{-d} } h^a$ for arbitrary $m_n \to \infty, \lambda > 0$ and $a\geq 0$.
		
		Compared to Theorem \ref{theorem:rate_prior2_d_n}, we require $\HH_J \subset \c{h_*,\frac{J^{-1/d}}{\ln^{\tau/d}N}}$ for some $h_*$ satisfying assumption \ref{assumption:h} in order to apply Lemma \ref{thm:norm_comparison} a posteriori (and importantly with constants that do not depend on $h,J$), see the detailed proof in Section \ref{appendix:d_N}. In fact the condition could be weakened to $\HH_J \subset \c{h_*,\frac{J^{-1/d}}{\ln^{\tau/d}N}}$ on a set of posterior probability tending to $1$, but this is not necessary for the examples of priors we provide.
		\end{remark}
		
		\begin{remark}
		The same condition $\HH_J \subset \left[h_*, \frac{J^{-1/d}}{\ln^{\tau/d} N}\right]$ prevents us to design priors adaptive to $d$, contrariwise to Theorem \ref{theorem:rate_prior2_d_n}. This condition is crucial in our setting in order to use Lemma \ref{thm:norm_comparison}. Whether or not such a restriction is necessary or is an artefact of our proof technique is therefore an interesting question. To design a procedure that is adaptive both to $\beta$ and $d$ one could plug in a consistent estimator of $d$, see e.g \cite{berenfeld_density_2021} (the procedure then becomes an empirical Bayes procedure), and our theoretical results remain unchanged since the estimator only depends on $x$ and not on $y$.
	\end{remark}

	\begin{remark}
	Theorems \ref{theorem:rate_prior2_d_n} \& \ref{theorem:rate_prior2_d_N} hold for discrete priors  on $h$, however we believe that this is not a restriction since a continuous prior would imply that for every value of $h$, one would need to compute the eigenpairs $(u_j, \lambda_j)$ of the Laplacian corresponding to $h$. Obtaining a result which would hold for a continuous prior on $h$, or for empirical Bayes procedure, where optimization with respect to $h$ is performed, is not straightforward since it would necessitate to prove the concentration lemma \ref{lemma:concentration} uniformly over $h \in ( h_0	, h_1) \varepsilon_n^{1/\beta}$.
	\end{remark}

	

	Notice that our posterior contraction rates in Theorems \ref{theorem:rate_prior1_d_n} \ref{theorem:rate_prior1_d_N} \ref{theorem:rate_prior2_d_n} \ref{corollary:pcr_le} \& \ref{theorem:rate_prior2_d_N} are all impacted by a power of $\ln N$ : this can be large, for instance if $N$ grows exponentially with $n$. However under the mild assumption $N \leq n^B, B>0$ this factor become a power of $\ln n$ which is typically considered small from an asymptotic viewpoint, ignoring the multiplicative constant. Extending our results to the case $N >> n^B$ for all $B>0$ would require more work, but we believe that in this case the proof strategy of \cite{sanz-alonso_unlabeled_2022} could be adapted.
	
	\subsection{Approximation result}\label{subsection:approximation_results}
	
All three  theorems \ref{theorem:rate_prior1_d_n}, \ref{theorem:rate_prior1_d_N} \ref{corollary:pcr_le}, \ref{theorem:rate_prior2_d_N} \& \ref{theorem:rate_prior2_d_n} are proved using approximation results of $f_0$ in the graph spectral domain, but since the smoothness assumption $f_0 \in \CC^\beta \p{\MM}$ is only formulated at the continuous level this is a non trivial task. To prove that $f_0$ can be well approximated by an element of $\text{span}(u_1, \cdots, u_J)$ we first find an approximation of $f_0$ (as a function on $V$) by an element of the form $e^{-t\LL} f_t$ for some explicit $f_t$, with high probability with respect to the sampling distribution of the covariates $x_{1:N}$. Then we show that the projection of $e^{-t\LL} f_t$ onto $\text{span}(u_1, \cdots, u_J)$ is close to $e^{-t\LL}f_t$ and thus to $f_0$.
	
	\begin{theo}\label{thm:approximation}
		Let $\beta>0, f_0 \in \CC^\beta\p{\MM}$, assume that $h_n$ satisfies assumption \ref{assumption:h} and define for  $t > 0$
		\[
		f_t = \sum_{l=0}^k \frac{1}{l!} \p{t \LL}^l f_0, \quad f_t : V \rightarrow \mathbb R.
		\]
		where $k = \lceil \beta/2 \rceil - 1$. Then 
		\begin{enumerate}
			\item If $0 < \beta \leq 1$, there  exists a constant $C(\beta, f_0)$ we have
			\[
			\norm{f_0 - e^{-t\LL}f_t}_{L^\infty \p{\nu}} \leq C\p{\beta, f_0} th_n^{\beta-2}
			\]
			\item if $\beta > 1, p_0 \in \CC^{\beta-1}\p{\MM}$ then for any $H>0$ there exists a constant $C(H,\beta, f_0, p_0, \MM)$, such that
			\[
			\Pro_0 \p{ \norm{f_0 - e^{-t\LL}f_0}_{L^\infty \p{\nu}} > C\p{H,\beta,f_0,p_0,\MM} \p{\frac{t}{h_n^2}}^{k+1}\c{ h_n^\beta + \p{\frac{\ln N}{Nh_n^d}}^{1/2} h_n}} \leq N^{-H}
			\]
		\end{enumerate}
	\end{theo}
	
The proof of Theorem \ref{thm:approximation} is given in Section \ref{sec:pr:thm:approx}.

Interestingly, when $\beta>1$, 
\[h_n^\beta\geq h_n \sqrt{ \frac{ \ln N }{ Nh_n^d }}, \quad \Leftrightarrow h_n \geq \p{\frac{\ln N}{N}}^{\frac{1}{2(\beta-1)+d}}\]
which is  - up to the $\ln N$ term - the condition required  \cite{green_minimax_2021} theorem 6 in the case $N=n$. Although the proof techniques are very different, in both cases this condition is used to  bound  deviations of $\mathcal Lf $ around its mean.
	
To prove that $f_0$ can be well approximated by a function of the form  $f = \sum_{j=1}^{J_n} z_j u_j$ for some choice of truncation $J_n$, we now show that   $e^{-t\LL} f_t$ is close to $p_{J_n}(e^{-t\LL} f_t)$, its $L^2\p{\nu}$ projection onto $\text{span}(u_1, \cdots, u_{J_n})$. Since we are able to lower bound the eigenvalues of $\LL$ with high probability, thanks to the exponential factor $e^{-t\LL}$ this is actually straightforward.
	
	\begin{theo}\label{thm:approximation2}
		Let $\beta>0, f_0 \in \CC^\beta\p{\MM}, h_n$ satisfying assumption \ref{assumption:h} and $J_n \in \b{1,\ldots,N}$ satisfying $J_n \geq \ln^\kappa N, \kappa > d$. Let $f_t : V \to \RR$ be defined by
		\[
		f_t = \sum_{l=0}^k \frac{1}{l!} \p{t \LL}^l f_0
		\]
		where $k = \lceil \beta/2 \rceil - 1$. Then there exists $c > 0$ such that with $t = c\lambda_{J_n}^{-1} \ln N$ and $\tilde \varepsilon_n(J_n, h_n)$ the rate
		\[
		\tilde \varepsilon_n (J_n, h_n) =
			\p{\ln N}^{ \lceil \beta/2 \rceil} \max\left( 1, \frac{J_n^{-2/d} \ln N}{h_n^2}\right)^{k+1} \p{ h_n^\beta + \II_{\beta > 1} \p{\frac{\ln N}{Nh_n^d}}^{1/2} h_n } 
		\]
		we have
		\begin{enumerate}
			\item If $0 < \beta \leq 1$ : there exists a constant $C(\beta, f_0,\rho)$ such that
			\[
			\norm{f_0 - p_{J_n} \p{e^{-t\LL}f_t}}_{L^\infty \p{\nu}} \leq C(\beta, f_0) \tilde \varepsilon_n (J_n,h_n) 
			\]
			\item If $\beta > 1, p_0 \in \CC^{\beta-1}\p{\MM}, H>0$ : for some constant $C(H,\beta,p_0,f_0,\MM,\rho)$
			\[
			\Pro_0 \p{\norm{f_0 - p_{J_n} \p{e^{-t \LL}f_t}}_{L^\infty\p{\nu}} \geq C(H,\beta,p_0,f_0,\MM) \tilde \varepsilon_n (J_n,h_n) } \leq N^{-H}
			\]
		\end{enumerate}
	\end{theo}

	The proof of Theorem \ref{thm:approximation2} is given in  Section \ref{sec:pr:thm:approx2}.

	Both theorem \ref{thm:approximation} \& \ref{thm:approximation2} are useful results in their own rights, and can be used to design and study the convergence properties of frequentist estimators, as we have done in Theorem \ref{corollary:pcr_le}.

\section{ Main proofs}\label{sec:main_proofs}
In this Section we present the proofs of Theorems \ref{thm:approximation} and \ref{thm:approximation2}. 
To control the approximation of $f_0$ by $e^{-t \LL}f_t$ we introduce a deterministic  approximation of  the operator $\mathcal L f$. More precisely, recall that 
\[\mathcal L f = \frac{ 1 }{ \mu_x h^2 } \sum_{y \sim x} f(x) -  f(y)\]
By analogy this leads us to the following definition :
\begin{definition}
		For every $f \in \CC\p{\MM}$ and $h>0$ we define
		\[
		T_h f(x) = \frac{1}{h^2 P_0(B_{\RR^D}(x,h))} \int_{B_{\RR^D}(x,h) \cap \MM} \p{f(x)-f(y)} p_0(y)\mu(dy)
		\]
		where $P_0 \p{B_{\RR^D}(x,h)} = \int_{B_{\RR^D}(x,h) \cap \MM} p_0(y)\mu(dy)$.
	\end{definition}
	$T_h$ acts as a second order nonlocal differential operator and is a deterministic approximation of the operator $\mathcal L$ as shown in Lemma \ref{lemma:concentration} below. 

\subsection{Proof of theorem \ref{thm:approximation} } \label{sec:pr:thm:approx}
	
By a Taylor expansion of $t \mapsto e^{-\lambda t}$ (with $\lambda \geq 0$) we have
		\[
		1 = e^{-\lambda t}\sum_{l=0}^k \frac{\p{\lambda t}^l}{l!} + R_k(t)
		\]
		where the remainder is given by
		\[
		R_k(t) = \frac{1}{k!} \int_0^1 \p{\lambda x t}^{k+1}e^{-\lambda x t} \frac{dx}{x}
		\]
		Hence by (finite dimensional) functional calculus we find
		\[
		f_0 = e^{-t \LL}f_t + \frac{1}{k!} \int_0^1 \p{x t \LL}^{k+1}e^{-x t \LL}f_0 \frac{dx}{x}
		\]
		Therefore, for all $x \in V$, using $\norm{e^{-t\LL}}_{L^\infty\p{\nu}} = 1$, we bound
		\begin{align*}
			\norm{f_0 - e^{-t\LL}f_t}_{L^\infty\p{\nu}} \leq & \frac{1}{k!} \int_0^1 \norm{\p{xt\LL}^{k+1}e^{-xt\LL}f_0}_{L^\infty\p{\nu}} \frac{dx}{x}  = \frac{1}{k!} \int_0^t \norm{\p{s\LL}^{k+1}e^{-s\LL}f_0}_{L^\infty\p{\nu}} \frac{ds}{s} \\
			\leq & \frac{1}{k!} \int_0^t \norm{\p{s\LL}^{k+1}f_0}_{L^\infty\p{\nu}} \frac{ds}{s} 
			=  \frac{\norm{\p{t\LL}^{k+1}f_0}_{L^\infty\p{\nu}}}{(k+1)!}.
		\end{align*}
		We  now bound $\norm{\p{t\LL}^{k+1}f_0}_{L^\infty\p{\nu}}$, considering the three cases: $\beta\leq1$, $1<\beta\leq 2$ and $\beta>2$. 
		\begin{itemize}
			\item If $\beta \leq 1$ then the result is trivial as
			\[
			\abs{\p{t\LL}f_0(x)} = t\abs{\frac{1}{h_n^2 \mu_x} \sum_{y \sim x} f_0(x) - f_0(y)} \lesssim \norm{f_0}_{\CC^\beta \p{\MM}} th_n^{\beta-2}
			\]
			\item If $1 < \beta \leq 2$, 
we have 
	\begin{align*}
				\norm{t\LL f_0}_{L^\infty \p{\nu}} \leq & t \norm{\LL f_0 - T_{h_n}f_0}_{L^\infty \p{\nu}} + t \norm{T_{h_n} f_0}_{L^\infty \p{\nu}}.
			\end{align*}			
			 Then lemma \ref{lemma:regularity} implies
			\[
			\norm{T_{h_n} f_0}_{L^\infty\p{\MM}} \leq C(\beta, p_0)\norm{f_0}_{\CC^\beta\p{\MM}} h_n^{\beta-2}
			\]
		and	Lemma \ref{lemma:concentration} implies that for any $H>0$ there exists $M_0>0$ such that 
			\[
			\Pro_0 \p{ \norm{\LL f_0 - T_{h_n}f_0}_{L^\infty \p{\nu}} > M_0 \p{\frac{\ln N}{N}}^{1/2} h_n^{-(1+d/2)}} \leq N^{-H}
			\]
			Hence, for some $C(H, \beta, f_0, p_0)>0$
			\[
			\Pro_0 \p{\norm{h_n^2\LL f_0}_{L^\infty \p{\nu}} >  C(H, \beta, p_0, f_0) \p{ \p{\frac{\ln N}{N}}^{1/2} h_n^{1-d/2} + h_n^{\beta}}} \leq N^{-H}
			\]
			
			\item When $\beta > 2$ we use an inductive argument. 
			We have shown that for $k=\lceil \beta/2 \rceil -1 = 0$, if $f \in \mathcal C^\beta$,  
			\[ \norm{h_n^2\LL f}_{L^\infty \p{\nu}} \leq C_0(H,\beta,f_0,p_0,\MM) [  h_n^{\beta} + \Delta_n(f)]\]
			where $\Pro_0 \p{|\Delta_n(f)| >  M_0 \p{\frac{\ln N}{Nh_n^d}}^{1/2} h_n} \leq N^{-H}$.
	Assume that for all $k'< k$, if $f \in \mathcal C^{\beta'}$ with $2k'<\beta'	\leq 2(k'+1)$, then	
				\[ \norm{(h_n^2\LL f)^{k'+1}}_{L^\infty \p{\nu} }\leq C_{k'}(H,\beta, p_0, f_0, \MM) [  h_n^{\beta'} + \Delta_{n,k'}(f)]\]
			where $\Pro_0 \p{|\Delta_{n,k'}(f)| >  M_0 \p{\frac{\ln N}{Nh_n^d}}^{1/2} h_n } \leq N^{-H}$.
			We prove that the same holds for $k$. Let $f \in \mathcal C^\beta$ with $2k< \beta \leq 2(k+1)$.  
			Since $\beta > 2$ by lemma \ref{lemma:regularity} we find $f_1,\ldots,f_{k}$ (that depend on $h_n$) such that for all $l \in \b{1,\ldots,k}$
			\[
			\norm{f_l}_{\CC^{\beta-2l}\p{\MM}} \lesssim \norm{f}_{\CC^\beta\p{\MM}}, \quad \norm{T_{h_n} f - \sum_{l=1}^k h_n^{2(l-1)}f_1^{(l)}}_{L^\infty\p{\MM}} \leq C_l(p_0,\MM) \norm{f}_{\CC^\beta\p{\MM}} h_n^{\beta-2}
			\]
			Using $\norm{h_n^2\LL}_{L^\infty\p{\nu}} \leq 2$ which holds since 
			\[
			\forall f \in \RR^V, x \in V, \abs{\p{h_n^2\LL f}(x)} = \abs{\frac{1}{\mu_x} \sum_{y \sim x} f(x) - f(y)} \leq \frac{2\norm{f}_{L^\infty \p{\nu}}}{\mu_x} \# \b{y : y \sim x} = 2 \norm{f}_{L^\infty \p{\nu}}
			\]
			we get
			\begin{align*}
				 \norm{\p{h_n^2\LL}^{k+1}f_0}_{L^\infty \p{\nu}} 
				&=  h_n^2 \norm{\p{h_n^2\LL}^k \p{T_{h_n} f_0 + \LL f_0 - T_{h_n}f_0}}_{L^\infty \p{\nu}} \\
				&\leq  h_n^2 \norm{\p{h_n^2 \LL}^k T_{h_n}f_0}_{L^\infty \p{\nu}} + h_n^2 \norm{\p{h_n^2\LL}^k \p{\LL f_0 - T_{h_n}f_0}}_{L^\infty \p{\nu}} \\
				&\lesssim  h_n^2 \norm{\p{h_n^2\LL}^k T_{h_n}f_0}_{L^\infty \p{\nu}} + h_n^2 \norm{\LL f_0 - T_{h_n}f_0}_{L^\infty \p{\nu}} \\
				&=  h_n^2 \norm{\p{h_n^2\LL}^k \p{\sum_{l=1}^k h_n^{2(l-1)} f_l + T_{h_n}f_0 - \sum_{l=1}^k h_n^{2(l-1)} f_l}}_{L^\infty \p{\nu}} + h_n^2 \norm{\LL f_0 - T_{h_n}f_0}_{L^\infty \p{\nu}} \\
				&\leq  h_n^2\sum_{l=1}^k h_n^{2(l-1)} \norm{\p{h_n^2 \LL}^k  f_l}_{L^\infty \p{\nu}} + h_n^2 2^k C_l(p_0,\MM) \norm{f}_{\CC^\beta\p{\MM}} h_n^{\beta-2}+ \Delta_n(f_0).			\end{align*}
	Now, 
	\begin{align*}
	\norm{\p{h_n^2 \LL}^k  f_l}_{L^\infty \p{\nu}} &\leq 2^{l-1} \norm{\p{h_n^2 \LL}^{k-l+1}  f_l}_{L^\infty \p{\nu}} \\
	& \leq C_{k-l}(H,\beta, p_0, f_0, \MM)[   h_n^{\beta-2l} + \Delta_{n,k-l}(f)]
	\end{align*}
	so that for all  $f \in \mathcal C^\beta$ 
	\begin{align*}
\norm{\p{h_n^2\LL}^{k+1}f}_{L^\infty \p{\nu}} & \leq 
h_n^2\sum_{l=1}^k C_{k-l}(\|f\|_{\mathcal C^\beta}, p_0, \MM) h_n^{2(l-1)}  h_n^{\beta-2l} +\sum_{l=0}^k C_{k-l}(\|f\|_{\mathcal C^\beta}, p_0, \MM) \Delta_{n,k-l}(f_0) \\
& \leq C'(\beta, f_0, p_0, \MM)[ h_n^\beta + \Delta_{n,k}(f)]
\end{align*}
where
$\Pro_0 \p{|\Delta_{n,k}(f)| >  M_0 \p{\frac{\ln N}{Nh_n^d}}^{1/2} h_n } \leq N^{-H}$ for $M_0$ large enough. This concludes the inductive argument. Therefore 
\begin{equation*}
	\norm{f_0 - e^{-t\LL}f_t}_{L^\infty\p{\nu}} \leq C'(H, \beta, f_0, p_0, \MM) \left( \frac{ t }{ h_n^2 } \right)^{k+1}[ h_n^\beta + \Delta_{n,k}(f_0)]
\end{equation*}
where $\Pro_0 \p{|\Delta_{n,k}(f_0)| >  M_0 \p{\frac{\ln N}{Nh_n^d}}^{1/2} h_n } \leq N^{-H}$.
		\end{itemize}

\subsection{Proof of Theorem \ref{thm:approximation2}} \label{sec:pr:thm:approx2}
	
We bound  $\norm{f_0 - p_{J_n} \p{e^{-t\LL}f_t}}_{L^\infty \p{\nu}} $ by showing that $\norm{ e^{-t\LL}f_t - p_{J_n} \p{e^{-t\LL}f_t}}_{L^\infty \p{\nu}} $ is small, using a lower bound on the eigenvalues $\lambda_j's$ given by lemma \ref{spectrum:lower_bound} : since $h_n$ satisfies assumption \ref{assumption:h}, choosing $h_-=h_n$ in the definition of $A_N$, together with Lemma \ref{spectrum:lower_bound}, we have that $\lambda_j^{(h_n)} \geq b_3 j^{2/d}$ for all  $b_1 \ln^d N \leq j \leq t_0(h_n) =  b \frac{h_n^{-d}}{\ln^{d/2} N}, $, so that since $ \ln^\kappa N \leq J$ with $\kappa > d$, we have $\lambda_{J} = \lambda_{J}^{(h_n)} \geq  b_3 J^{2/d}$ when 
$ J \leq b \frac{h_n^{-d}}{\ln^{d/2} N}:= \bar J_n$.

Moreover, because the $\lambda_j's$ are ordered we have 
\[
\forall J \geq\bar J_n , \quad \lambda_J \geq \lambda_{\lfloor \bar J_n\rfloor } \geq b_3 b^{2/d} \frac{h_n^{-2}}{\ln N}
\]
Therefore
		\begin{align*}
		\norm{f_0 - p_{J} \p{e^{-t\LL}f_t}}_{L^\infty \p{\nu}} \leq & \norm{f_0 - e^{-t\LL}f_t}_{L^\infty \p{\nu}} + \norm{\sum_{j > J} \inner{u_j|e^{-t\LL} f_t}_{L^2\p{\nu}} u_j }_{L^\infty \p{\nu}} \\
		= & \norm{f_0 - e^{-t\LL}f_t}_{L^\infty \p{\nu}} + \norm{\sum_{j > J} e^{-t\lambda_j} \inner{u_j|f_t}_{L^2\p{\nu}} u_j }_{L^\infty \p{\nu}} \\
		\leq & \norm{f_0 - e^{-t\LL}f_t}_{L^\infty \p{\nu}} + \sum_{j > J} e^{-t\lambda_j} \abs{\inner{u_j|f_t}_{L^2\p{\nu}}} \norm{u_j}_{L^\infty \p{\nu}} \\
		\leq & \norm{f_0 - e^{-t\LL}f_t}_{L^\infty \p{\nu}} + \norm{f_t}_{L^2\p{\nu}} e^{-t\lambda_{J}} \sum_{j > J} \norm{u_j}_{L^\infty \p{\nu}}
		\end{align*}
		By lemma \ref{proposition:uniform_norm_u_j} we have $\max_{1 \leq j \leq N} \norm{u_j}_{L^\infty \p{\nu}} \leq N$ and therefore
		\[
		\norm{f_0 - p_{J} \p{e^{-t\LL}f_t}}_{L^\infty \p{\nu}} \lesssim \norm{f_0 - e^{-t\LL}f_t}_{L^\infty \p{\nu}} + N^2 \norm{f_t}_{L^2\p{\nu}} e^{-t\lambda_J}
		\]
		Moreover, since $f_t = \sum_{l=0}^k \frac{1}{l!} \p{t\LL}^l f_0$ and $\norm{h_n^2\LL}_{\LL \p{L^\infty \p{\nu}}} \leq 2$ we get $\norm{f_t}_{L^\infty \p{\nu}} \leq  e^2 \p{\frac{t}{h_n^2}}^k \norm{f_0}_{L^\infty\p{\MM}}$, which implies
		\[
		\norm{f_0 - p_J \p{e^{-t\LL}f_t}}_{L^\infty \p{\nu}} \lesssim \norm{f_0 - e^{-t\LL}f_t}_{L^\infty \p{\nu}} + N^2 e^2 \p{\frac{t}{h_n^2}}^k e^{-t\lambda_J}
		\]
		Thus for any $K>0$, choosing $t = c\lambda_{J}^{-1}\ln N$ for $c >0$ large enough yields
		\[
		\norm{f_0 - p_J \p{e^{-t\LL}f_t}}_{L^\infty \p{\nu}} \lesssim \norm{f_0 - e^{-t\LL}f_t}_{L^\infty \p{\nu}} + N^{-K}
		\]
		Finally we bound $\norm{f_0 - e^{-t\LL}f_t}_{L^\infty \p{\nu}}$ using Theorem \ref{thm:approximation} which gives us for $K>0$ large enough and arbitrarily large $H>0$
		\[
		\Pro_0 \p{ \norm{f_0 - p_J \p{e^{-t\LL}f_t}} > C(H,\beta,f_0,p_0,\MM) \tilde \varepsilon_n } \leq N^{-H}
		\]
		where
		\begin{align*}
		\tilde \varepsilon_n = & \p{\frac{t}{h_n^2}}^{\lceil \beta/2 \rceil} \p{ h_n^\beta + \II_{\beta > 1} \p{\frac{\ln N}{Nh_n^d}}^{1/2} h_n } 
		\lesssim  \p{\frac{\lambda_{J}^{-1} \ln N}{h_n^2}}^{\lceil \beta/2 \rceil} \p{ h_n^\beta + \II_{\beta > 1} \p{\frac{\ln N}{Nh_n^d}}^{1/2} h_n } \\
		\lesssim & \begin{cases}
			\p{\frac{J^{-2/d} \ln N}{h_n^2}}^{\lceil \beta/2 \rceil} \p{ h_n^\beta + \II_{\beta > 1} \p{\frac{\ln N}{Nh_n^d}}^{1/2} h_n } \text{ if } J \leq \bar J_n \\
			\p{\ln N}^{2 \lceil \beta/2 \rceil} \p{ h_n^\beta + \II_{\beta > 1} \p{\frac{\ln N}{Nh_n^d}}^{1/2} h_n } \text{ if } J > \bar J_n
		\end{cases}
		\end{align*}
		which concludes the proof.

\subsection{Lemmas \ref{lemma:regularity} and \ref{lemma:concentration}}

below we present two Lemmas which control the terms $\Delta_n(f)$ together with the behaviour of $T_h f$.

	\begin{lem}\label{lemma:regularity}
		Let $a_\MM,C_0,h_+$ the constants defined in Appendix \ref{sec:regularity_M_holder_spaces} and \ref{appendix:geometry_of_the_random_graph}. Then
		\begin{enumerate}
			\item If $0 < \beta \leq 2, h \leq \frac{\pi a_\MM}{C_0}, f \in \CC^\beta \p{\MM}$ (with the additional assumption $p_0 \in \CC^{\beta-1}\p{\MM}$ if $1 < \beta \leq 2$), then $\norm{T_hf}_{L^\infty\p{\MM}} \lesssim \norm{f}_{\CC^\beta\p{\MM}} h^{\beta-2}$
			\item If $\beta > 2, f \in \CC^\beta \p{\MM}, p_0 \in \CC^{\beta-1}\p{\MM}, h \leq h_+$, then there exists $\p{g_h^{(l)}}_{l=1}^k ,g_h^{(l)} \in \CC^{\beta-2l}\p{\MM}$ such that
			\[
			\norm{g_h^{(l)}}_{\CC^{\beta-2l}\p{\MM}} \lesssim \norm{f}_{\CC^\beta\p{\MM}}, \norm{T_hf - \sum_{l=1}^k h^{2(l-1)}g_h^{(l)}}_{L^\infty\p{\MM}} \lesssim \norm{f}_{\CC^\beta\p{\MM}}h^{\beta-2}
			\]
		\end{enumerate}
	\end{lem}	
	
	\begin{lem} \label{lemma:concentration}
		Let $h_n > 0$ satisfying assumption \ref{assumption:h}. Then
		\begin{itemize}
			\item If $f \in \CC^\beta\p{\MM}, 0 < \beta \leq 1$ we have
			\[
			\norm{\LL f}_{L^\infty \p{\nu}} \lesssim h_n^{\beta-2}, \norm{T_{h_n} f}_{L^\infty \p{\MM}} \lesssim h_n^{\beta-2}
			\]
			\item If $f \in \CC^1\p{\MM}$ then any $H>0$ there exists $M_0 > 0$ such that
			\[
			\Pro \p{\norm{\LL f - T_{h_n} f}_{L^\infty \p{\nu}} > M_0 \p{\frac{\ln N}{N}}^{\frac{1}{2}} h_n^{-(1+d/2)}} \leq N^{-H}
			\]
		\end{itemize}
	\end{lem}
	The proofs of both lemmas are provided in Sections \ref{proof:lemma:regularity} and \ref{proof:lemma:concentration} respectively. 
	
	\subsection{Proof of Theorems \ref{theorem:rate_prior1_d_N} and \ref{theorem:rate_prior2_d_N}} \label{appendix:d_N}
	
	We now show how to extend the posterior contraction rates in terms of $\|f-f_0 \|_n$ obtained in Theorems \ref{theorem:rate_prior1_d_n} and \ref{theorem:rate_prior2_d_n} to rates in terms of $\|  f-f_0\|_N$. 
	
	 As in \cite{vaart_information_2011,rosa_posterior_2023} where a posterior contraction rate with respect to the empirical $\|f-f_0 \|_n$ norm is used to show a contraction rate with respect to the continuous $L^2$ norm, we will use a concentration inequality : indeed, by exchangeability the variables $\p{x_i}_{i=1}^n$ can be considered as sampled uniformly without replacement from the full sample $\p{x_i}_{i=1}^N$, and therefore the quantity $\norm{f-f_0}_n^2 = \frac{1}{n} \sum_{i=1}^n \p{f(x_i) - f_0(x_i)}$ is close to its mean $\frac{1}{N} \sum_{i=1}^N \p{f(x_i) - f_0(x_i)} = \norm{f-f_0}_N^2$. We make this informal argument rigorous below.
	
 	We start by the non adaptive case, i.e the proof of theorem \ref{theorem:rate_prior1_d_N}. Since $h_n \geq \p{\ln N}^s n^{-\frac{1}{2\beta+d}}$ , it   satisfies assumption \ref{assumption:h}, and since  $\frac{h_n^{-d}}{\ln^{\tau_1} N} \leq J_n \leq \frac{h_n^{-d}}{\ln^{\tau_2} N}, \tau_1 \geq \tau_2 > 2d$,
	Theorem \ref{theorem:rate_prior1_d_n} implies that, for some $C>0$
	\[
	\Pi \c{\norm{f-f_0}_n > C\varepsilon_n(J_n, h_n) | \mathbb{X}^n, J_n, h_n} \xrightarrow[n \to \infty]{P_0^\infty} 0
	\]
	where $ \epsilon_n(J_n; h_n)$ is given by \eqref{epsilonn}. Note that the choice of $J_n$ implies that
	\[
		\varepsilon_n(J_n, h_n) =
		\sqrt{\frac{J_n \ln N}{n}} + \p{\frac{J_n^{-2/d} \ln N}{h_n^2}}^{\lceil \beta/2 \rceil} \p{ h_n^\beta + \II_{\beta > 1} \p{\frac{\ln N}{Nh_n^d}}^{1/2} h_n }
	\]

Let
	\[
	f_n = p_{J_n}\p{e^{-t_n \LL}f_{t_n}}, \quad f_{t_n} = \sum_{l=0}^{\lceil \beta/2 \rceil - 1} \frac{\p{t_n\LL}^l f_0}{l!},  \quad t_n = c \lambda_{J_n}^{-1} \ln N, c>0.
	\]
	Then Theorem \ref{thm:approximation2} implies that for some $c>0$ and any $H>0$, choosing $C$ large enough, 
	$$ \mathbb P_0\left( \norm{f_n-f_0}_{L^\infty\p{\nu}} \leq C\varepsilon_n(J_n, h_n) \right)   = 1 + O(N^{-H}).$$ 

	Now if $M > C$, then on $\{\norm{f_n-f_0}_{L^\infty\p{\nu}} \leq C\varepsilon_n \} = V_n $,  
	\begin{align}\label{bound1dN}
		&\Pi \c{\norm{f-f_0}_N > M\varepsilon_n(J_n, h_n)| \mathbb{X}^n, J_n, h_n} 
		\leq  \Pi \c{\norm{f-f_n}_N > \p{M-C}\varepsilon_n(J_n, h_n)| \mathbb{X}^n, J_n, h_n} \nonumber \\
		&\leq \Pi \c{\norm{f-f_n}_n > \p{C+1} \varepsilon_n(J_n, h_n) | \mathbb{X}^n, J_n, h_n} + \Pi \c{\norm{f-f_n}_N > \p{M-C}\varepsilon_n(J_n, h_n) \geq \frac{M-C}{C+1} \norm{f-f_n}_n | \mathbb{X}^n, J_n, h_n}.
	\end{align}
	Since
$	\norm{f-f_n}_n \leq \norm{f-f_0}_n + \norm{f_0-f_n}_{L^\infty \p{\nu}} \leq \norm{f-f_0}_n + C\varepsilon_n(J_n, h_n)$,
	we have on $V_n$
	\[
	\Pi \c{\norm{f-f_n}_n > \p{C+1} \varepsilon_n(J_n, h_n) | \mathbb{X}^n, J_n, h_n} \leq \Pi \c{\norm{f-f_0}_n > \varepsilon_n(J_n, h_n) | \mathbb{X}^n, J_n, h_n} = o_{\Pro_0}(1).
	\]
	Hence it remains to bound the second term of the right hand side of \eqref{bound1dN}. 
 For this notice that a consequence of the proof of Theorem \ref{theorem:master_theorem}, see \cite{ghosal:vdv:07}, is that there exists $c'>0$ such that 
 $$	\mathbb P_0\left[ D_n \leq e^{ - c_1 n \varepsilon_n(J_n, h_n)^2 } \right] = o(1), \quad D_n = \int_{\Sigma^{J_n}} e^{\ell_n(f) - \ell_n(f_0)} d\Pi( f | J_n, h_n)  $$ 
 where $\ell_n(f)$ is the log-likelihood at $f$. As a consequence, 
\begin{align*}
		 \EE_0 &\c{\Pi \c{\norm{f-f_n}_N  \geq \frac{M-C}{C+1}\norm{f-f_n}_n | \mathbb{X}^n, J_n, h_n}} \\
		\leq & o(1) + e^{c'n\varepsilon_n^2(J_n, h_n)} \EE_0 \c{ \int_{\Sigma^{J_n}} \II_{\norm{f-f_n}_N > \frac{M-C}{C+1}\norm{f-f_n}_n}  \II_{\norm{f-f_n}_n < (C+1)\epsilon_n} e^{\ell_n(f) - \ell_n(f_0)} d\Pi( f | J_n, h_n) } \\
		\leq & o(1) + e^{c'n\varepsilon_n^2(J_n, h_n)} \EE_0 \int_{f \in \Sigma^{J_n}} \II_{\norm{f-f_n}_N > \frac{M-C}{C+1}\norm{f-f_n}_n} \Pi (df)
	\end{align*}
	where the last expectation is now only with respect to the distribution of $x_{1:N}$. 
	Notice that by exchangeability of the $x_i$'s, for any permutation $\tau \in \Scal_N$ the set of permutations of $\b{1,\ldots,N}$ we have $\p{x_i}_{i=1}^N \overset{(d)}{\equiv} \p{x_{\tau(i)}}_{i=1}^N$ and therefore, writing $\norm{f-f_n}_{\tau,n}^2 = \frac{1}{n} \sum_{i=1}^n \p{f(x_{\tau(i)}) - f_n(x_{\tau(i)})}^2
$, 
	\[
	\EE_0 \int_{f \in \Sigma^{J_n}} \II_{\norm{f-f_n}_N > \frac{M-C}{C+1}\norm{f-f_n}_n} \Pi (df) = \EE_0 \int_{f \in \Sigma^{J_n}} \II_{\norm{f-f_n}_N >\frac{M-C}{C+1}\norm{f-f_n}_{\tau,n}} \Pi (df). 
	\]
	Hence
	\begin{align*}
		 \EE_0& \int_{f \in \Sigma^{J_n}} \II_{\norm{f-f_n}_N > \frac{M-C}{C+1}\norm{f-f_n}_n} \Pi (df) 
		=  \frac{1}{N!} \sum_{\tau \in \Scal_N} \EE_0 \int_{f \in \Sigma^{J_n}} \II_{\norm{f-f_n}_N >\frac{M-C}{C+1}\norm{f-f_n}_{\tau,n}} \Pi (df) \\
		= & \EE_{\tau \sim \UU \p{\Scal_N} } \EE_0 \int_{f \in \Sigma^{J_n}} \II_{\norm{f-f_n}_N >\frac{M-C}{C+1}\norm{f-f_n}_{\tau,n}} \Pi (df) 
		=  \EE_0 \int_{f \in \Sigma^{J_n}} \Pro_{\tau \sim \UU \p{\Scal_N}} \p{\norm{f-f_n}_N > \frac{M-C}{C+1}\norm{f-f_n}_{\tau,n}} \Pi(df).
	\end{align*}
	Define for 
	$ i = 1,\ldots,N$,  $W_i = \p{f(x_i) - f_n(x_i)}^2$ and $ W_i' = \p{f(x_{\tau(i)}) - f_n(x_{\tau(i)})}^2$. 
	Then given $\p{x_i}_{i=1}^N$, $\p{W_i'}_{i=1}^n$ is a uniform sampling without replacement from $\p{W_i}_{i=1}^N$. Hoeffding's lemma for random variables sampled without replacement from a finite population \cite{hoeffding} then yields, for $M = 3C+2 > C$
	\begin{align*}
	\Pro_{\tau \sim \UU \p{\Scal_N}} &\p{\norm{f-f_n}_N > \frac{M-C}{C+1}\norm{f-f_n}_{\tau,n}} =  \Pro_{\tau \sim \UU \p{\Scal_N}} \p{\frac{1}{N} \sum_{i=1}^N W_i^2 >  \frac{4}{n}\sum_{i=1}^n W_i'^2} \\
	& =  \Pro_{\tau \sim \UU \p{\Scal_N}} \p{\frac{-3}{4} \frac{1}{N} \sum_{i=1}^N W_i^2 > \frac{1}{n}\sum_{i=1}^n W_i'^2 - \frac{1}{N} \sum_{i=1}^N W_i^2} 
	=  \Pro_{\tau \sim \UU \p{\Scal_N}} \p{\frac{-1}{n}\sum_{i=1}^n W_i'^2 + \frac{1}{N} \sum_{i=1}^N W_i^2 > \frac{3}{4} \frac{1}{N} \sum_{i=1}^N W_i^2 } \\
&	\leq  \exp \p{-\frac{3n\norm{f-f_n}_N^2}{2\norm{f-f_n}_{L^\infty \p{\nu}}^2}}.
	\end{align*}
	Using theorem \ref{thm:norm_comparison} yields, on $V_n \cap A_N$ (which is an event of probability tending to $1$), with $h_- = h_n$ and if $b_4 \leq J_n \leq b_5 \frac{t_0^{-d/2}}{\ln^{3d/2}N}$ (where $t_0 = 64 h_n^2 \ln \p{2c_- Nh_n^d}$)
	\[
	\forall f \in \Sigma^{J_n}, \quad \frac{\norm{f-f_n}_N^2}{\norm{f-f_n}_{L^\infty \p{nu}}^2} \geq e^{-1} a_3^{-1} b_6^{-d/2} J_n^{-1} \ln^{-3d/2} N.
	\]
Note that since by assumption  $\frac{h_n^{-d}}{\ln^{\tau_1} N} \leq J_n \leq \frac{h_n^{-d}}{\ln^{\tau_1} N}, \tau_1 \geq \tau_2 > 2d$, then $b_4 \leq J_n \leq b_5 \frac{t_0^{-d/2}}{\ln^{3d/2}N}$. All in all we get
	\begin{align*}
		 \EE_0 &\c{\Pi \c{\norm{f-f_n}_N > \p{M-C}\varepsilon_n \geq \frac{M-C}{C+1}\norm{f-f_n}_n | \mathbb{X}^n, J_n, h_n}} \\
		&\leq  o(1) + e^{c'n\varepsilon_n^2(J_n, h_n)} \EE_0 \left[\II_{A_N\cap V_n}\int_{f \in \Sigma^{J_n}} \exp \p{-\frac{3n\norm{f-f_n}_N^2}{2\norm{f-f_n}_{L^\infty \p{\nu}}^2}} \Pi (df)\right] \\
		& \leq  o(1) + e^{c'n\varepsilon_n^2(J_n, h_n)} \EE_0 \int_{f \in \Sigma^{J_n}} \exp \p{- \frac{3nJ_n^{-1}  }{2e a_3 b_6^{d/2} \ln^{3d/2} N}} \Pi (df) \\
		& = o(1) + \exp \p{c'n\varepsilon_n^2(J_n, h_n) - \frac{3nJ_n^{-1}  }{2e a_3 b_6^{d/2} \ln^{3d/2} N}  }.
	\end{align*}
	Moreover
	\[
	J_n^{-1} \ln^{-3d/2}N \geq h_n^d \ln^{\tau_2 - 3d/2}N
	\]
	and by definition of $\varepsilon_n(J_n, h_n)$ we have
	\[
	\varepsilon_n^2(J_n, h_n) \lesssim \frac{J_n \ln N}{n} + \p{\frac{J_n^{-2/d} \ln N}{h_n^2}}^{2\lceil \beta/2 \rceil} h_n^{2\beta} \leq \frac{h_n^{-d} \ln^{1-\tau_2} N}{n} + h_n^{2\beta}\ln^{2\lceil \beta/2 \rceil\p{1+2\tau_1/d}} N.
	\]
	But since by assumption $h_n \geq \p{\ln N}^s n^{-\frac{1}{2\beta+d}}$ for some $s \in \RR$ and $\beta > d/2$, this actually implies $nJ_n^{-1} \ln^{-3d/2} N >> n\varepsilon_n^2(J_n, h_n)$, which concludes the proof.
	
	The proof of the adaptive case, i.e of theorem \ref{theorem:rate_prior2_d_N}, is fairly similar. Again, we use Theorem \ref{theorem:rate_prior2_d_n}, so that there exists $C>0$
	$$\EE_0 \left( \Pi \c{\norm{f-f_0}_n > C\varepsilon_n| \mathbb{X}^n}\right) = o(1).$$ 
	where $\varepsilon_n= n^{-\beta/(2\beta+d)} (\ln N)^{\frac{ (2\tau + d)\lfloor \beta/2\rfloor- \beta(\tau + 2\beta/d) }{ 2\beta+d} }$. 
	
Moreover, since by assumption on $\Pi_J$ for any $J_n,k$ we have $\Pi \c{J > kJ_n} \lesssim e^{-a_2k J_n}$, the remaining mass theorem (theorem 8.20 in \cite{ghosal_fundamentals_2017}) together with our prior thickness result \ref{lem:prior_thickness_prior1} yields, for some $k > 0$
	\[
	\Pi \c{J > kJ_n | X^n} \xrightarrow[n \to \infty]{P_0^\infty} 0
	\]
	for $J_n = n \varepsilon_n^2$. 

	Therefore, a similar proof as in the non adaptive case above leads to the upper bound: on $V_n \cap A_N$, 
	\begin{align*}
\Pi \c{\norm{f-f_0}_N > M\varepsilon_n| \mathbb{X}^n} \leq & o_{\Pro_0}(1) + e^{c'n\varepsilon_n^2(J_n, h_n)} \sum_{J \leq k J_n} \EE_0 \int_{ f \in \Sigma^J} \exp \p{-\frac{3n\norm{f-f_n}_N^2}{2\norm{f-f_n}_{L^\infty \p{\nu}}^2}} d\Pi (f|F) \Pi_J(J)
	\end{align*}
	for some $c'>0$. Moreover, since by construction $\Pi-$almost surely we have $J \leq \frac{h^{-d}}{\ln^{\tau/d}N}, \tau > 2d$, we can still apply theorem \ref{thm:norm_comparison} to get  on $V_n \cap A_N$, 
	 \begin{align*}
\Pi \c{\norm{f-f_0}_N > M\varepsilon_n| \mathbb{X}^n} &\leq  o_{\Pro_0}(1) + e^{c'n\varepsilon_n^2}\sum_{J \leq k J_n} \EE_0 \int_{f \in \Sigma^J} \exp \p{- \frac{3nJ^{-1} }{2e a_3 b_6^{d/2} \ln^{3d/2} N} } d\Pi(f|J)\Pi_J(J)\\
	 	&\leq  o_{\Pro_0}(1) + e^{c'n\varepsilon_n^2}  \sum_{J \leq k J_n}\exp \p{- \frac{3nJ^{-1} }{2e a_3 b_6^{d/2} \ln^{3d/2} N} }\\ 
	 	&\leq   o_{\Pro_0}(1) + J_n\exp \p{c'n\varepsilon_n^2  - \frac{3nJ_n^{-1} }{2e a_3 b_6^{d/2} \ln^{3d/2} N} } \\ 
	 	&=  o_{\Pro_0}(1) + \exp \p{c'n\varepsilon_n^2 - \frac{3}{2ek a_3 b_6^{d/2}} \varepsilon_n^{-2} \ln^{-3d/2} N}
	 \end{align*}
	 Given the expression of $\varepsilon_n$ and since $\beta > d/2$, the right hand side is $o_{\Pro_0}(1)$, which concludes the proof of the adaptive case.

\section*{Acknowledgements}
The project leading to this work has received funding from the European Research Council
(ERC) under the European Union’s Horizon 2020 research and innovation programme (grant agreement No 834175).

	\bibliographystyle{plain}
	\bibliography{ref}
	
	\appendix 
	
	\section{Regularity assumption of the manifold and definition of the H\"older spaces}\label{sec:regularity_M_holder_spaces}
	
	In this section we give details on the regularity of the submanifold $\MM$ and the definition of the H\"older spaces. It is usually assumed in Riemannian geometry that the manifolds are $\CC^\infty-$smooth but here our results also apply for manifolds with finite H\"older regularity if the latter is large enough, hence we precise our notations and recall some general facts. We assume that $\MM$ to be a $d-$dimensional $\CC^\alpha$ compact and connected submanifold of $\RR^D$ in the following sense : we can find a family $\p{\varphi_i,\UU_i}_{i \in I}$ such that $\p{\UU_i}_{i \in I}$ is an open cover of $\MM$ (that is, each $\UU_i$ is an open set of $\MM$ and $\MM = \cup_{i \in I} \UU_i$), $\varphi_i : \varphi_i^{-1}\p{\UU_i} \subset \RR^d \to \UU_i \subset \MM \subset \RR^D$ is injective, of class $\CC^\alpha, \alpha \geq 1$ and its differential $d \varphi_i : \RR^d \to \RR^D$ is everywhere injective (hence $\varphi_i$ is a $\CC^\alpha$ immersion). Since $\MM$ is compact we can assume without loss of generality that $I$ is finite, but the regularity results below stay true in the general case, at least locally.
	
	The geodesics on the submanifold are defined as the curves $\gamma(t)$ on $\MM$ satisfying the geodesics equation : in local coordinates $y = \varphi^{-1}(\gamma)$ (we drop the subscript $i$ for convenience) it takes the form
	\[
	\forall k = 1,\ldots,d, \ddot{y}^k + \Gamma_{ij}^k \dot{y}^i \dot{y}^j = 0
	\]
	where the Christoffel symbols are defined by
	\[
	\Gamma_{ij}^k = \frac{1}{2} g^{kl} \p{\partial_i g_{jl} + \partial_j g_{il} - \partial_l g_{ij}}
	\]
	Here both the metric tensor $g$, its inverse $\p{g^{kl}}$ and the Christoffel symbols depend on the point $\gamma(t) = \varphi (y(t))$. Moreover, with $\p{\partial_i := \frac{\partial \varphi}{\partial y_i}}_{i=1}^d$ the coordinates vector fields and $\varphi = \p{\varphi^l}_{l=1}^D$, the metric tensor is given by the Gram matrix
	\[
	\forall p \in \MM, g_{ij}(p) = \inner{\partial_i | \partial_j}_{T_p\MM} = \sum_{l=1}^D \frac{\partial \varphi^l}{\partial y^i}(\varphi^{-1}(p)) \frac{\partial \varphi^l}{\partial y^j}(\varphi^{-1}(p)) = \p{\frac{\partial \varphi}{\partial y_i}}^T\p{\frac{\partial \varphi}{\partial y_j}}\p{\varphi^{-1}(p)}
	\]
	and at $p = \gamma(t)$
	\[
	\partial_l g_{ij} = \frac{\partial \p{g_{ij} \circ \varphi}}{\partial y^l} \p{y(t)} = \c{ \frac{\partial}{\partial y^l} \sum_{k=1}^D \frac{\partial \varphi^k}{\partial y^i} \frac{\partial \varphi^k}{\partial y^j} }(y(t))
	\]
	Hence for each $i,j,k$, $\Gamma_{ij}^k$ is a $\CC^{\alpha-2}$ H\"older function of $y(t)$. Therefore by standard ODE theory, if $\alpha \geq 3$ existence and uniqueness of geodesics $t \mapsto \gamma_{x,v}(t)$ defined on maximal intervals and such that $\varphi(y(0)) = \gamma_{x,v}(0) = x, \frac{\partial \varphi}{\partial y}(\varphi^{-1}(x)) \dot{y}(0) = \gamma_{x,v}'(0) = v$ for arbitrary $x \in \MM, v \in T_x\MM$ is guaranteed. Moreover the flow $\p{x,v,t} \mapsto \p{\gamma_{x,v}(t),\gamma_{x,v}'(t)}$ is of class $\CC^{\alpha-2}$, in the sense that 
	\[
	\p{y_0,\dot{y}_0,t} \mapsto \p{\gamma_{\varphi(y_0), \frac{\partial \varphi}{\partial y}(y_0) \dot{y}_0}(t), \gamma_{\varphi(y_0), \frac{\partial \varphi}{\partial y}(y_0) \dot{y}_0}'(t)}
	\]
	is of class $\CC^{\alpha-2}$ on its domain. In particular the exponential map defined as $\exp_x(v) = \gamma_{x,v}(1)$ for all $\norm{v}_{T_x\MM} < r_x$ (the injectivity radius of $\MM$ at $x$) is in this sense of class $\CC^{\alpha-2}$ on its domain, as seen as a function with values in $\RR^D$. Since $\MM$ is compact it has a positive global injectivity radius $r_\MM := \inf_{x \in \MM} r_x$ and therefore for any $x \in \MM$ the exponential map $\exp_x$ is well defined from $B_{T_x\MM}\p{0,r_\MM}$ to $\MM$ and is a diffeomorphism onto its image. Moreover, by compactness of $\MM$ the Hopf-Rinow theorem implies that for each $x \in \MM$ the exponential map $\exp_x$ is actually defined on the whole of $T_x\MM$ (while obviously no longer being guaranteed to be injective). Hence by the above, for any $i \in I$ the following map
	\[
	\p{\begin{array}{ccc}
			\varphi_i^{-1}\p{\UU_i} \times \RR^d & \to & \RR^D \\
			\p{y_0, \dot{y}_0} & \mapsto & \exp_{\varphi_i(y_0)} \p{\frac{\partial \varphi_i}{\partial y}(y_0)\dot{y}_0}
	\end{array}} 
	\]
	is of class $\CC^{\alpha-2}$.
	
	Finally let us define the different H\"older spaces over $\MM$ used in the paper. Let $\p{\chi_i}_{i \in I}$ be a partition of the unity subordinated to the open cover $\p{\UU_i}_{i\in I}$. We say that $f : \MM \to \RR$ is of class $\CC^\beta$ ($\beta$-H\"older) on $\MM$ if $f \circ \varphi_i : \varphi_i^{-1} \p{\UU_i} \to \RR$ is of class $\CC^\beta$ for any $i \in I$. For a finite dimensional vector space $V$ we say that $f : \MM \to V$ is of class $\CC^\beta$ ($\beta$-H\"older) on $\MM$ if $T \circ f \circ \varphi_i : \varphi_i^{-1} \p{\UU_i} \to \RR$ is of class $\CC^\beta$ for any $i \in I$ and linear form $T \in V^*$. Equivalently the coordinates of $f$ in any choice of basis of $V$ are all of class $\CC^\beta$. We also define $\CC^\infty\p{\MM} = \cap_{\beta \geq 0} \CC^\beta\p{\MM}$. Furthermore we define the $\CC^\beta$ norm of $f$ by $\norm{f}_{\CC^\beta \p{\MM}} = \max_{i \in I} \norm{\p{\chi_i f} \circ \varphi_i^{-1}}_{\CC^\beta\p{\UU_i}}$. The definition of the H\"older space $\CC^\beta \p{\MM,V}$ can be seen to be independent of the chosen family $\p{\varphi_i,\UU_i,\chi_i}_{i \in I}$ as long as $\alpha \geq \beta \vee 1$, with equivalence of the resulting H\"older norms.
	
	Recall that we can always define regular local orthonormal frames on $T\MM$ : on $\UU_i$ simply apply the Gram-Schmidt orthonormalization algorithm to the coordinates vector fields $x \mapsto \frac{\partial \varphi}{\partial y_i}$, which gives a family $\p{e_i^j}_{i \in I, j = 1,\ldots,d}$ of orthonormal vector fields whose components in local coordinates are rational functions of the components of the coordinates vector fields $\frac{\partial \varphi_i}{\partial y_i}$. As $\varphi_i : \varphi_i^{-1}\p{\UU_i} \subset \RR^d \mapsto \UU_i \subset \RR^D$ is $\CC^{\alpha-1}$, the maps: $x \in \varphi_i^{-1}\p{\UU_i} \subset \RR^d \mapsto e_i^j(x) \in \UU_i \subset \RR^D$ are all $\CC^{\alpha-1}$. Moreover by construction the function $\omega  \mapsto \psi_i(x,\omega) := \sum_j \omega_j e_i^j(x) $ is a linear isometry  from $\in \RR^d$ to $ T_x\MM$  for each $x \in \UU_i$ where $\RR^d$ is equipped with its standard Euclidean structure and $T_x\MM$ its inner product induced by the metric.
	
	In particular if $\alpha - 2 \geq \beta \vee 1 \iff \alpha \geq \p{\beta+2} \vee 3$ then any $f \in \CC^\beta \p{\MM}$ satisfies
	\begin{align*}
	& \sup_{x \in \UU_i} \norm{f \circ \exp_x\p{\psi_i(x,\cdot)}}_{\CC^\beta \p{B_{\RR^d}(0,r_\MM)}} = \sup_{x \in \UU_i} \norm{\sum_j \p{\chi_j f} \circ \exp_x\p{\psi_i(x,\cdot)}}_{\CC^\beta \p{B_{\RR^d}(0,r_\MM)}} \\
	& \leq \sum_j \sup_{x \in \UU_i} \norm{\p{\chi_j f} \circ \exp_x\p{\psi_i(x,\cdot)}}_{\CC^\beta \p{B_{\RR^d}(0,r_\MM)}} \\
	& = \sum_j \sup_{x \in \UU_i} \norm{\p{\p{\chi_j f} \circ \varphi_j} \circ \p{\varphi_j^{-1} \circ \exp_x\p{\psi_i(x,\cdot)}}}_{\CC^\beta \p{B_{\RR^d}(0,r_\MM)}} \\
	& \lesssim \max_j \norm{ \p{\chi_j f} \circ \varphi_j }_{\CC^\beta \p{\UU_j}} \sum_j \sup_{x \in \UU_i} \norm{\varphi_j^{-1} \circ \exp_x\p{\psi_i(x,\cdot)}}_{\CC^\beta \p{ B_{\RR^d}\p{0,r_\MM},\varphi_j^{-1}\p{\UU_i}}} \lesssim \norm{f}_{\CC^\beta\p{\MM}},
	\end{align*}
	the last step using resulting from the compactness of $\MM$, finiteness of $I$ and the fact that $\alpha - 2 \geq \beta$. We will use this property in our approximation results.

We recall that we assume $\alpha \geq \p{\beta + 3} \vee 6$. This is for technical reasons that will appear later in the proofs.
	
We finish this section by showing a differential geometry lemma that will be useful for the proofs of the results in section \ref{section:main_results} : very roughly speaking, this lemma states that for small $h$ we can find a function $t = t_h(x,v)$ of $x \in \MM, v \in T_x\MM$ such that $\norm{\exp_x(thv) - x}_{\RR^D} = h$, and that moreover this function can be written as $t_h = 1 + h^2 s_h$ where $s_h$ is H\"older regular and has bounded derivatives of order less than $\alpha - 5$, even when $h$ is close to $0$. First, recall that the radius of curvature (see e.g \cite{bernstein_graph_nodate}) $a_\MM$ of $\MM$ is positive by compactness of $\MM$.
	
\begin{lem}\label{lemma:geometry1}
	There exists $r_\MM' > 0$ such that for each $x \in \MM, v \in T_x \MM, \norm{v} = 1$ the function $r \mapsto \norm{\exp_x(rv)-x}$ is increasing on $[0,r_\MM']$ and satisfies $\norm{\exp_x(rv)-x}^2 \geq \frac{r^2}{4}, r \in [0,r_\MM']$. 
		
	As a consequence for any $h \in [0,\frac{r_\MM'}{2}]$ there is a unique $t = t_h(x,v) \in [0, h^{-1} r_\MM']$ satisfying $\norm{\exp_x(thv)-x} = h$. Moreover $t_h \p{x,v} \in [1,1+ \frac{\pi^3 h^2}{192 a_\MM^2}]$ and there exists $h_+>0$ such that
	\[
	\max_{\mathbf k \in \mathbb{N}^d, \abs{\mathbf k} \leq \alpha - 5} \sup_{\substack{i \in I \\ y \in \varphi_i^{-1}\p{\UU_i} \\ 0 < h < h_+' \\ \norm{v}_{T_{\varphi_i(y)}\MM} = 1}} \abs{\frac{\partial^{\abs{\mathbf k}}}{\partial y^{\mathbf k}} \p{\frac{t_h(\varphi_i(y),v) - 1}{h^2}}} < +\infty
	\]
\end{lem}
\begin{proof}
	We have for $r < r_\MM$,
	\begin{align*}
		\frac{\partial}{\partial r} \norm{\exp_x(rv)-x}_{\RR^D}^2 
		&=  2 \inner{d\exp_x(rv)v|\exp_x(rv)-x}_{\RR^D}  =  2\inner{v + d\exp_x(rv)v - v|rv + \exp_x(rv)-x - rv}_{\RR^D} \\
		=&  2\inner{v|rv}_{\RR^D} + 2\inner{d\exp_x(rv)v - v|rv}_{\RR^D} + 2\inner{\exp_x(rv)- x -rv|v}_{\RR^D} \\ 
		& \qquad + 2\inner{d\exp_x(rv)v - v|\exp_x(rv)-x - rv}_{\RR^D} \\
		= & 2r + \mathcal{O}\p{r^2}
	\end{align*}
	where the last equality uses 
	\[ \exp_x(rv)- x -rv =  rd\exp_x(0)v - rv + O(r^2) = O(r^2), \quad  d\exp_x(rv)v - v = d\exp_x(0)v - v + O(r^2) = O(r^2).
	\] 
	By compactness of $\MM$ and since $\p{x,v} \mapsto \exp_x(v) \in \CC^{\alpha-3}, \alpha \geq 3$ this implies that
	\[
	\sup_{\substack{0 < r < r_\MM \\ x \in \MM \\ \norm{v}_{T_x \MM}=1}} \frac{\abs{\frac{\partial}{\partial r} \norm{\exp_x(rv)-x}_{\RR^D}^2 - 2r}}{r^2} < +\infty,
	\]
	In particular there exists $r = r_\MM'>0$ such that for any $r \in [0,r_\MM']$ we have 
	\[
	\forall x \in \MM, \forall v \in T_x\MM, \norm{v}_{T_x\MM}=1, \frac{\partial}{\partial r} \norm{\exp_x(rv)-x}_{\RR^D}^2 \geq \frac{r}{2}.
	\]
	As a consequence,  the function $r \mapsto \norm{\exp_x(rv)-x}_{\RR^D}^2$ is increasing on $[0,r_{\MM}']$ and so is $r \mapsto \norm{\exp_x(rv)-x}_{\RR^D}$. Moreover by integrating we get
	\[
	\forall r \in [0,r_\MM'], \norm{\exp_x(rv)-x}_{\RR^D}^2 \geq \frac{r^2}{4},
	\]
	so that,  if $h \in [0, \frac{r_\MM'}{2}]$, then  $r = 2h \in [0,r_\MM']$ and  $\norm{\exp_x(rv)-x}_{\RR^D}^2 \geq \frac{\p{2h}^2}{4} = h^2$. By continuity and strict monotonicity there exists a unique $r = r_h(x,v) \in [0,2h]$  solution of $\norm{\exp_x(rv)-x}=h$ on $[0,r_\MM']$. Defining $t_h(x,v) = h^{-1}r_h(x,v)$ gives the first part of the statement.
		
	Let $i \in I, x \in \UU_i, \omega \in \RR^d, \norm{\omega} = 1$ and consider the Taylor expansion (justified since $\alpha \geq 5$ and $\exp \in \CC^{\alpha-2}$)
	\[
	\exp_x(th \psi_i(x,\omega)) = x + th\psi_i(x,\omega) + \frac{t^2 h^2}{2} d^2 \exp_x(0).\psi_i(x,\omega) + h^3 \int_0^t \frac{(t-s)^2}{2} d^3 \exp_x(sh\psi_i(x,\omega))\psi_i(x,\omega)^3 ds.
	\]
	Changing of variables with $y \in \varphi_i^{-1}\p{\UU_i}, x = \varphi_i(y)$, we then have 
	\begin{align*}
		F(y,t) := & h^{-2}\norm{\exp_x(th \psi_i(x,\omega))-x}^2-1 \\
		= & h^{-2}\norm{th\psi_i(x,\omega) + \frac{t^2 h^2}{2} d^2 \exp_x(0).\psi_i(x,\omega) + h^3 \int_0^t \frac{(t-s)^2}{2} d^3 \exp_x(sh\psi_i(x,\omega))\psi_i(x,\omega)^3 ds}^2-1 \\
		= & \norm{t\psi_i(x,\omega) + \frac{t^2 h}{2} d^2 \exp_x(0).\psi_i(x,\omega) + h^2 \int_0^t \frac{(t-s)^2}{2} d^3 \exp_x(sh\psi_i(x,\omega))\psi_i(x,\omega)^3 ds}^2-1
	\end{align*}
	Therefore we have
	\[
	F(y,t) = t^2 - 1 + h^2 G(y,t)
	\]
	where, with
	\[
	A = t\psi_i(x,\omega), \quad B = \frac{t^2}{2} d^2 \exp_x(0).\psi_i(x,\omega),\quad C = \int_0^t \frac{(t-s)^2}{2} d^3 \exp_x(sh\psi_i(x,\omega))\psi_i(x,\omega)^3 ds
	\]
	we have defined
	\[
	G(y,t) = G_{i,h,\omega}(y,t) = 2h^{-1}\inner{A|B} + 2\inner{A|C} + \norm{B}^2 + 2h \inner{B|C} + h^2 \norm{C}^2
	\]
	Moreover, since $\norm{t\psi_i(x,\omega)}^2 = t^2$ and $d^2 \exp_x(0).\psi_i(x,\omega) \in N_x \MM$ (as the second derivative at time $0$ of the geodesic $t \mapsto \exp_x \p{t \psi_i(x,\omega)}$) this implies $\inner{A|B} = 0$ and therefore
	\[
	G(y,t) = 2\inner{A|C} + \norm{B}^2 + 2h \inner{B|C} + h^2 \norm{C}^2
	\]
	Furthermore since the exponential map is of class $\CC^{\alpha-2}$ we see that  $G$ is of class $\CC^{\alpha - 5}$ and that
	\[
	M := \max_{\substack{i \in I \\ \p{\mathbf k,l} \in \mathbb{N}^{d+1} \\ \abs{\mathbf k} + l \leq \alpha - 5}} \sup_{\substack{y \in \varphi_i^{-1}\p{\UU_i} \\ 1 \leq t \leq 1 + \frac{\pi}{2} \\ 0 < h < r_\MM'/2 \\ \norm{\omega} = 1}} \abs{\frac{\partial^{\abs{\mathbf k}+l} G}{\partial y^{\abs{\mathbf k}} \partial t^l}} < +\infty.
	\]
	Differentiating yields
	\[
	\frac{\partial F}{\partial y} = h^2 \frac{\partial G}{\partial y}, \frac{\partial F}{\partial t} = 2t + h^2 \frac{\partial G}{\partial t}.
	\]
	Also, notice that $1 \leq t_h(x,v) \leq 1 + \frac{\pi^3 h^2}{192a_\MM^2} \wedge \frac{\pi}{2}$ : indeed, without loss of generality $r_\MM' \leq \pi a_\MM$ and therefore if $h \leq r_\MM'/2$ then for $t = t_h(x,v) \in [0,2]$ we have $th = \rho \p{x,\exp_x(th\psi_i(v))} \leq r_\MM' \leq \pi a_\MM$ and
	\[
	\frac{2}{\pi} th = \frac{2}{\pi} \rho(x,\exp_x(th\psi_i(x,\omega))) \leq \norm{\exp_x(th\psi_i(x,\omega))-x} = h \leq \rho(x,\exp_x(th\psi_i(x,\omega))) = th
	\]
	\[
	\implies 1 \leq t_h(x,v) \leq \frac{\pi}{2}.
	\]
	Reusing this gives
	\[
	th - \frac{\p{th}^3}{24a_\MM^2} = \rho \p{x,\exp_x(th\psi_i(x,\omega))} - \frac{\rho \p{x,\exp_x(th\psi_i(x,\omega))}^3}{24a_\MM^2} \leq \norm{\exp_x(th\psi_i(x,\omega))-x} = h \leq \rho \p{x,\exp_x(th\psi_i(x,\omega))} = th
	\]
	\[
	\implies 1 \leq t_h(x,v) \leq 1 + \frac{t^3 h^2}{24a_\MM^2} \leq 1 + \frac{\pi^3 h^2}{192a_\MM^2}.
	\]
	In particular, using $\alpha \geq 6$
	\[
	\sup_{\substack{i \in I \\ y \in \varphi_i^{-1} \p{\UU_i} \\ 0 < h < r_\MM'/2 \\ \norm{v}_{T_{\varphi_i^{-1}(y)}\MM}=1}} \frac{\abs{\frac{\partial F}{\partial t}(\varphi_i^{-1}(y),t_h(\varphi_i^{-1}(y),v)) - 2t_h(\varphi_i^{-1}(y),v)}}{h^2} = \sup_{\substack{i \in I \\ y \in \varphi_i^{-1} \p{\UU_i} \\ 0 < h < r_\MM'/2 \\ \norm{v}_{T_{\varphi_i^{-1}(y)}\MM}=1}} \abs{\frac{\partial G}{\partial t} \p{y,t_h(\varphi_i^{-1}(y),v)}} \leq M < +\infty
	\]
	Since $2t_h(x,v) \geq 2$, there exists $0 < h_+ \leq r_\MM'/2$ such that for any $i \in I, h \leq h_+, x \in \UU_i, \norm{v}_{T_x\MM}=1$ we have $\frac{\partial F}{\partial t} \neq 0$. Hence the implicit function theorem implies that $y \mapsto t_h(\varphi_i(y),v)$ is smooth on $\varphi_i^{-1}\p{\UU_i}$ with differential given by
	\[
	\frac{\partial}{\partial y} t_h(\varphi_i(y),v) = - \frac{\frac{\partial F}{\partial y}}{\frac{\partial F}{\partial t}} = - \frac{h^2 \frac{\partial G}{\partial y}}{2t_h(\varphi_i(y),v) + h^2\frac{\partial G}{\partial t}} \in \LL \p{\RR^d, \RR}
	\]
	To conclude notice first that since $t_h(x,v) \in [1,1+\frac{\pi^3 h^2}{192a_\MM^2}]$ we have, setting $s_h(x,v) = \frac{t_h(x,v)-1}{h^2}$
	\[
	\sup_{\substack{x \in \MM \\ 0 < h < r_\MM'/2 \\ \norm{v}_{T_x\MM} = 1}} \abs{s_h(x,v)} \leq \frac{\pi^3}{192a_\MM^2}
	\]
	Moreover for the derivatives we get, for any $y \in\varphi_i^{-1}\p{\UU_i}$
	\begin{align*}
		\frac{\partial}{\partial y} s_h(\varphi_i(y),v) &=  h^{-2} \frac{\partial}{\partial y} t_h(\varphi_i(y),v)
		=  - h^{-2} \frac{h^2 \frac{\partial G}{\partial y}}{2t_h(\varphi_i(y),v) + h^2 \frac{\partial G}{\partial t}} \\
		&=  - \frac{\frac{\partial G}{\partial y}}{2t_h(\varphi_i(y),v) + h^2 \frac{\partial G}{\partial t}} 
		 = - \frac{\frac{\partial G}{\partial y}}{2 + 2h^2 s_h(\varphi_i(y),v) + h^2 \frac{\partial G}{\partial t}}.
\end{align*}
Hence for any $y \in \varphi_i^{-1}\p{\UU_i}, \norm{v}_{T_{\varphi_i(y)}\MM} = 1, 0 < h < h_+' = h_+ \wedge M^{-1/2}$
\[
\forall 1 \leq j \leq d, \abs{\frac{\partial}{\partial y_j} s_h(\varphi_i(y),v)} = \frac{\abs{\frac{\partial G}{\partial y_j}}}{\abs{2t_h(\varphi_i(y),v) + h^2 \frac{\partial G}{\partial t}}} \leq \frac{M}{2 - h^2 M} \leq M
\]
In the same way by induction we can prove that
\[
\max_{\alpha \in \mathbb{N}^d, \abs{\alpha} \leq \alpha - 5} \sup_{\substack{i \in I \\ y \in \varphi_i^{-1}\p{\UU_i} \\ 0 < h < h_+' \\ \norm{v}_{T_{\varphi_i(y)}\MM} = 1}} \abs{\frac{\partial^{\abs{\alpha}}}{\partial y^\alpha} s_h(\varphi_i(y),v)} < +\infty
\]
\end{proof}
	
\begin{remark}
	We have stated our lemma with $r_\MM'$ such that $\norm{\exp_x(rv)-x}^2 \geq \frac{r^2}{4}$ for any $r \in [0,r_\MM']$, but more generally for any $c \in (0,1)$ we could have found $r_{\MM,c}'>0$ such that $r \in [0,r_{\MM,c}']$ implies $\norm{\exp_x(rv)-x}^2 \geq c^2 r^2$.
\end{remark}

\section{Geometrical and analytical properties of the random graph}\label{appendix:geometry_of_the_random_graph}

In this section we prove/recall  useful facts on the graph coming from geometrical properties of the manifold. 

We start by a simple but useful property satisfied by the graph Laplacian eigenvectors :

\begin{proposition}\label{proposition:uniform_norm_u_j}
	For any $h > 0, j \in \b{1,\ldots,N}$ we have $\norm{u_j}_{L^\infty \p{\nu}} \leq N$.
\end{proposition}\begin{proof}
	For all $x \in V$ we have $\nu_x = \frac{\mu_x}{\sum_{y \in V} \mu_y} \geq \frac{1}{N^2}$, which implies 
	\[	
	\forall x \in V, u_j(x)^2 \leq \frac{1}{\nu_x} \sum_{y \in V} u_j(y)^2 \nu_y = \frac{1}{\nu_x} \norm{u_j}_{L^2\p{\nu}}^2 = \frac{1}{\nu_x} \leq N^2
	\]
\end{proof}

As we will see below, we actually have $\nu_x \asymp N^{-1}$ on a high probability event, so that the result of proposition \ref{proposition:uniform_norm_u_j} could actually be improved to $\max_{1 \leq j \leq N} \norm{u_j}_{L^\infty \p{\nu}} \lesssim \sqrt{N}$, but this would not improve the final rates anyway.

In the next subsections we will be discussing various properties of the random geometric graph, namely volume regularity, on diagonal heat kernel bounds, Weyl type upper and lower bounds on the graph Laplacian eigenvalues and finally a norm comparison theorem. These properties will be used in the proof of Theorems \ref{theorem:rate_prior1_d_n},\ref{theorem:rate_prior1_d_N},\ref{theorem:rate_prior2_d_n} \& \ref{theorem:rate_prior2_d_N}. More precisely, the proof of the Kullback-Leibler prior mass condition (Lemma \ref{lem:prior_thickness_prior1}, which is then subsequently used in the proof of Theorems \ref{theorem:rate_prior1_d_n} \& \ref{theorem:rate_prior2_d_n}) uses a lower bound on the eigenvalues of the graph Laplacian $\LL^{(h_n)}$ for a determined $h_n$ (the one appearing in the statements of Theorems \ref{theorem:rate_prior1_d_n} \& \ref{theorem:rate_prior2_d_n}). To prove contraction rates with respect to $\|\cdot\|_N$ (Theorems \ref{theorem:rate_prior1_d_N} \& \ref{theorem:rate_prior2_d_N}), we combine the results for $\|\cdot\|_n$ together with Lemma \ref{thm:norm_comparison}, which states an inequality of the form : with $\Pro_0-$probability going to $1$
\begin{equation}\label{eqn1}
\forall h_- \leq h \leq h_0, \quad J_0 \leq J \leq J_1 h^{-d}, \quad f \in \text{span}\b{u_1,\ldots,u_J}, \quad \|f\|_{L^\infty \p{\nu}} \lesssim \sqrt{J} \|f\|_{L^2\p{\nu}}
\end{equation}
for some constants $h_0,J_0,J_1 > 0$, up to logarithmic factors, as long as $h_-$ satisfies assumption \ref{assumption:h}. To prove theorem \ref{theorem:rate_prior2_d_N} (adaptive posterior contraction rates with respect to $\|\cdot\|_N$) we use inequality \ref{eqn1} for any $h$ in the support of the posterior distribution. Thus in this section we analyse the properties of the random graph associated to the connectivity parameter $h$ for $h$ belonging to $\c{h_-,h_0}$. Note that these properties will also be used in the proof of the other results but for a specific value $h_n$ in $\c{h_-,h_0}$.

\subsection{Volume regularity}

In this subsection we establish a volume regularity property for the random geometric graph : roughly speaking, with high probability, for every suitable $r,h$ and geodesic ball of radius $r$ we have $\nu^{(h)} \p{B} \asymp r^d$. To start with notice that since $\MM$ is compact it has a positive radius of curvature $a_\MM$ and in particular by lemma 3 in \cite{bernstein_graph_nodate} for any $x,y \in \MM$ with $\rho(x,y) \leq \pi a_\MM$ we find 
\begin{equation}\label{eq:rho:norm}
	\frac{2}{\pi}\rho(x,y) \leq 2a_\MM \sin \p{\rho(x,y)/2ra_\MM} \leq \norm{x-y} \leq \rho(x,y).
\end{equation}
and, using $\sin \p{\rho/2a_\MM} \geq \frac{\rho}{2ra_\MM} - \frac{1}{6} \p{\frac{\rho}{2a_\MM}}^3$
\[
\rho(x,y) - \frac{\rho(x,y)^3}{24a_\MM^2} \leq \norm{x-y} \leq \rho(x,y)
\]
Moreover the set $\MM^2 \backslash \b{\p{x,y} \in \MM^2 : \rho(x,y) < \pi a_\MM}$ is compact with respect to the topology inherited from the Euclidean distance or equivalently the geodesic distance. Hence the function $\p{x,y} \mapsto \frac{\rho(x,y)}{\norm{x-y}}$ is bounded on $\MM^2 \backslash \b{\p{x,y} \in \MM^2 : \rho(x,y) < \pi a_\MM}$ by some constant $C_0$, while it satisfies $\frac{\rho(x,y)}{\norm{x-y}} \leq \frac{\pi}{2}$ when $\rho(x,y) \leq \pi a_\MM$. Combining the two cases and assuming without loss of generality $C_0 \geq \frac{\pi}{2}$ we find the existence of $C_0 > 0$ such that $\norm{x-y} \leq \rho(x,y) \leq C_0 \norm{x-y}$ on $\MM$.

\begin{theorem}\label{thm:volume_regularity}
	Let $h_- $ satisfy  Assumption \ref{assumption:h}, i.e. $ Nh_-^d >> \ln N$,  and  $A_N$ be the event
	\[
	A_N = \b{\forall i = 1,\ldots,N, \forall h_- \leq r \leq diam_\rho\p{\MM}/C_0, c_-r^d \leq \frac{\mu_{X_i}^{(r)}}{N} \leq c_+ r^d}, \quad \mu_{X_i}^{(r)} = \mu( B(X_i, r) )
	\]
	Then, for some $c_-, c_+, c > 0$ we have $\Pro \p{A_N} \geq 1 - e^{-cNh_-^d}$. Moreover, on $A_N$ we have $\frac{c_-/c_+}{N} \leq \nu_y^{(h)} \leq \frac{c_+/c_-}{N}$ for any $h_- \leq h \leq diam_\rho \p{\MM}$ and $y \in V$.
\end{theorem}

. A consequence of Theorem \ref{thm:volume_regularity} is that on $A_N$ for all $r \geq h_-$, $\nu^{(h)}(B(X_i, r)) \geq r^d c_-^2/c_+$ for all $i \leq N$. 

\begin{proof}
	Using
	\[
	\forall x \in \MM, r > 0, B_\rho(x,r) \subset B_{\RR^D}\p{x,r} \subset B_\rho \p{x, C_0 r}
	\]
	theorem 3.8 in \cite{gobel_volume_2020} implies that for some $b_-,b_+,c>0$, for all $h_- \leq r \leq diam_\rho\p{\MM}/C_0, diam_\rho\p{\MM} := \max_{x,y \in \MM} \rho(x,y), i = 1,\ldots,N$, with probability at least $1-e^{-cNh_-^d}$
	\[
	b_- r^d \leq \frac{\mu_{X_i}^{(r)}-1}{N-1} \leq b_+ r^d
	\]
	Using $\frac{Nh_-^d}{\ln N} \to \infty$ and a union bound we find that the last inequality holds simultaneously for all $i \in \b{1,\ldots,N}$ with probability at least $1 - e^{-cNh_-^d}$ for some $c > 0$. To make the result uniform in $r$, define $r_k = 2^k h_-, k_+ = \floor{\frac{\ln \frac{diam\p{\MM}}{h_-}}{\ln 2}} \leq c_\MM^{-1} \ln N$ (without loss of generality for small $c_\MM$, because $h_- >> \p{\frac{\ln N}{N}}^{1/d}$). Using a union bound we find that with probability $1-\p{k_+ +1} e^{-c_\MM Nh_-^d} \geq 1 - c_\MM^{-1} \ln N e^{-cNh_-^d}$, for every $i=1,\ldots,N$ and $k=0,\ldots,k_+ +1$
	\[
	b_- r_k^d \leq \frac{\mu_{X_i}^{(r_k)}-1}{N-1} \leq b_+ r_k^d
	\]
	But then, for every $k = 0,\ldots,k_+ + 1, r_k \leq r \leq r_{k+1}$ and $i \in \b{1,\ldots,N}$ we have (assuming $N \geq 2$)
	\begin{align*}
		\frac{N-1}{N} \frac{\mu_{X_i}^{(r)}-1}{N-1} \leq \frac{\mu_{X_i}^{(r)}}{N} = & \frac{1}{N} + \frac{\mu_{X_i}^{(r)}-1}{N-1} \frac{N-1}{N} \\
		\implies \frac{N-1}{N} \frac{\mu_{X_i}^{(r_k)}-1}{N-1} \leq \frac{\mu_{X_i}^{(r)}}{N} \leq & \frac{1}{N} + \frac{\mu_{X_i}^{(r_{k+1})}-1}{N-1} \frac{N-1}{N} \\
		\implies \frac{1}{2} \frac{\mu_{X_i}^{(r_k)}-1}{N-1} \leq \frac{\mu_{X_i}^{(r)}}{N} \leq & \frac{1}{N} + \frac{\mu_{X_i}^{(r_{k+1})}-1}{N-1} \\
		\implies \frac{1}{4} b_- r_{k+1}^d = \frac{1}{2} b_- r_k^d \leq \frac{\mu_{X_i}^{(r)}}{N} \leq & \frac{1}{N} + b_+ r_{k+1}^d = \frac{1}{N} + 2b_+ r_k^d \\
		\implies \frac{1}{4} b_- r^d \leq \frac{\mu_{X_i}^{(r)}}{N} \leq & h^d + 2b_+ r^d  \quad
		\implies \frac{1}{4} b_- r^d \leq \frac{\mu_{X_i}^{(r)}}{N} \leq  \p{2b_+ + 1} r^d.
	\end{align*}
	Therefore, using again $\frac{Nh_-^d}{\ln N} \to \infty$ with probability $1-e^{-c Nh_-^d}$, for every $i = 1,\ldots,N$ and $h_- \leq r \leq diam_\rho \p{\MM}$
	\[
	\frac{1}{4}b_- r^d \leq \frac{\mu_{X_i}^{(r)}}{N} \leq \p{2b_+ + 3} r^d
	\]
	which proves the first statement with $c_- = \frac{1}{4}b_-, c_+ = 2b_+ + 3$. The second statement is an immediate consequence of the first one. Indeed, using again
	\[
	\forall x \in \MM, r > 0, B_\rho(x,r) \subset B_{\RR^D}\p{x,r} \subset B_\rho \p{x, C_0 r}
	\]
	And
	\[
	\forall y \in V, h_- \leq h \leq diam_\rho\p{\MM}/C_0, \frac{c_-/c_+}{N} = \frac{c_- Nr^d}{c_+ N^2 r^d} \leq \nu_y^{(h)} = \frac{\mu_y^{(h)}}{\sum_{x \in V} \mu_x^{(h)}} \leq \frac{c_+ Nr^d}{c_- N^2 r^d} = \frac{c_+/c_-}{N}
	\]
\end{proof}

In what follows we take $h_0 \leq diam_\rho\p{\MM}, r_0 \leq diam_\rho\p{\MM}/C_0$ (that we will actually reduce along the proof) and we fix $h_-$ satisfying assumption \ref{assumption:h} in order to work on the corresponding $A_N$.

\subsection{Heat kernel bounds}

The goal of this section is to show the following result : we consider $A_N$  and $h_-$ as defined in Theorem \ref{thm:volume_regularity}. 

\begin{theo}\label{thm:hkb}
There exist constants $a_0,a_1,a_2,a_3,h_0 > 0$ (that also depend on $\MM,p_0$) such that, on $A_N$, for all $h_- \leq h \leq h_0$ and $t_0 (h) := a_0 h^2 \ln \p{Nh^d} \leq t \leq \frac{a_1}{\ln^2 N}$ 
\[
\forall x \in V, \quad  a_2 \frac{t^{-d/2}}{\ln^d N} \leq p_t^{(h)}(x,x) \leq a_3 t^{-d/2}.
\] 
\end{theo}

This result is a very weak form (it is often called an "on diagonal" upper bound) of the heat kernel bound used in $\cite{castillo_thomas_2014,coulhon_heat_2012}$ over continuous spaces, but fortunately will be enough for our purposes. It is possible to apply further techniques in \cite{Barlow_2017} and get full off diagonal bounds (i.e, bounds on $p_t^{(h)}(x,y), x \neq y$) yielding a Gaussian type behaviour of the heat kernel, but we won't need this. It should be noted that the restriction $t \gtrsim h^2$ (up to a logarithmic factor) for the heat kernel bound to be valid substantially complicates the analysis of the situation. Indeed, if instead we could prove the bound for every $0 < t \leq \frac{a_1}{\ln^2 N}$ then combining the techniques from \cite{coulhon_sikora_gaussian_hkb,coulhon_heat_2012}, our approximation results with the proof techniques of \cite{castillo_thomas_2014}. Since we require $t \gtrsim h^2 \ln \p{Nh^d}$ with $h$ not arbitrarily small (otherwise our approximation results \ref{thm:approximation}\ref{thm:approximation2} become vacuous) this is not possible and we have to prove things differently.

As in section \ref{section:model_and_notations}, we emphasize the dependence of the graph (and in particular the graph Laplacian, the heat kernel and its eigendecomposition) with respect to $h$ with a $(h)$ exponent. While the notation become heavier, this is important to keep in mind as theorem \ref{thm:hkb} deals with a high probability control of the heat kernels associated with different values of $h$ \textit{simultaneously}.

\subsubsection{On diagonal upper bound up to $t_1$}

\begin{lem}\label{lem:on_diag_bound1}
	With $A_N, c_-, c_+$ the event and constant defined in Theorem \ref{thm:volume_regularity} we have, on $A_N$
	\[
	\forall h_- \leq h \leq diam_\rho\p{\MM}, t \geq 16 h^2 \ln \p{2c_- Nh^d}, x,y \in V, p_t^{(h)}(x,y) \leq 2\frac{c_+}{c_-} h^{-d}
	\]
\end{lem}
\begin{proof}
	Recall that for all $x \in V, t > 0, \nu_y p_t^{(h)}(x,y) = \Pro_x^{(h)} \p{W_t=y}$. On the event $A_N$ we have $\nu_y^{(h)} \geq \frac{c_-}{c_+N}$ for any $h_- \leq h \leq h_0$, which implies
	\[
	\forall x,y \in V , \quad p_t^{(h)}(x,y) \leq \frac{c_+}{c_-} N \Pro_x^{(h)} \p{W_t = y} = \frac{c_+}{c_-}N \sum_{l \geq 0} \Pro \p{N_{t/h^2} = l} \Pro_x^{(h)} \p{Y_l = y}
	\]
	where the last equality uses the independence of $N$ and $Y$. Moreover, using lemma 5.13 (e) from \cite{Barlow_2017} we find
	\[
	\forall \rho \geq 2, \Pro \p{N_\rho \leq \frac{\rho}{2}} \leq \Pro \p{\abs{N_\rho - \rho} \geq \frac{\rho}{2}} \leq 2e^{-\rho/16}
	\]
	which implies
	\begin{align*}
		\forall x,y \in V,\quad  p_t^{(h)}(x,y) \leq \frac{c_+}{c_-} N \Pro_x^{(h)} \p{W_t = y} \leq & \frac{c_+}{c_-} N \b{2e^{-t/16h^2} + \sum_{l \geq t/2h^2} \Pro \p{N_{t/h^2} = l} \Pro_x^{(h)} \p{Y_l = y} } \\
		\leq & \frac{c_+}{c_-} N \b{2e^{-t/16h^2} + \sup_{l \geq 1} \Pro_x^{(h)} \p{Y_l = y} }
	\end{align*}
	For each $l \geq 1$ we have $\Pro_x^{(h)}\p{Y_l = y} = \EE_x^{(h)} \c{\Pro_x^{(h)} \p{Y_l = y|Y_{l-1}}}$, and given $Y_{l-1}$ the probability of going from $Y_{l-1}$ to $Y_l = y$ is by definition $0$ (if $Y_{l-1}$ and $y$ are not neighbours) or $\frac{1}{\mu_x^{(h)}}$ (if $Y_{l-1}$ and $y$ are neighbours). In either case on $A_N$ we have $\mu_x^{(h)} \geq c_- Nh^d$ which implies
	\[
	\forall x,y \in V, p_t^{(h)}(x,y) \leq \frac{c_+}{c_-} N \b{2e^{-t/16h^2} + \frac{1}{c_- Nh^d}}.
	\]
	All in all on $A_N$ we get
	\[
	\forall h_- \leq h \leq h_0,\,\,   t \geq 16 h^2 \ln \p{2c_- Nh^d}, \,\,  x,y \in V, \quad p_t^{(h)}(x,y) \leq 2\frac{c_+}{c_-} h^{-d}.
	\]
\end{proof}

The goal is now, roughly speaking, to refine the inequality $p_t^{(h)}(x,y) \lesssim h^{-d}$ to $p_t^{(h)}(x,y) \lesssim t^{-d/2}$ (up to a logarithmic factor) by using the approach of \cite{Barlow_2017}. This requires a local Poincare inequality, i.e an inequality of the form: 
\begin{align*}
\forall x \in V,& \, \forall  r \in [r_-,r_+], \, B = B_\rho\p{x,r},\, \forall   f : B \to \RR, \\ 
\sum_{y \in B} \p{f(y) - f_B} \nu_y^{(h)} &\leq C r^2 \sum_{x,y \in B, y \sim x} \p{f(x) - f(y)}^2, \quad \text{where } \, f_B = \frac{\sum_{y \in B} f(y) \nu_y^{(h)}/}{\nu^{(h)}(B)}.
\end{align*}
where the constants $r_-,r_+,C$ need to be controlled. G\"obel \& Blanchard \cite{gobel_volume_2020} provide a way to prove such an inequality using the construction of random Hamming paths, however we give detailed application of their results to our setting in order to control the dependence of the constant $C$ of the inequality in the parameters $N,h$.

\begin{theo}{Poincare inequality}\label{thm:poincare}\\
	Let $A_N' = \b{G^{(h_-)} \text{ is connected}}$. Then there exist constants $r_0 = r_0\p{\MM}, h_0 = h_0\p{\MM}, C = C\p{\MM} > 0$ such that on $A_N $, for any $x \in V, 0 < r < r_0, h_- \leq h < h_0$ and $f : B = B_\rho\p{x,r} \to \RR$, on $A_N \cap A_N'$ we have
	\[
	\sum_{y \in B} \p{f(y) - f_B} \nu_y^{(h)} \leq C\frac{\p{r/h}^2}{N^2h^d} \sum_{z \in B, y \in B, y \sim x} \p{f(z) - f(y)}^2
	\]
	where $f_B = \sum_{y \in B} f(y) \nu_y^{(h)}/\nu^{(h)}(B)$. Moreover, $\Pro_0 \p{A_N \cap A_N'} \geq 1 - e^{-cNh_-^d}$ for some $c>0$.
\end{theo}
\begin{proof}
	Let $h_->0$ and $A_N, A_N'$ defined accordingly. The event $A_N \cap A_N'$ satisfies $\Pro_0 \p{A_N \cap A_N'} \geq 1 - e^{-cNh_-^d}$ for some $c>0$, see remark \ref{remark:connectedness} below.
	For any $h \geq h_-, x \in V, B = B_\rho \p{x,r}$, corollary 5.6 in \cite{gobel_volume_2020} shows that 
	\[
	\forall f : B \to \RR, \sum_{y \in B} \p{f(y) - f_B} \nu_y^{(h)} \leq \tilde{\kappa}_B r^2 \sum_{z \in B, y \in B, y \sim x} \p{f(z) - f(y)}^2
	\]
	where
	\[
	\tilde{\kappa}_B = \frac{\max_{z \in B} \nu^{(h)}(y)^2}{2\nu^{(h)}\p{B}} l_{\max}\p{\tilde{\Gamma}_B} b_{\max}\p{\tilde{\Gamma}_B}
	\]
	and $l_{\max}\p{\tilde{\Gamma}_B}, b_{\max}\p{\tilde{\Gamma}_B}$ are quantities defined in  \cite{gobel_volume_2020}. Notice that \cite{gobel_volume_2020} requires connectedness of the subgraph restricted to $B$. This is shown in remark \ref{remark:connectedness} below. 
	
	By following the proof of corollary 5.13 in \cite{gobel_volume_2020} we find the existence of $C = C\p{\MM, p_0}, r_0 = r_0\p{\MM}, h_0 = h_0\p{\MM} >0$ such that on $A_N$, for any $0 < r < r_0, x \in V, B = B_\rho\p{x,r}$
	\[
	l_{\max}\p{\tilde{\Gamma}_B} \leq C h^{-1}, b_{\max}\p{\tilde{\Gamma}_B} \leq C h^{-(d+1)}
	\]
	(in \cite{gobel_volume_2020} the result is proved with high probability over a single ball, but an inspection of the proof shows that the only thing needed is a control of the number of points in the balls on the graph for different values of $r$ and $h$, which is precisely what $A_N$ is for as shown in lemma \ref{thm:volume_regularity}, hence the uniformity).
	This shows the desired inequality
	\[
	\sum_{y \in B} \p{f(y) - f_B} \nu_y^{(h)} \lesssim \frac{\p{r/h}^2}{N^2h^d} \sum_{z \in B, y \in B, x \sim y} \p{f(z) - f(y)}^2
	\]
\end{proof}

\begin{remark}\label{remark:connectedness}
	We show the following simple result : if $h_-$ satisfies assumption \ref{assumption:h} then the probability that the resulting random geometric graph $G^{(h_-)}$ is connected (i.e the probability of the event $A_N'$) converges to $1$. In particular, by monotonicity the probability that all random geometric graphs with connectivity parameter $h \geq h_-$ are connected converges to $1$ as well.
	
	Indeed, since $\MM$ is connected, consider $\b{y_1,\ldots,y_p}$ an $h/8-$net of $\MM$ with respect to the Euclidean distance $\norm{\cdot}$. By standard arguments and compactness of $\MM$ we can assume that $p = \mathcal{O}(h^{-d})$. Then the event 
	\[
	A_n' = \b{\forall l, \,\,  \exists i_l, \, \norm{y_l - x_{i_l}} < h/8}
	\]
	satisfies 
	\[
	\Pro \p{(A_n')^c}  \lesssim p \p{1-ch^d}^N =  \mathcal{O}\p{\frac{1}{h^d} e^{-c8^{-d}Nh^d}} \lesssim e^{-c'Nh^d}
	\]
	 for some $c'>0$ by assumption \ref{assumption:h}. Here $c>0$ is small enough such that $\int_{B_{\RR^D}\p{x,h}} p_0(x) \mu(dx) \geq ch^d$ for all $ x \in \MM$ ($c$ exists since $p_0\geq p_{\text{min}} >0$).
	
	Moreover by connectedness of $\MM$ (which implies path connectedness since $\MM$ is a submanifold of $\RR^D$) for each $i,j \in \b{1,\ldots,N}$ there exists a continuous path $c : [0,1] \to \MM$ with $c(0) = x_i, c(1) = x_j$. Hence for each $t \in [0,1]$ there exists $l(t) \in \b{1,\ldots,p}$ such that $\norm{c(t)-y_{l(t)}} < h/8$, and therefore on $A_n'$ for each $t \in [0,1]$ there exists $i_t = i_{l(t)}$ such that
	\[
	\norm{c(t) - x_{i_t}} \leq \norm{c(t) - y_{l(t)}} + \norm{y_{l(t)} - x_{i_{l(t)}}} < h/4
	\]
	By continuity of $c$, this implies the existence of $K \geq 1$ and $1 \leq i_k \leq N$ for $k = 1,\ldots,K$ such that $\norm{x_{i_k} - x_{i_{k+1}}} \leq h/2$ for any $1 \leq k < K$ and such that $c(0) = x_i \sim x_{i_1}, c(1) = x_j = x_{i_K}$. Therefore the path $x_i, x_{i_1},\ldots,x_{i_{K-1}},x_j$ connects $x_i$ and $x_j$ in the graph and therefore the graph is connected on the event $A_n'$.
	
	Notice that this is a very basic and coarse result about connectivity and more generally connected components of random geometric graphs. More details and results can be found in \cite{penrose_rggs}.
\end{remark}

With the Poincare inequality \ref{thm:poincare} we can now apply the proof strategy of \cite{Barlow_2017}. We first prove an  upper bound of the form
\[
\forall x,y \in V, p_t^{(h)}(x,y) \lesssim t^{-d/2}
\]
for $t \gtrsim h^2 \ln \p{Nh^d}$ (up to a logarithmic factor) but less than a random time $t_1 = t_1(x,h)$ that we shall lower bound adequately later.

\begin{lem}\label{lem:on_diag_bound_up_to_t1}
	There exists $h_0 = h_{0|\MM,p_0}, K = K_{\MM,p_0}, K' = K'_{\MM,p_0} > 0$ such that, on $A_N$, for every $h_- \leq h \leq h_0, x_0 \in V$, with the random times
	\[
	t_1(x,h) = \inf \b{t > 0 : p_{2t}^{(h)}(x,x) \leq K'}
	\]
	and
	\[
	t_1(h) = \min_{x \in V} t_1(x,h),
	\]
	we have, with $t_0(h) = 64h^2 \ln \p{2c_-Nh^d}$,
	\[
	\forall h_- \leq h < h_0, t \in [t_0(h),2t_1(h)], x,y \in V, p_t^{(h)}(x,y) \leq K t^{-d/2}.
	\]
\end{lem}
\begin{proof}
	First consider the case $x=y$. Take $x_0 \in V$ and let $\phi(t) = p_{2t}^{(h)}(x_0,x_0) = \sum_{x \in V} f_t (x)^2 \nu_x^{(h)}, f_t(x) = p_t^{(h)}(x_0,x)$. Then since $\nu^{(h)}_y \mathcal L_{yx} = \nu^{(h)}_x \mathcal L_{xy} $, $\frac{d}{dt} f_t = - \LL f_t$ we have
	\[
	- \phi'(t) = 2 \inner{f_t|\LL f_t}_{L^2(\nu)} = \frac{1}{\mu \p{V} h^2} \sum_{x \sim y} \p{f_t(x) - f_t(y)}^2 \geq \frac{1}{c_+ N^2 h^{d+2}} \sum_{x \sim y} \p{f_t(x) - f_t(y)}^2,
	\]
	where we have used theorem \ref{thm:volume_regularity} to bound $\mu(V) \leq c_+ N^2 h^d$. This  proves that $\phi$ is non-increasing. Taking a covering of $V$ by balls $B_i = B_{\rho}\p{x_i,r}, h_- \leq r \leq diam_\rho\p{\MM}/C_0$ given by lemma \ref{lem:covering} below, since each $x_i$ belongs to at most $M = 3^d \p{c_+/c_-}^3$ of the balls we have 
	\[
	- \phi'(t) \geq  \frac{1}{c_+ N^2 h^{d+2}} \sum_{x \sim y} \p{f_t \p{x} - f_t \p{y}}^2 \geq \frac{1}{c_+ M N^2 h^{d+2}} \sum_i \sum_{x \sim y \in B_i} \p{f_t \p{x} - f_t \p{y}}^2.
	\]
	On $A_N$, if $h_- < r < r_0$ , $  h_- \leq h<h_0$ and using Theorem \ref{thm:poincare} in each of the balls we find
	\[
	- \phi'(t) \geq \frac{r^{-2}}{MC c_+} \sum_i \sum_{x \in B_i} \p{f_t(x) - f_{t|B_i}}^2 \nu_x^{(h)}.
	\]
	Defining $\nu_{B_i}^{(h)} = \frac{\nu_{|B_i}^{(h)}}{\nu^{(h)} \p{B_i}}$ the normalised restriction of $\nu^{(h)}$ to $B_i$ and $Var\p{X} = EX^2 - \p{EX}^2$ we get
	\begin{align*}
		- \phi'(t) \geq & \frac{r^{-2}}{MC c_+} \sum_i \sum_{x \in B_i} \p{f_t(x) - f_{t|B_i}}^2 \nu_x^{(h)} 
		=  \frac{r^{-2}}{MC c_+} \sum_i \nu^{(h)} \p{B_i}\sum_{x \in B_i} \p{f_t(x) - f_{t|B_i}}^2 \nu_{B_i}^{(h)}(x) \\
		= & \frac{r^{-2}}{MC c_+} \sum_i \nu^{(h)} (B_i)\b{ \sum_{x \in B_i} f_t(x)^2 \nu_{B_i}^{(h)}(x) - \p{\sum_{x \in B_i} f_t(x) \nu_{B_i}^{(h)}(x)}^2 } \\
		= & \frac{r^{-2}}{MC c_+} \sum_i \b{ \sum_{x \in B_i} f_t(x)^2 \nu_x - \nu^{(h)} \p{B_i} \p{\sum_{x \in B_i} f_t(x) \nu_{B_i}^{(h)}(x)}^2 }.
	\end{align*}
	But the balls form a covering of $V$, therefore
	\[
	- \phi'(t) \geq \frac{r^{-2}}{MC c_+} \b{\sum_{x \in V} f_t(x)^2 \nu_x - \sum_i \nu^{(h)} \p{B_i} \p{\sum_{x \in B_i} f_t(x) \nu_{B_i}^{(h)}(x)}^2 }.
	\]
	We have
	\[
	\sum_{x \in V} f_t(x)^2 \nu_x^{(h)} = \phi(t),
	\]
	and, still on $A_N$
	\begin{align*}
		\sum_i \nu^{(h)} \p{B_i} &\p{\sum_{x_j \in B_i} f_t (x_j) \nu_{B_i}^{(h)}(x_j)}^2 =  \sum_i \nu^{(h)}\p{B_i}^{-1} \p{\sum_{x_j \in B_i} f_t (x_j) \nu_{x_j}^{(h)}}^2 \\
		\leq & \sum_i \frac{c_+}{c_-^2} r^{-d} \p{\sum_{x_j \in B_i} f_t (x_j) \nu_{x_j}^{(h)}}^2 
		\leq \frac{c_+}{c_-^2} r^{-d} \p{\sum_i \sum_{x_j \in B_i} f_t (x_j) \nu_{x_j}^{(h)}}^2 
		=  \frac{c_+}{c_-^2} r^{-d} \p{ \sum_{x_j \in V} f_t (x_j)\nu_{x_j}^{(h)} \sum_{i : x_j \in B_i} 1 }^2 \\
		\leq & \frac{c_+}{c_-^2} r^{-d} \p{ \sum_{x_j \in V} f_t (x_j)\nu_{x_j}^{(h)} M }^2   =   \frac{ M^2 c_+}{c_-^2} r^{-d}
	\end{align*}
	where  $h_- \leq r \leq r_0, h_- \leq h \leq h_0$ and we have used Theorem \ref{thm:volume_regularity} to bound $\nu^{(h)}\p{B_i} \geq r^d c_-^2/c_+$. Hence
	\[
	- \phi'(t) \geq \frac{r^{-2}}{MC c_+} \b{ \phi(t) - \frac{c_+ M^2}{c_-^2} r^{-d} },
	\]
	as long as $h_- \leq r \leq r_0, h_- \leq h \leq h_0$. Let $r = r(t) = \p{\frac{\phi(t)}{K}}^{-1/d}$, $K = 2c_+M^2 /c_-^2 \vee 2c_+/c_-$.
Define $t_1(x,h)$ the random time
	\[
	 t_1(x,h) = \inf \b{t \geq 0 : r(t) \geq r_0} = \inf \b{t \geq 0 : \phi(t) \leq Kr_0^{-d}}, \quad t_1(h) = \inf_x  t_1(x,h); 
	\]
	now, using Lemma \ref{lem:on_diag_bound1} and since $K \geq 2c_+/c_-$
	\[
	\forall t \geq t_0(h)/2, \quad \phi(t) \leq  \frac{ 2c_+ }{ c_-} h^{-d} \leq Kh^{-d} \leq Kh_-^{-d} ,
	\]
	which in turns implies that $ r(t) \geq h_-$.  Moreover by definition of $t_1(h)$ we have
	\[
	\forall t \leq t_1, \quad \phi(t) \geq Kr_0^{-d} \quad  \text{so that } \quad  r(t) \leq r_0.
	\]
	Therefore $h_- \leq r(t) \leq r_0$ whenever $t_0(h)/2 \leq t \leq t_1(h)$.	
 
	Then on $A_N$ we obtain 
	\[\forall t \in [t_0(h)/2,t_1(h)], \quad 
	- \phi'(t) \geq \frac{\p{\phi(t)/K}^{2/d}}{MC c_+} \frac{1}{2} \phi(t) = \frac{1}{2MCK^{2/d}c_+} \phi(t)^{1+2/d}.
	\]
	
	Integrating yields, for any $t_1(h) \geq t \geq t_0(h)/2$
	\[
	\c{\frac{d}{2}\phi^{-2/d}}_{t_0(h)/2}^t \geq \frac{1}{2MCK^{2/d} c_+^2} \p{t - t_0(h)/2} \geq \frac{t}{4MC K^{2/d} c_+^2}
	\]
	which implies in particular
	\[
	\frac{d}{2} \phi(t)^{-2/d} \geq \frac{d}{2} \phi \p{t_0(h)/2}^{-2/d} + \frac{t}{4MC K^{2/d} c_+^2} \geq \frac{t}{4MC K^{2/d} c_+^2}
	\]
	i.e
	\[
	\forall t \in [t_0(h)/2,t_1(h)], \phi(t) \leq \p{\frac{MCc_+^2K^{2/d}}{dt}}^{d/2}.
	\]
	To conclude we apply Cauchy-Schwarz inequality to find
	\[
	\forall x,y \in V, p_t(x,y) = \sum_{j=1}^N e^{-t\lambda_j} u_j(x)u_j(x) \leq \p{\sum_{j=1}^N e^{-t\lambda_j} u_j(x)^2}^{1/2} \p{\sum_{j=1}^N e^{-t\lambda_j} u_j(y)^2}^{1/2} = \sqrt{p_t(x,x)p_t(y,y)} \leq \max_{x \in V} p_t(x,x)
	\]
\end{proof}

\begin{lem}\label{lem:covering}
	For all $h_- \leq h \leq h_0, h \leq r \leq r_0$, on $A_N$ there exists a covering of $V$ by balls $B_{\rho}\p{x_i,r}$ such that the balls $B_{\rho}\p{x_i,r/2}$ are disjoints and each $x_i$ belongs to at most $3^d \p{c_+/c_-}^3$ balls $B_\rho \p{x_j,r}$.
\end{lem}
\begin{proof}
	Take a maximal packing $\p{B_{\rho}\p{x_i,r/2}}$. Then the balls $B_{\rho}\p{x_i,r/2}$ cover $V$ : indeed, if there exists $x_j \in V$ such that $\rho \p{x_i,x_j} \geq r$ for all $i$, then this contradicts the maximal packing property by considering the ball $B_{\rho}\p{x_j,r/2}$. Moreover if $y \in V$ belong to $M$ of the balls $B_{\rho}\p{x_i,r}$ (and we may assume without loss of generality that these balls are those corresponding to $i=1,\ldots,M$), using theorem \ref{thm:volume_regularity} we have
	\begin{align*}
		\frac{Mc_-^2}{c_+ N}\p{r/2}^d \leq & M \min_i \nu \p{B_\rho\p{x_i,r/2}}
		\leq  \sum_{i=1}^M \nu \p{B_{\rho}\p{x_i,r/2}} 
		=  \nu \p{\bigsqcup_{i=1}^M B_{\rho}\p{x_i,r/2}} \\
		\leq & \nu \p{\bigsqcup_{i=1}^M B_{\rho}\p{y,3/2r}} 
		=  \nu \p{B_{\rho}\p{y,3/2r}} 
		\leq \frac{c_+^2}{c_- N} \p{3/2r}^d,
	\end{align*}
	which implies the result.
\end{proof}

\subsubsection{Escape from the origin for the Markov process $W_t$}

In this section we find an upper bound on the quantities $m(t,x) = \EE_x^{(h)} \c{\rho \p{x,W_t}}$. Let $q(t,x) = \EE_x^{(h)} \c{\ln p_t^{(h)}(x,W_t)} = -\sum_{y \in V} p_t^{(h)}(x,y) \ln p_t^{(h)}(x,y) \nu_y^{(h)}$ as well as $M(s,x) = m(st_0,x)/\sqrt{t_0}, Q(s,x) = q(st_0,x), t_0(h)  = 64h^2 \ln \p{2c_- Nh^d}$.

\begin{lem} \label{lem:dm_dM}
For all $x\in V$, 
\begin{enumerate}
	\item 
	\[
	\forall t > 0, \quad q'(t,x) \geq \frac{1}{4C_0^2} m'(t,x)^2 , \quad \text{ and } \quad \forall s \geq 1, \, \, Q'(s,x) \geq \frac{1}{4C_0^2} M'(s)^2
	\]
	\item On $A_N$ we have 
	\[
	\forall 1 \leq s \leq t_1(h)/t_0(h), \quad Q(s,x) \geq - \p{\ln K - \frac{d}{2} \ln \p{st_0(h) }} = \frac{d}{2} \ln s + \frac{d}{2} \ln t_0(h) - \ln K
	\]
	\item
	\[
	1+M(s,x) \geq e^{- \frac{e+1}{ed}} e^{Q(s)/d}
	\]
\end{enumerate}
\end{lem}
\begin{proof}
	\begin{enumerate}
		\item Let $f_t(y) = p_t^{(h)}(x,y), b_t(y,z) = f_t(y) + f_t(z)$. Throughout the proof we write $m(t), M(t), q(t), Q(t)$ in place of $m(t,x), M(t,x), q(t,x) , Q(t,x)$ and $t_0= t_0(h)$ for shortness sake. We have
		\begin{align*}
			|m'(t)| = &\left| \frac{\partial}{\partial t} \sum_{y \in V} \rho \p{x,y} f_t(y) \nu_y^{(h)}\right|
			=  \left| \sum_{y \in V} \rho \p{x,y} \p{\LL f_t}(y) \nu_y^{(h)} \right| \\
			= & \left|  \inner{\rho \p{x,\cdot} | \LL f_t}_{L^2\p{\nu}} \right|
			=  \left| \frac{1}{2 h^2 \mu \p{V}} \sum_{y \sim z} \p{\rho(x,y) - \rho(x,z)} \p{f_t(y) - f_t(z)} \right| \\
			\leq & \frac{C_0}{h\mu \p{V}} \sum_{y \sim z} \abs{f_t(y) - f_t(z)} 
	= \frac{C_0}{h\mu \p{V}} \sum_{y \sim z} \frac{\abs{f_t(y) - f_t(z)}}{\abs{f_t(y) + f_t(z)}^{1/2}} \p{f_t(y) + f_t(z)}^{1/2} \\
			\leq & \frac{C_0}{h\mu \p{V}} \b{ \sum_{y \sim z} \frac{\p{f_t(y) - f_t(z)}^2}{f_t(y) + f_t(z)} }^{1/2} \b{\sum_{y \sim z} \p{f_t(y) + f_t(z)}}^{1/2} \\
			= & \frac{C_0}{h\mu \p{V}} \b{ \sum_{y \sim z} \frac{\p{f_t(y) - f_t(z)}^2}{f_t(y) + f_t(z)} }^{1/2} \b{2 \mu \p{V}}^{1/2} 
			=  \frac{C_0\sqrt{2}}{h\mu \p{V}^{1/2}} \b{ \sum_{y \sim z} \frac{\p{f_t(y) - f_t(z)}^2}{f_t(y) + f_t(z)} }^{1/2}.
		\end{align*}
		Now using the inequality (page 157 in \cite{Barlow_2017})
		\[
		\forall u,v > 0, \frac{(u-v)^2}{u+v} \leq (u-v)\p{\ln u - \ln v}
		\]
		We find
		\[
		m'(t)^2 \leq \frac{2C_0^2}{h^2\mu \p{V}} \sum_{y \sim z} \p{f_t(y) - f_t(z)} \p{\ln f_t(y) - \ln f_t(z)}
		\]
		On the other hand
		\begin{align*}
			q'(t) = & \frac{\partial}{\partial t} \sum_{y \in V} - f_t(y) \ln f_t(y) \nu_y^{(h)} \\
			= & - \sum_{y \in V} \b{\p{-\LL f_t}(y) \ln f_t(y) + f_t(y) \frac{\p{-\LL f_t}(y)}{f_t(y)}} \nu_y^{(h)} \\
			= & \sum_{y \in V} \p{\LL f_t}(y) \b{1 + \ln f_t(y)} \nu_y^{(h)} \\
			= & \frac{1}{2h^2 \mu\p{V}} \sum_{y \sim z} \p{f_t(y) - f_t(z)} \p{\ln f_t(y) - \ln f_t(z)}
		\end{align*}
		Hence
		\[
		q'(t) \geq \frac{1}{4C_0^2} m'(t)^2
		\]
		To conclude we use the fact that $Q'(s) = t_0q'(st_0), M'(s) = m'(st_0) \sqrt{t_0}$ which implies via $t = st_0$
		\[
		q'(t) \geq \frac{1}{4C_0^2} m'(t)^2 \quad \text{and } \quad Q'(s) \geq \frac{t_0}{4C_0^2} \c{ M'(s)/\sqrt{t_0} }^2 = \frac{1}{4C_0^2} M'(s)^2
		\]
		\item Using lemma \ref{lem:on_diag_bound_up_to_t1}, for any $h_- \leq h \leq h_0, x,y \in V$
		\[
		\forall 1 \leq s \leq t_1(h)/t_0, p_{st_0}(x,y) \leq K\p{st_0}^{-d/2}.
		\]
		Therefore
		\[
		\forall 1 \leq s \leq t_1(h)/t_0, Q(s) \geq - \p{\ln K - \frac{d}{2} \ln \p{st_0}} = \frac{d}{2} \ln s + \frac{d}{2} \ln t_0 - \ln K
		\]
		\item Let $a \leq 1, b \geq 0$. Then
		\begin{align*}
			-Q(s) + aM(s) + b = & \EE_x \c{ \ln p_{st_0}^{(h)}(x,W_{st_0}) + a\rho \p{x,W_{st_0}}/\sqrt{t_0} +b} \\
			= & \sum_{y \in V} \nu_y^{(h)} p_{st_0}^{(h)}(x,y) \b{ \ln p_{st_0}^{(h)}(x,y) + a\rho \p{x,y}/\sqrt{t_0} +b }
		\end{align*}
		Using the inequality $u \p{\ln u + \lambda} \geq -e^{-1-\lambda}$ and setting $a = \p{1+M(s)}^{-1}, b = d \ln \p{1+M(s)}$ we find
		\begin{align*}
			-Q(s) + aM(s) + b = & \sum_{y \in V} \nu_y^{(h)} p_{st_0}^{(h)}(x,y) \b{ \ln p_{st_0}^{(h)}(x,y) + a\rho \p{x,y}/\sqrt{t_0} +b } \\
			\geq & - \sum_{y \in V} \exp \b{-1 - a\rho \p{x,y}/\sqrt{t_0} - b } \nu_y^{(h)} \\
			= & -e^{-1-b} \sum_{y \in V} \exp \b{- a\rho \p{x,y}/\sqrt{t_0}} \nu_y^{(h)} \\
			\geq & -e^{-1-b} \sum_{y \in V} \nu_t^{(h)} 
			=  -e^{-1-b} 
			\geq  -e^{-1}.
		\end{align*}
		Hence
		\[
		-Q(s) + 1 + d \ln \p{1 + M(s)} \geq -Q(s) + \frac{M(s)}{1+M(s)} + d\ln \p{1+M(s)} \geq -e^{-1}
		\]
		and rearranging yields
		\[
		1+M(s) \geq e^{- \frac{e+1}{ed}} e^{Q(s)/d}
		\]
	\end{enumerate}

\end{proof}

\begin{lem}\label{lem:escape_rate}
There exists $c = c_{\MM,p_0,d}$ such that, on $A_N$ we have, for every $h_- \leq h \leq h_0$ and $t_0(h) \leq t \leq t_1(h)$
\[
m(t,x) \leq c \ln N \sqrt{t}, \quad \forall x \in V
\]
\end{lem}
\begin{proof}
We follow the proof of lemma 6.13 in \cite{Barlow_2017} (it is stated for $s \geq 1$, but the same proof remains valid for $1 \leq s \leq t_1(h)/t_0$). As done in the previous computations we drop $x$ in the notations $m(t,x), M(t,x), q(t,x), Q(t,x)$. From Lemma \ref{lem:dm_dM} we  bound $M'(s) \leq 2 C_0 \sqrt{Q'(s)}$ and defining  $r(s) = \frac{1}{d} \p{Q(s) + \ln K - \frac{d}{2} \ln \p{st_0(h)}} \geq 0$, we bound $\forall 1 \leq s \leq t_1(h)/t_0(h)$, 
\begin{align*}
e^{-\frac{e+1}{ed}}K^{-1/d}\sqrt{st_0(h)} e^{r(s)} \leq & 1 + M(s) \\
\leq & 1 + M(1) + 2C_0 \sqrt{d} \int_1^s \p{r'(\sigma) + \frac{1}{2\sigma}}^{1/2} d\sigma \\
\leq & 1 + C_0 \sqrt{t_0(h) }/h + 2C_0 \sqrt{2ds} \p{1+r(s)} \\
= & 1 + 8C_0 \sqrt{ \ln \p{2c_- Nh^d}} + 2C_0 \sqrt{2ds} \p{1+r(s)}
\end{align*}
where we have used the bound 
$$m(t_0) = \EE_x^{(h)}[ \rho(x, W_{t_0(h)}) ] \leq C_0 h \EE(N_{t_0(h)/h^2}) = \frac{ C_0 t_0(h) }{h} .$$
 Dividing by $\sqrt{s} \geq 1$ we get
\[
\b{ e^{-\frac{e+1}{ed}}K^{-1/d}\sqrt{t_0}} e^{r(s)} \leq \frac{1+M(s)}{\sqrt{s}} \leq \p{1 + C_0 \sqrt{64 \ln \p{2c_- Nh^d}}} + 2C_0 \sqrt{2d} \p{1+r(s)}
\]
In particular, weakening the constants we find $c$ that depend on $d,K,C_0,c_-$ such that
\begin{align*}
e^{r(s)} \leq c t_0^{-1/2} \p{1 + \sqrt{ \ln \p{Nh^d}} + r(s)} \iff & \frac{e^{1+\sqrt{ \ln \p{Nh^d}}+r(s)}}{1+\sqrt{ \ln \p{Nh^d}}+r(s)} \leq ce t_0^{-1/2} \exp \p{\sqrt{\ln \p{Nh^d}}} \\
\iff & \underbrace{\frac{e^{\frac{1}{2}\p{1+\sqrt{ \ln \p{Nh^d}}+r(s)}}}{\frac{1}{2}\p{1+\sqrt{ \ln \p{Nh^d}}+r(s)}}}_{ \geq 2e^{1/2}} \frac{1}{2} e^{\frac{1}{2}\p{1+\sqrt{ \ln \p{Nh^d}}+r(s)}} \leq ce t_0^{-1/2} \exp \p{\sqrt{ \ln \p{Nh^d}}} \\ 
\implies & e^{\frac{1}{2}\p{1+\sqrt{ \ln \p{Nh^d}}+r(s)}} \leq ce^{1/2} t_0(h)^{-1/2} \exp \p{\sqrt{ \ln \p{Nh^d}}}
\end{align*}
In particular this implies
\[
r(s) \leq 2 \b{\ln \p{ce^{1/2}} + \ln \p{1/\sqrt{t_0(h)}} + \sqrt{ \ln \p{Nh^d}}}
\]
Using $t_0(h) = 64h^2 \ln \p{2c_-Nh^d}$ and $h_- \leq h \leq h_0$ we thus find, for some $c = c_{d,c_-,K,C_0}$
\[
r(s) \leq c \p{\ln \frac{1}{h_-} \vee \sqrt{\ln N}}
\]
Since $h_- >> \p{\frac{\ln N}{N}}^{1/d}$ this implies $\ln \frac{1}{h_-} \leq \sqrt{\ln N}$, and therefore
\[
M(s) \leq c \ln N \sqrt{s}
\]
\end{proof}

\subsubsection{On diagonal lower bound}

In this section we set again $t_0 = t_0(h) = 64h^2 \ln \p{2c_- Nh^d}$.

\begin{definition}\label{def:exit_time}
	For every $x \in V$ and $r > 0$ we define the exit time $\tau \p{x,r}$ of the ball $B_\rho(x,r)$ by
	\[
	\tau \p{x,r} = \inf \b{t \geq 0 : \rho \p{x,W_t}>r} = \inf \b{t \geq 0 : W_t \notin B_\rho(x,r)}
	\]
\end{definition}

\begin{lem}\label{lem:exit_time}
	There exists $c = c_{\MM,p_0,d}$ such that on $A_N$ and with $2t_0 \leq t \leq t_1(h), r = c\ln N \sqrt{t}$ we have
	\[
	\Pro_x^{(h)}\p{\tau(x,r) \leq t} \leq \frac{1}{2}
	\]
\end{lem}
\begin{proof}
	We have, with $\tau = \tau \p{x,r}, t_0 \leq t \leq t_1(h), r>0$
	\begin{align*}
		r \Pro_x^{(h)} \p{\tau \leq t} = & r \EE_x^{(h)} \c{\II_{\tau \leq t}} \\ 
		\leq & \EE_x^{(h)} \c{r \II_{\rho \p{x,W_{t \wedge \tau}} > r} } \\
		\leq & \EE_x^{(h)} \c{\rho \p{x,W_{t \wedge \tau}}} \\ 
		\leq & \EE_x^{(h)} \c{\rho \p{x,W_{2t}}} + \EE_x^{(h)} \c{\rho \p{W_{2t},W_{t \wedge \tau}}} \\
		= & \EE_x^{(h)} \c{\rho \p{x,W_{2t}}} + \EE_x^{(h)} \c{ \EE_x^{(h)}\c{\rho \p{W_{2t},W_{t \wedge \tau}}|W_{t \wedge \tau}} } \\
		\leq & \EE_x^{(h)} \c{\rho \p{x,W_{2t}}} + \sup_{z \in  V, s \leq t} \EE_z^{(h)} \c{\rho \p{z,W_{2t - s}}} 
		\end{align*}
	Now using lemma \ref{lem:escape_rate} we find, as long as $t_0(h) \leq t \frac{1}{2} t_1(h)$
	\begin{align*}
		r \Pro_x^{(h)} \p{\tau \leq t} \leq & m( 2t, x) + \sup_{z \in  V, s \leq t} m(2t-s, z)  \\
		\leq & c \ln N \p{\sqrt{t} + \sqrt{2t-s}}
		\leq c \ln N \sqrt{t}
	\end{align*}
	the value of $c$ changing from line to line. Hence by setting $r = r(t) = 2c \ln N \sqrt{t}$ we have
	\[
	\Pro_x^{(h)} \p{\tau(x,r) \leq t} \leq \frac{1}{2}.
	\]
\end{proof}

\begin{lem}\label{lem:hklb}
	There exists $c = c_{\MM,p_0,d}$ such that on $A_N$, for any $t_0(h) \leq t \leq t_1(h) \wedge \p{\frac{r_0}{c \ln N}}^2$ we have
	\[
	\forall h_- \leq h \leq h_0, \forall x \in V, p_t^{(h)}(x,x) \geq c \frac{t^{-d/2}}{\ln^d N} 
	\]
\end{lem}
\begin{proof}
	We have, with $\tau = \tau \p{x,r}, t_0/2 \leq t \leq t_1(h), r = r(t)$ from Lemma \ref{lem:exit_time} and $B = B_\rho \p{x,r}$
	\[
	\Pro_x^{(h)}\p{W_t \notin B} \leq \Pro_x^{(h)}\p{\tau(x,r) \leq t} \leq \frac{1}{2}
	\]
	Hence
	\begin{align*}
		\frac{1}{2} \leq & \Pro_x^{(h)}\p{W_t \in B} 
		=  \sum_{y \in B} p_t^{(h)}\p{x,y} \nu_y^{(h)} \\
		& \leq  \p{\nu^{(h)}\p{B}}^{1/2} \p{\sum_{y \in B} p_t^{(h)}(x,y)^2 \nu_y^{(h)}}^{1/2} 
		 \leq  \p{\nu^{(h)}\p{B}}^{1/2} \p{\sum_{y \in V} p_t^{(h)}(x,y)^2 \nu_y^{(h)}}^{1/2} \\
		&=  \p{\nu^{(h)}\p{B}}^{1/2} p_{2t}^{(h)}(x,x)^{1/2}
	\end{align*}
	Therefore we have the inequality
	\[
	p_{2t}^{(h)}(x,x) \geq \frac{1}{4 \nu^{(h)}\p{B}}
	\]
	We now want to conclude using Theorem \ref{thm:volume_regularity}. For this we need to guarantee that
	\[
	h_- \leq r(t) \leq r_0 \iff \p{\frac{h_-}{c \ln N}}^2 \leq t \leq \p{\frac{r_0}{c \ln N}}^2
	\]
	Since $2t \geq t_0 = 64h^2 \ln \p{2c_-Nh^d} \geq 64h_-^2 \ln \p{2c_- Nh_-^d}$ and $h_-$ satisfies \ref{assumption:h}, we have $t \geq \p{\frac{h_-}{c \ln N}}^2$. Since by assumption $t \leq \p{\frac{r_0}{c \ln N}}^2$, we conclude that $h_- \leq r(t) \leq r_0$. Hence theorem \ref{thm:volume_regularity} implies
	\[
	\nu^{(h)}\p{B} \leq \frac{c_+^2}{c_-} r(t)^d
	\]
	and therefore
	\[
	p_{2t}^{(h)}(x,x) \geq \frac{c_-}{4 c_+^2} r(t)^{-d} = \frac{c_- c^{-d}}{4 c_+^2} \frac{t^{-d/2}}{\ln^d N}.
	\]
\end{proof}

\begin{lem}\label{lem:lower_bound_t1}
There exists $c = c_{\MM,p_0,d}$ such that on $A_N$ we have, for any $h_- \leq h \leq h_0$
\[
t_1(h) \geq \frac{c}{\ln^2 N}
\]
\end{lem}
\begin{proof}
If $t_1(h) < \p{\frac{r_0}{c \ln N}}^2$ then for any $t_0/2 \leq t \leq t_1(h)$ using lemma \ref{lem:hklb} we have
\[
p_{2t}^{(h)}(x,x) \geq \frac{c_- c^{-d}}{4 c_+^2} \frac{t^{-d/2}}{\ln^d N}.
\]
In particular
\[
K' = p_{2t_1(h)}^{(h)}(x,x) \geq \frac{c_- c^{-d}}{4 c_+^2} \frac{t_1(h)^{-d/2}}{\ln^d N},
\]
i.e
\[
t_1(h) \geq \p{\frac{c_- c^{-d}}{4 c_+^2} \frac{K'^{-1}}{\ln^d N}}^{2/d}.
\]
Hence
\[
t_1(h) \geq \p{\frac{r_0}{c \ln N}}^2 \wedge \p{\frac{c_- c^{-d}}{4 c_+^2} \frac{K'^{-1}}{\ln^d N}}^{2/d},
\]
which concludes the proof.
\end{proof}

\subsection{Control of the Laplacian spectrum}

Recall that here $t_0 = t_0(h) = 64h^2 \ln \p{2c_-Nh^d}$ and $h_0 = h_{0|\MM,p_0,d}$ is a theoretical constant.

\begin{lem}\label{spectrum:lower_bound}
	There exists constants $b_1,b_2,b_3>0$ such that, on $A_N$, for any $h_- \leq h \leq h_0$
	\[
	\forall b_1  \ln^d N \leq j \leq b_2 t_0^{-d/2}, \quad \lambda_j^{(h)} \geq b_3 j^{2/d}
	\]
\end{lem}
\begin{proof}
	Using the inequality $\II_{\lambda \leq \Lambda} \leq e^{1-\lambda/\Lambda}$ we find
	\begin{align*}
		\# \b{1 \leq j \leq N : \lambda_j^{(h)} \leq \Lambda} = & \sum_{j=1}^N \II_{\lambda_j^{(h)} \leq \Lambda} 
		= \sum_{j=1}^N \II_{\lambda_j^{(h)} \leq \Lambda} \norm{u_j}_N^2 \\
		= & \sum_{j=1}^N \II_{\lambda_j^{(h)} \leq \Lambda} \sum_{x \in V} u_j(x)^2 \nu_x 
		=  \sum_{x \in V} \b{ \sum_{j=1}^N \II_{\lambda_j^{(h)} \leq \Lambda} u_j(x)^2 } \nu_x 
		\leq  \sum_{x \in V} \b{ \sum_{j=1}^N e^{1-\lambda_j^{(h)}/\Lambda} u_j(x)^2 } \nu_x \\
		= & \sum_{x \in V} ep_{\Lambda^{-1}}^{(h)}(x,x) \nu_x 
		\leq  e \max_{x \in V} p_{\Lambda^{-1}}^{(h)}(x,x)
	\end{align*}
	
	We now use theorem \ref{thm:hkb} with $t = \Lambda^{-1}$: there exist constants $a_2,a_3,h_0 > 0$ (that also depend on $\MM,p_0$) such that, on $A_N$, for all $h_- \leq h \leq h_0$ and $a_1^{-1}\ln^2 N \leq \Lambda \leq t_0^{-1}$ 
	\[
	\forall x \in V, p_t^{(h)}(x,x) \leq a_3 \Lambda^{d/2}
	\]
	which implies
	\[
	\# \b{1 \leq j \leq N : \lambda_j \leq \Lambda} \leq e a_3 \Lambda^{d/2}
	\]
	In particular $j = 2ea_3 \Lambda^{d/2} \iff \Lambda = \p{j/2ea_3}^{2/d}$ yields
	\[
	\lambda_j^{(h)} > \Lambda = \p{j/2ea_3}^{2/d}
	\]
	This is valid if
	\[
	a_1^{-1}\ln^2 N \leq \p{j/2ea_3}^{2/d} \leq t_0^{-1}
	\]
	which gives the result by inverting the inequality.
\end{proof}

\begin{lem}\label{spectrum:upper_bound}
	There exist constants $b_4,b_5,b_6>0$ such that on $A_N$, for any $h_- \leq h \leq h_0$
	\[
	\forall b_4 \leq j \leq b_5 \frac{t_0^{-d/2}}{\ln^{3d/2} N}, \quad \lambda_j^{(h)} \leq b_6 j^{2/d} \ln^3 N
	\]
\end{lem}
\begin{proof}
	As in in the proof of lemma 3.19 in \cite{coulhon_heat_2012} we start by the following inequality
	\begin{align*}
		e^{-t\lambda} =  \II_{[0,\Lambda]}(\lambda) e^{-t\lambda} + \sum_{l \geq 0} \II_{[2^l \Lambda,2^{l+1} \Lambda]}(\lambda) e^{-t\lambda} 
		\leq  \II_{[0,\Lambda]}(\lambda) + \sum_{l \geq 0} e^{-t2^l \Lambda}								 \II_{[2^l \Lambda,2^{l+1} \Lambda]}(\lambda)
	\end{align*}
	which holds for any $t, \Lambda > 0$. Hence
	\begin{align*}
	p_t^{(h)}(x,x) = & \sum_{j=1}^N e^{-t\lambda_j} u_j(x)^2 \\
	\leq & \sum_{j=1}^N \II_{[0,\Lambda]}\p{\lambda_j^{(h)}} u_j(x)^2 + \sum_{l \geq 0} e^{-t2^l \Lambda}																													\sum_{j=1}^N \II_{[2^l \Lambda,2^{l+1} \Lambda]}\p{\lambda_j^{(h)}} u_j(x)^2
	\end{align*}
	Since for any $h > 0, j = 1,\ldots,N$ we have $\lambda_j^{(h)} \leq 2h^{-2}$ (because $\norm{\LL}_{L^\infty \p{\nu}} \leq 2h^{-2}$), we can stop the sum at $L \in \mathbb{N}$ such that $2^L \Lambda \leq 2h^{-2} \leq 2^{L+1} \Lambda$ :
	\[
	p_t^{(h)}(x,x) \leq \sum_{j=1}^N \II_{[0,\Lambda]}\p{\lambda_j^{(h)}} u_j(x)^2 + \sum_{l = 0}^L e^{-t2^l \Lambda}																							\sum_{j=1}^N \II_{[2^l \Lambda,2^{l+1} \Lambda]}\p{\lambda_j^{(h)}} u_j(x)^2.
	\]
	Using theorem \ref{thm:hkb} we find, on $A_N$ and for $t_0 \leq t \leq \frac{a_1}{\ln^2 N}$
	\[
	p_t^{(h)}(x,x) \geq a_2 \frac{t^{-d/2}}{\ln^d N} \implies a_2 \frac{t^{-d/2}}{\ln^d N} \leq \sum_{j=1}^N \II_{[0,\Lambda]}\p{\lambda_j^{(h)}} u_j(x)^2 + \sum_{l = 0}^L e^{-t2^l \Lambda}																																																		 \sum_{j=1}^N \II_{[2^l \Lambda,2^{l+1} \Lambda]}\p{\lambda_j^{(h)}} u_j(x)^2
	\]
	We now define $K \in \mathbb{N}$ such that $2^{K+1}\Lambda \leq t_0^{-1} \leq 2^{K+2}\Lambda$. If $a_1^{-1} \ln^2 N \leq \Lambda \leq t_0^{-1}$ then for $1 \leq l \leq K$ we have $a_1^{-1} \ln^2 N \leq \Lambda \leq 2^{l+1}\Lambda \leq t_0^{-1}$, and therefore
	\begin{align*}
	\sum_{l = 0}^K e^{-t2^l \Lambda}																																																		 \sum_{j=1}^N \II_{[2^l \Lambda,2^{l+1} \Lambda]}\p{\lambda_j^{(h)}} u_j(x)^2 \leq & \sum_{l = 0}^K e^{-t2^l \Lambda}																																																		 \sum_{j=1}^N \II_{[0 \Lambda,2^{l+1} \Lambda]}\p{\lambda_j^{(h)}} u_j(x)^2 \\
	\leq & \sum_{l = 0}^K e^{-t2^l \Lambda}																																																		 \sum_{j=1}^N e^{1-2^{-(l+1)} \Lambda^{-1}\lambda_j^{(h)}} u_j(x)^2 \\
	= & e p_{2^{-(l+1)}\Lambda^{-1}}^{(h)}(x,x) 
	\leq  e a_3 2^{\frac{d(l+1)}{2}}\Lambda^{d/2},
	\end{align*}
	which implies
	\begin{align*}
	a_2 \frac{t^{-d/2}}{\ln^d N} \leq & \sum_{j=1}^N \II_{[0,\Lambda]}\p{\lambda_j^{(h)}} u_j(x)^2 + e a_3 2^{d/2}\sum_{l \geq 0} e^{-t2^l \Lambda} \p{2^l\Lambda}^{d/2} + \sum_{j=1}^N \sum_{l = K+1}^L e^{-t2^l \Lambda} u_j(x)^2 \\
	\leq & \sum_{j=1}^N \II_{[0,\Lambda]}\p{\lambda_j^{(h)}} u_j(x)^2 + e a_3 2^{d/2} \sum_{j=1}^N \p{2^l\Lambda}^{d/2} + L e^{-t2^{K+1}\Lambda} \sum_{j=1}^N u_j(x)^2.
	\end{align*}
	Taking the expectation of this quantity with respect to $x \sim \nu$ and using $\sum_{x \in V} u_j(x)^2 \nu_x = 1$ then yields
	\[
	a_2 \frac{t^{-d/2}}{\ln^d N} \leq \# \b{j : \lambda_j^{(h)} \leq \Lambda} + ea_3 2^{d/2} t^{-d/2} \sum_{l \geq 0} e^{-t2^l \Lambda} \p{2^l t\Lambda}^{d/2} + LNe^{-t2^{K+1} \Lambda}
	\]
	Using an integral comparison (see e.g \cite{Barlow_2017} lemma A.36) we find, for any $x \geq 1$
	\[
	\sum_{l \geq 0} e^{-2^l x} \p{2^lx}^{d/2} \leq c_d x^{d/2-1} e^{-x/2} \leq c_d c_d'e^{-x/4},\quad  c_d' = \sup_{x \geq 1} x^{d/2-1} e^{-x/4} < +\infty.
	\]
	In particular, if $t \geq \Lambda^{-1}$, with $c = c_d c_d'$
	\[
	\sum_{l \geq 0} e^{-t2^l \Lambda} \p{2^l t\Lambda}^{d/2} \leq c e^{-t\Lambda/4}.
	\]
	Therefore we obtain : for any $a_1^{-1} \ln^2 N \leq \Lambda^{-1} \leq t \leq t_0^{-1}$
	\begin{align*}
	\# \b{j : \lambda_j^{(h)} \leq \Lambda} \geq & t^{-d/2} \p{ a_2 \frac{1}{\ln^d N} - cea_3 2^{d/2} e^{-t\Lambda/4} - LNt^{d/2}e^{-t2^{K+1} \Lambda}} \\
	\geq & t^{-d/2} \p{ a_2 \frac{1}{\ln^d N} - cea_3 2^{d/2} e^{-t\Lambda/4} - LNt^{d/2}e^{-t\Lambda}}
	\end{align*}
	Let $t = \kappa \Lambda^{-1} \ln N, \kappa > 0$. Then
	\[
	\# \b{j : \lambda_j^{(h)} \leq \Lambda} \geq \frac{\Lambda^{d/2}}{\ln^{d/2} N} \p{ a_2 \frac{1}{\ln^d N} - cea_3 2^{d/2} N^{-\kappa/4} - LN\Lambda^{d/2}\kappa^{-d/2} \p{\ln N}^{-d/2}N^{-\kappa}}
	\]
	But notice that $L \leq \frac{\ln \p{2h^{-1}\Lambda^{-1}}}{\ln 2} \leq \frac{\ln \p{2h_-^{-1}a_1 \p{\ln N}^{-2}}}{\ln 2} \lesssim \ln N$, which implies that for $\kappa>0$ sufficiently large we find
	\[
	\# \b{j : \lambda_j^{(h)} \leq \Lambda} \geq \frac{a_2\Lambda^{d/2}}{2\ln^{3d/2} N}
	\]
	To conclude we only need to guarantee the conditions on $t,\Lambda,\kappa$ :
	\[
	t_0 \leq \Lambda^{-1} \leq t = \kappa \Lambda^{-1} \ln N \leq \frac{a_1}{\ln^2 N} \iff \kappa a_1^{-1} \ln^3 N \leq \Lambda \leq t_0^{-1}
	\]
	i.e we have obtained
	\[
	\forall \kappa a_1^{-1} \ln^3 N \leq \Lambda \leq t_0^{-1}, \# \b{j : \lambda_j^{(h)} \leq \Lambda} \geq \frac{a_2\Lambda^{d/2}}{2\ln^{3d/2} N}
	\]
	In particular with  $j = \frac{a_2\Lambda^{d/2}}{2\ln^{3d/2} N} $ so that $ \Lambda = \p{2j/a_2}^{2/d} \ln^3 N$ we get
	\[
	\forall \frac{a_2}{2} \p{\kappa a_1^{-1}}^{d/2} \leq j \leq \frac{a_2}{2} \frac{t_0^{-d/2}}{\ln^{3d/2} N}, \lambda_j^{(h)} \leq \p{2j/a_2}^{2/d} \ln^3 N
	\]
	This concludes the proof.
\end{proof}

\subsection{$L^\infty\p{\nu}-L^2\p{\nu}$ comparison theorem}

Again recall that here $t_0 = t_0(h) = 32h^2 \ln \p{2c_- Nh^d}$.

\begin{lem}\label{thm:norm_comparison}
	On $A_N$, for any $h_- \leq h \leq h_0$ with
	\[
	\Sigma_\Lambda = span \b{u_j^{(h)} : \lambda_j^{(h)} \leq \Lambda}
	\]
	we have
	\[
	\forall a_1^{-1} \ln^2 N \leq \Lambda \leq t_0^{-1}, \quad  f \in \Sigma_\Lambda, \quad  \norm{f}_{L^\infty \p{\nu}}^2 \leq ea_3 \Lambda^{d/2} \norm{f}_{L^2\p{\nu}}^2
	\]
	Moreover, with
	\[
	\Sigma^J = span \b{u_j^{(h)} : j \leq J}
	\]
	we have
	\[
	\forall b_4 \leq J \leq b_5 \frac{t_0^{-d/2}}{\ln^{3d/2} N}, \quad f \in \Sigma^J, \quad \norm{f}_{L^\infty \p{\nu}}^2 \leq ea_3 b_6^{d/2} J \ln^{3d/2} N \norm{f}_{L^2\p{\nu}}^2
	\]
\end{lem}
\begin{proof}
	Since $f \in \Sigma_\Lambda$ we have $\II_{[0,\Lambda]}\p{\LL}(f) = f$, therefore, with $\II_{[0,\Lambda]}\p{\LL}(x,y)$ the kernel of $\II_{[0,\Lambda]}\p{\LL}$ with respect to $\nu^{(h)}$ given by
	\[
	\II_{[0,\Lambda]}\p{\LL}(x,y) = \sum_{j=1}^N \II_{[0,\Lambda]}\p{\lambda_j^{(h)}} u_j^{(h)}(x) u_j^{(h)}(y).
	\]
	Then $\forall x\in V$, 
	\begin{align*}
 f(x)^2 & = \c{\II_{[0,\Lambda]}\p{\LL}(f)}(x)^2 
		 = \c{\sum_y \II_{[0,\Lambda]}\p{\LL}(x,y) f(y) \nu_y^{(h)}}^2 \\
		& \leq \c{\sum_y \II_{[0,\Lambda]}\p{\LL}(x,y)^2 \nu_y} \c{\sum_y f(y)^2 \nu_y^{(h)}} 
		 \leq \norm{f}_N^2 \max_y \II_{[0,\Lambda]}\p{\LL}(x,y)^2 \\
		& = \norm{f}_N^2 \max_y \c{\sum_{j=1}^N \II_{[0,\Lambda]}(\lambda_j) u_j^{(h)}(x) u_j^{(h)}(y)}^2 \\
		& \leq \norm{f}_N^2 \max_y \c{\sum_{j=1}^N \II_{[0,\Lambda]}(\lambda_j) u_j^{(h)}(x)^2} \c{\sum_{j=1}^N \II_{[0,\Lambda]}(\lambda_j) u_j^{(h)}(y)^2} \\
		& = \norm{f}_N^2 \max_y\II_{[0,\Lambda]}\p{\LL}(x,x) \II_{[0,\Lambda]}\p{\LL}(y,y)  \leq \norm{f}_N^2 \max_x \II_{[0,\Lambda]}\p{\LL}(x,x)^2 \\
		& = \norm{f}_N^2 \max_x \sum_{j=1}^N \II_{[0,\Lambda]}\p{\lambda_j} u_j^{(h)}(x)^2 \\
		 &\leq \norm{f}_N^2 \max_x \sum_{j=1}^N e^{1-\lambda_j/\Lambda} u_j^{(h)}(x)^2 
		 = \norm{f}_N^2 \max_x ep_{\Lambda^{-1}}^{(h)}(x,x)
	\end{align*}
	Using theorem \ref{thm:hkb} we thus find
	\[
	\forall a_1^{-1} \ln^2 N \leq \Lambda \leq t_0^{-1}, \quad \forall f \in \Sigma_\Lambda,\quad \norm{f}_{L^\infty\p{\nu}}^2 \leq ea_3 \Lambda^{d/2} \norm{f}_{L^2\p{\nu}}^2.
	\]
	Now if $b_4 \leq J \leq b_5 \frac{t_0^{-d/2}}{\ln^{3d/2} N}$, lemma \ref{spectrum:upper_bound} implies $f \in \Sigma_\Lambda, \Lambda = b_6 J^{2/d} \ln^3 N$, which concludes the proof.
\end{proof}
		
\section{Proofs of Lemmas \ref{lemma:regularity} and  \ref{lemma:concentration}}\label{appendix:proofs_approximation_results}
	
\subsection{Proof of Lemma \ref{lemma:regularity}} \label{proof:lemma:regularity}
	
We first bound $T_h f$ for $\beta\leq 1$, then we treat the case $1<\beta \leq 2$ followed by the case $\beta>2$. Recall that $h \leq h_0\wedge \pi r_0 /C_0$. 

\noindent 
\textbf{ Case $\beta\leq 1$} In this case the result is trivial since $\forall x \in \MM$, 
			\begin{align*}
			\abs{T_h f(x)}  &=  \abs{\frac{1}{h^2 P_0(B_{\RR^D}(x,h))} \int_{B_{\RR^D}(x,h) \cap \MM} \p{f(x)-f(y)} p_0(y)\mu(dy)} \\
			\leq & h^{-2} \sup_{y \in B_{\RR^D}(x,h) \cap \MM} \abs{f(x)-f(y)} 
			\leq h^{-2} \sup_{y \in B_{\RR^D}(x,h)} \sum_{i \in I} \abs{\chi_i f(x)-\chi_i f(y)} \\
			\leq & h^{-2} \max_{i \in I} \sup_{y \in B_{\RR^D}(x,h)} \abs{(\chi_i f \circ \phi_i^{-1}) (\phi_i (y)) - (\chi_i f \circ \phi_i^{-1}) (\phi_i (y))} \\
			\leq & h^{-2} \norm{f}_{\CC^\beta \p{\MM}} \max_{i \in I} \sup_{y \in B_{\RR^D}(x,h)} \norm{\phi_i(y)-\phi_i(x)}^\beta
			\end{align*}
			But since $h \leq h_0$, for any $y \in B_{\RR^D}(x,h) \cap \MM$ we have $\rho(x,y) \leq C_0 \norm{x-y} \leq C_0h$. 
			Hence since the functions $\phi_i$ are Lipschitz
			\[
			\abs{T_hf(x)} \leq h^{-2} \norm{f}_{\CC^\beta \p{\MM}} \max_{i \in I} \sup_{y \in B_\rho(x,h)} \norm{\phi_i(y)-\phi_i(x)}^\beta \lesssim \norm{f}_{\CC^\beta \p{\MM}} h^{\beta-2}
			\]
			
\noindent 
\textbf{ Case $1<\beta\leq 2$} Recall that $C_0$ is such that $\rho(x,y) \leq C_0 \norm{x-y}$ for any $x,y \in \MM$. Recall that $f_0 \in \CC^\beta \p{\MM}$, so that using a  Taylor expansion of $f_x = f \circ \phi_x, $  with $\phi_x = \exp_x \p{\psi(x,\cdot)}$ at some $x \in V$ and  for any $\norm{x-y}<h$ we have $\rho(x,y) \leq C_0 \norm{x-y} \leq C_0 h$ and therefore
			\[
			f_0(y) = f(x) + df_x(0).\omega + R(x,y)
			\]
			where $\exp_x \p{\psi(x,\omega)} = y$ and
			\[
			\max_{x \neq y \in \MM} \frac{\abs{R (x,y)}}{\rho(x,y)^{\beta}} \lesssim \norm{f}_{\CC^{\beta}\p{\MM}}.
			\]
			This, together with the fact that when $x\sim y$ then $\norm{x-y}<h$,  implies that 
			\begin{align*}
				 \left| T_hf(x)  - \frac{1}{h^2 P_0(B_{\RR^D}(x,h))} \int_{\exp_x \circ \phi_x^{-1}B_{\RR^D}(x,h)} df_x(0).\omega p_x(\omega) J_x(\omega) d\omega \right| \lesssim   \frac{\norm{f}_{\CC^{\beta}} h^\beta }{h^2 }.
			\end{align*}
			where $p_x = p_0 \circ \phi_x, J_x = \abs{g(\exp_x(\psi(x,\omega)))}^{1/2}$. It remains to bound
			\[
			\frac{1}{h^2 P_0(B_{\RR^D}(x,h))} \int_{\exp_x \circ \phi_x^{-1} B_{\RR^D}(x,h) \cap \MM} df_x(0).\omega p_x(\omega) J_x(\omega) d\omega
			\]
			For this notice that since $p_0 \in \CC^{\beta-1} \p{\MM}$ and $J_x(\omega) = 1 + \mathcal{O}(\norm{\omega}^2)$ we have
			\[
			p_x\p{\omega}J(\omega) = \p{p_0(x) + \mathcal{O}\p{h^{\beta-1}}}\p{1 + \mathcal{O}\p{h^2}} = p_0(x) + \mathcal{O}\p{h^{\beta-1}}
			\]
			Therefore
			\begin{align*}
				\frac{\int_{\exp_x \circ \phi_x^{-1} B_{\RR^D}(x,h)} df_x(0).\omega p_x(\omega) J_x(\omega) d\omega}{P_0(B_{\RR^D}(x,h))} = & \frac{\int_{\exp_x \circ \phi_x^{-1} \p{B_{\RR^D}(x,h)}} df_x(0).\omega \p{p_0(x) + \mathcal{O}\p{h^{\beta-1}}} d\omega}{P_0(B_{\RR^D}(x,h))} \\
				= & p_0(x) \frac{\int_{\exp_x \circ \phi_x^{-1}\p{B_{\RR^D}(x,h)}} df_x(0).\omega d\omega}{P_0(B_{\RR^D}(x,h))} + \mathcal{O}\p{\norm{f}_{\CC^\beta\p{\MM}} h^\beta}
			\end{align*}
			Note that since  $\exp_x \circ \phi_x^{-1} \p{B_\rho(x,h)} = \b{\omega \in T_x\MM : \norm{\omega} < h}$ is a symmetric set, we get
			\[
			\int_{\exp_x \circ \phi_x^{-1}\p{B_\rho(x,h)}} df_x(0).\omega d\omega = 0,
			\]			
			therefore to control the first term of the right hand side of the above term, it is enough to show that $B_{\RR^D}(x,h)$ is close to $B_\rho(x,h)$. 
			
			Since $h \leq \frac{\pi a_\MM}{C_0}$,  for all $x,y \in \MM, \norm{x-y} < h$ we have $\rho(x,y) \leq C_0 \norm{x-y} \leq C_0 h \leq \pi a_\MM$. Hence, using \eqref{eq:rho:norm}, $\rho(x,y) - \frac{\rho(x,y)^3}{24a_\MM^2} \leq \norm{x-y} \leq \rho(x,y)$ . Therefore if $y \in \MM, \norm{x-y}<h$ but $\rho(x,y) \geq h$ then $\rho(x,y) \leq \norm{x-y} + \frac{\rho(x,y)^3}{24a_\MM^2} \leq h + \frac{C_0^3 }{24a_\MM^2}h^3$. This implies 
			\begin{align*}
				vol &\p{\b{\omega : \exp_x \circ \phi_x(\omega) \in B_\rho(x,h)^c \cap B_{\RR^D}(x,h)}} 
				\leq  vol \p{\b{\omega : h \leq \norm{\omega} < h + \frac{C_0^3 }{24a_\MM^2} h^3}} \\
				&=  vol \p{\b{\omega : \norm{\omega} \leq 1}} \p{\p{h+ \frac{C_0^3 }{24a_\MM^2} h^3}^d - h^d} \lesssim  h^d \p{\p{1 + \frac{C_0^3 }{24a_\MM^2} h^2}^d - 1} \\
				&\leq  C_\MM h^{d+2}
			\end{align*}
			where the constant $C_\MM$ depends on $\MM$ through $a_\MM, h_0, C_0$ only. Therefore, 
			\begin{align*}
			  \int_{\exp_x \circ \phi_x^{-1}\p{B_{\RR^D}(x,h)}} df_x(0).\omega d\omega 
				&=  \int_{\exp_x \circ \phi_x^{-1}\p{B_\rho(x,h)}} df_x(0).\omega d\omega +  \int_{\exp_x^{-1}\p{B_{\RR^D}(x,h) \cap B_\rho(x,h)^c}} df_x(0).\omega d\omega \\
				&=  \mathcal{O}\p{\norm{f}_{\CC^\beta \p{\MM}} h^{d+3}} \leq \mathcal{O}\p{\norm{f}_{\CC^\beta \p{\MM}} h^{d+\beta}}
			\end{align*}
			where we have used that $\beta \leq 2 \leq 3$. Finally 
		since $P_0(B(x,h)) \gtrsim h^d$ we obtain 
			\[
			\frac{1}{h^2 P_0(B(x,h))} \int_{\exp_x \circ \phi_x^{-1} B_{\RR^D}(x,h)} df_x(0).\omega p_x(\omega) J_x(\omega)d\omega = \mathcal{O}\p{\norm{f}_{\CC^\beta\p{\MM}}h^{\beta-2}}, 
			\]
			which in turns implies that 
			\[
			\norm{T_hf_0}_{L^\infty \p{\MM}} = \mathcal{O}\p{\norm{f}_{\CC^\beta\p{\MM}}h^{\beta-2}}.
			\]

\noindent 
\textbf{ Case $\beta > 2$}. Let $k = \lceil \beta/2 \rceil-1$, then $k \geq 1$. 
Since $f \in \CC^\beta \p{\MM}$ and $ p_0 \in \CC^{\beta-1}\p{\MM}$, then for any $i \in I, x \in \UU_i, y \in B_\rho \p{x,r_0}$ (with $I,\p{\UU_i}_{i\in I},\p{\psi_i}_{i\in I},\p{\varphi_i}_{i\in I}$ the objects defined in section \ref{sec:regularity_M_holder_spaces})
			\[
			f_i(y) = f_i(x) + \sum_{l=1}^k \b{\frac{d^{2l-1} f_x(0).\exp_x^{-1}(y)^{2l-1} }{ (2l-1)!} + \frac{d^{2l} f_x(0).\exp_x^{-1}(y)^{2l}}{ (2l)!}  }+ R_{k,1}(x,y)
			\]
			\[
			p_{0,i}(y) = p_{0,i}(x) + \sum_{l=1}^{k-1} \b{ \frac{d^{2l-1} p_x(0).\exp_x^{-1}(y)^{2l-1}  }{ (2l-1)!}+ \frac{d^{2l} p_x(0).\exp_x^{-1}(y)^{2l} }{ (2l)!}} + \frac{ d^{2k-1}p_x(0).\exp_x^{-1}(y)^{2k-1} }{ (2k-1)!}+R_{k,2}(x,y)
			\]
			\[
			\sqrt{\abs{g}}_i (y) = \sqrt{\abs{g}}_i(x) + \sum_{l=1}^{k-1} \b{\frac{d^{2l-1} J_x(0).\exp_x^{-1}(y)^{2l-1}  }{ (2l-1)!}+ \frac{d^{2l} J_x(0).\exp_x^{-1}(y)^{2l} }{ (2l)!}} + \frac{d^{2k-1}p_x(0).\exp_x^{-1}(y)^{2k-1}}{ (2k-1)!} + R_{k,3}(x,y)
			\]
			where
			\begin{align*}
			f_i &= \chi_i f, p_{0,i} = \chi_i p_0, \quad f_x = \p{\chi_i f} \circ \exp_x \circ \psi_i(x,\cdot) \\
		p_x &= \p{\chi_i p_0} \circ \exp_x \circ \psi_i(x,\cdot), \quad  J_x = \p{\chi_i\sqrt{\abs{g}}} \circ \exp_x \circ \psi_i(x,\cdot)
			\end{align*}
			(we drop the index $i$ for convenience), and the remainders satisfy
			\[
			\abs{R_{k,1}(x,y)} \leq \norm{f}_{\CC^\beta \p{\MM}} \rho(x,y)^\beta, \quad \abs{R_{k,2}(x,y)} \leq \norm{p_0}_{\CC^{\beta-1} \p{\MM}} \rho(x,y)^{\beta-1}, \quad \abs{R_{k,3}(x,y)} \leq C_\MM \rho(x,y)^{\beta-1},
			\]
			wheere $C_\MM$ depends on the H\"older constant of the exponential map. 
Thus, by Riemannian change of variables, there exist  coefficients
			\[
			A_{i,m} \in \CC^{\beta-m}\p{\UU_i,\LL\p{\p{\RR^d}^{\otimes m},\RR}}
			\]
			and remainders
			\[
			\abs{B_{i,k}(x,\omega)} \leq C(\beta, \MM, p_0)\norm{f}_{\CC^\beta\p{\MM}} \norm{\omega}^\beta
			\]
			such that 
			\begin{align*}
				& \int_{B_{\RR^D}(x,h)} \p{f_0(x)-f_0(y)}p_0(y) \mu(dy) \\
				= & \int_{\psi_i(x,\cdot)^{-1} \circ \exp_x^{-1}\p{B_{\RR^D}(x,h)}} \p{f_x(0)-f_x(\omega)} p_x(\omega) J_x \p{\omega} d\omega \\
				= & \sum_{i \in I} \int_{\psi_i(x,\cdot)^{-1} \circ \exp_x^{-1}\p{B_{\RR^D}(x,h)}} \b{\sum_{m=1}^{2k} A_{i,m}(x)\omega^m + B_{i,k}(x,\omega)} d\omega \\
				= & \sum_{i \in I} \b{\sum_{m=1}^{2k} A_{i,m}(x) \int_{\psi_i(x,\cdot)^{-1} \circ \exp_x^{-1}\p{B_{\RR^D}(x,h)}} \omega^m d\omega + \int_{\psi_i(x,\cdot)^{-1} \circ \exp_x^{-1}\p{B_{\RR^D}(x,h)}} B_{i,k}(x,\omega) d\omega} \\
				= & \sum_{i \in I} \sum_{m=1}^{2k} A_{i,m}(x) \int_{\psi_i(x,\cdot)^{-1} \circ \exp_x^{-1}\p{B_{\RR^D}(x,h)}} \omega^m d\omega + \mathcal{O}\p{\norm{f}_{\CC^\beta\p{\MM}}h^{d+\beta}}.
			\end{align*}
			We thus get
			\[
			T_hf(x) = \sum_{i \in I} \sum_{m=1}^{2k} A_{i,m}(x) \frac{h^{-(d+2)}\int_{\psi_i(x,\cdot)^{-1} \circ \exp_x^{-1}\p{B_{\RR^D}(x,h)}} \omega^m d\omega}{h^{-d} P_0\p{B_{\RR^D}\p{x,h}}} + \mathcal{O}\p{\norm{f}_{\CC^\beta\p{\MM}} h^{\beta-2}}
			\]
			By lemma \ref{lemma:geometry1} we have, with $s_h(x,v) := \frac{t_h(x,v)-1}{h}$
			\[
			\psi_i(x,\cdot)^{-1} \circ \exp_x^{-1}\p{B_{\RR^D}(x,h)} = \b{thv : v \in \RR^d, \norm{v} = 1, 0 \leq t \leq 1+ h^2 s_h(x,v)}
			\]
			And using the  spherical coordinates, 
			\begin{align*}
				& h^{-(d+2)}\int_{\psi_i(x,\cdot)^{-1} \circ \exp_x^{-1}\p{B_{\RR^D}(x,h)}} \omega^m d\omega \\
				= & h^{-(d+2)}\int_0^\pi \ldots \int_0^\pi \int_0^{2\pi} \int_0^{h + h^3 s_h(x,v)} \p{re_\theta}^m r^{d-1}\sin\p{\theta_1}^{d-2} \sin\p{\theta_2}^{d-3}\ldots \sin\p{\theta_{d-1}}dr d\theta_1\ldots d\theta_{d-1} \\
				= & h^{-(d+2)}\int_0^\pi \ldots \int_0^\pi \int_0^{2\pi} \int_0^{1+h^2 s_h(x,v)} \p{the_\theta}^m (th)^{d-1}\sin\p{\theta_1}^{d-2} \sin\p{\theta_2}^{d-3}\ldots \sin\p{\theta_{d-1}}hdt d\theta_1\ldots d\theta_{d-1} \\
				= & h^{m-2} \int_0^\pi \ldots \int_0^\pi \int_0^{2\pi} \int_0^{1+h^2 s_h(x,v)} \p{te_\theta}^m t^{d-1}\sin\p{\theta_1}^{d-2} \sin\p{\theta_2}^{d-3}\ldots \sin\p{\theta_{d-1}}dt d\theta_1\ldots d\theta_{d-1} \\
				= & h^{m-2} \int_{\mathbb{A}^d} \int_0^{1+h^2 s_h(x,v)}\p{te_\theta}^m t^{d-1} K(\theta)dtd\theta \\
				= & \frac{h^{m-2}}{m+d} \int_{\mathbb{A}^d} e_\theta^m \p{1+h^2 s_h(x,e_\theta)}^{m+d} K(\theta)d\theta
			\end{align*}
			where $\mathbb{A}^d = (0,\pi)^{d-2} \times (0,2\pi)$ and
			\[
			e_\theta = \c{\begin{array}{c}
					\cos (\theta_1) \\
					\sin (\theta_1)\cos (\theta_2) \\
					\sin (\theta_1) \sin (\theta_2) \cos (\theta_3) \\
					\vdots \\
					\sin (\theta_1) \sin (\theta_2) \ldots \sin (\theta_{d-2}) \cos (\theta_{d-1}) \\
					\sin (\theta_1) \sin (\theta_2) \ldots \sin (\theta_{d-2}) \sin (\theta_{d-1})
			\end{array}}, K(\theta) = \sin\p{\theta_1}^{d-2} \sin\p{\theta_2}^{d-3}\ldots \sin\p{\theta_{d-1}}
			\]
			Moreover if $m$ is odd, by symmetry
			\[
			h^{-(d+2)}\int_{\psi_i(x,\cdot)^{-1} \circ \exp_x^{-1}\p{B_{\RR^D}(x,h)}} \omega^m d\omega = h^{-(d+2)}\int_{\psi_i(x,\cdot)^{-1} \circ \exp_x^{-1}\p{B_{\RR^D}(x,h)\backslash B_\rho(0,h)} } \omega^m d\omega
			\]
			and therefore the same computations show that
			\begin{align*}
				h^{-(d+2)}\int_{\psi_i(x,\cdot)^{-1} \circ \exp_x^{-1}\p{B_{\RR^D}(x,h)}} \omega^m d\omega = & \frac{h^{m-2}}{m+d} \int_{\mathbb{A}^d} e_\theta^m \p{\p{1+h^2 s_h(x,e_\theta)}^{m+d}-1} K(\theta)d\theta \\
				= & \frac{h^{m-2}}{m+d} \int_{\mathbb{A}^d} e_\theta^m \p{\sum_{r=1}^{m+d} \binom{m+d}{r} h^{2r} s_h(x,e_\theta)^r} K(\theta)d\theta \\
				= & \frac{h^m}{m+d} \sum_{r=1}^{m+d} \binom{m+d}{r} h^{2r-2} \int_{\mathbb{A}^d} e_\theta^m s_h(x,e_\theta)^r K(\theta)d\theta
			\end{align*}
			All in all for each $l \in \b{1,\ldots,k}, m = 2l-1$ or $2l$ we have
			\begin{align*}
				h^{-(d+2)}\int_{\psi_i(x,\cdot)^{-1} \circ \exp_x^{-1}\p{B_{\RR^D}(x,h)}} \omega^m d\omega = & \frac{h^{2(l-1)}}{m+d} \sum_{r=1}^{m+d} \binom{m+d}{r} h^{2r-1} \int_{\mathbb{A}^d} e_\theta^m s_h(x,e_\theta)^r K(\theta)d\theta \text{ if } m = 2l-1\\				
				h^{-(d+2)}\int_{\psi_i(x,\cdot)^{-1} \circ \exp_x^{-1}\p{B_{\RR^D}(x,h)}} \omega^m d\omega = & h^{2(l-1)} \int_{\mathbb{A}^d} e_\theta^m \p{1+h^2s_h(x,e_\theta)}^{m+d} K(\theta)d\theta \text{ if } m = 2l
			\end{align*}
			This gives
			\[
			T_hf(x) = \frac{1}{h^{-d} P_0\p{B \p{x,h}}} \sum_{i \in I} \sum_{l=1}^k h^{2(l-1)} \b{ A_{i,2l-1}(x) V_{i,h,2l-1}(x) + A_{i,2l}(x)V_{i,h,2l}(x) } + \mathcal{O}\p{\norm{f}_{\CC^\beta\p{\MM}} h^{\beta-2}}
			\]
			where by lemma \ref{lemma:geometry1}, composition and integration we have
			\[
			\max_{\alpha \in \mathbb{N}^d, \abs{\alpha} \leq \alpha - 5} \sup_{\substack{i \in I \\ y \in \varphi_i^{-1}\p{\UU_i} \\ 0 < h < h_+ \\ \norm{v}_{T_{\varphi_i(y)}\MM} = 1}} \norm{\frac{\partial^{\abs{\alpha}}}{\partial x^\alpha} V_{i,h,m}(\varphi_i(y))}_{\p{\RR^d}^{\otimes m}} < +\infty
			\] 
			To conclude it suffices to show that $x \mapsto h^{-d}P_0(B_{\RR^D}(x,h))$ is lower bounded and satisfies
			\[
			\sup_{0 < h < h_+} \norm{h^{-d}P_0(B_{\RR^D}(\cdot,h))}_{\CC^{\beta-1}\p{\MM}} < +\infty
			\]
			Since $p_0$ is lower bounded and $\mu \p{B_{\RR^D}(x,h)} \asymp h^d$ (as can be seen by the Riemannian change of variables formula), we get
			\[
			h^{-d} \int_{B_{\RR^D}(x,h)} p_0(y)\mu(dy) \geq \inf_{\MM}p_0 \times h^{-d} \mu\p{B_{\RR^D}(x,h)} \gtrsim \inf_{\MM}p_0
			\]
			hence $x \mapsto h^{-d} P_0 \p{B_{\RR^D}(x,h)}$ is indeed lower bounded. For the derivatives, doing the same computations as above
			\begin{align*}
				h^{-d}P_0\p{B_{\RR^D}(x,h)} = & \sum_{i \in I} h^{-d} \int_{B_{\RR^D}(x,h)} p_0(y) \mu(dy) \\
				= & h^{-d} \int_{\varphi_i^{-1} \circ \exp_x^{-1}\p{B_{\RR^D}(x,h)}} \p{\chi_i p_0} \circ \exp_x\p{\psi_i(x,\omega)} \sqrt{\abs{g(\exp_x(\psi_i(x,\omega)))}} d\omega \\
				= & h^{-d} \int_{\varphi_i^{-1} \circ \exp_x^{-1}\p{B_{\RR^D}(x,h)}} p_x(\omega) J_x(\omega) d\omega \\
				= & h^{-d} \int_{\mathbb{A}^d} \int_0^{1+h^2 s_h(x,v)} p_x(the_\theta) J_x(the_\theta) \p{th}^{d-1} K(\theta) hdt d\theta \\
				= & \int_{\mathbb{A}^d} \int_0^{1+h^2 s_h(x,v)} p_x(the_\theta) J_x(the_\theta) t^{d-1} K(\theta) dt d\theta
			\end{align*}
			And again, under this form, by composition and integration, $x \mapsto h^{-d} P_0(B_{\RR^D}(x,h))$ has indeed the regularity of $p_0$, i.e is $\CC^{\beta-1}\p{\MM}$ with uniform bound on the derivatives as $h \to 0$.
			
			This concludes the proof by setting
			\[
			g_h^{(l)} = \frac{1}{h^{-d}P_0\p{B_{\RR^D}(x,h)}} \sum_{i \in I} h A_{i,2l-1}(x) V_{i,h,2l-1}(x) + A_{i,2l}(x)V_{i,h,2l}(x)
			\]
			which is $\CC^{\beta-2l}$ since by assumption $\alpha \geq \beta+3$ which implies $\alpha - 5 \geq \beta-2 \geq \beta - 2l$ for any $l \in \b{1,\ldots,k}$.

	\subsubsection{Proof of Lemma \ref{lemma:concentration}}\label{proof:lemma:concentration}

		First if $\beta \leq 1$,  $\norm{T_{h_n} f}_{L^\infty \p{\MM}} \lesssim h_n^{\beta-2}$ follows from lemma \ref{lemma:regularity}. For $\norm{\LL f}_{L^\infty \p{\nu}}$ notice that for each $x \in V$ we have
			\[
			\abs{\p{\LL f}(x)} = \frac{1}{h_n^2 \mu_x} \abs{ \sum_{i=1}^N \p{f(x)-f(x_i)} \II_{x_i \sim x} } \lesssim \norm{f}_{\CC^\beta \p{\MM}} h_n^{\beta-2} .
			\]

Now, if $\beta>1$, 
			\[
			\forall x \in V, h_n^2 \mu_x \LL f (x) = \sum_{i=1}^N \p{f(x)-f(x_i)} \II_{x_i \sim x}
			\]
			Therefore by Bernstein's inequality (see e.g lemma 2.2.9 in \cite{noauthor_weak_nodate}), using $\abs{f(x) - f(x_i)} \II_{x \sim x_i} \lesssim \norm{f}_{\CC^\beta \p{\MM}} h_n$,
			\[
			\Pro_0 \p{\abs{h_n^2 \mu_x \LL f (x) - Nh_n^2 P_0(B_{\RR^D}(x,h_n)) T_{h_n}f(x)} > u} \leq \exp \p{-c \frac{u^2}{Nh_n^{d+2} + h_n u}}
			\]
			and
			\begin{align*}
				\Pro_0 \p{\abs{\mu_x - NP_0(B_{\RR^D}(x,h_n))} > u/h_n^2} \leq & \exp \p{-c \frac{u^2/h_n^4}{Nh_n^d + u/h_n^2}} \\
				\leq & \exp \p{-c \frac{u^2}{Nh_n^{d+4} + uh_n^2}}.
			\end{align*}
			Hence
			\begin{align*}
				 \Pro_0 &\p{\abs{h_n^2 \mu_x \LL f (x) - h_n^2 \mu_x T_{h_n} f (x)} > 2u} 
				\leq  \Pro_0 \p{\abs{h_n^2 \mu_x \LL f (x) - Nh_n^2 P_0(B_{\RR^D}(x,h_n)) T_{h_n}f(x)} > u} \\
				&+\qquad  \Pro_0 \p{\abs{Nh_n^2 P_0(B_{\RR^D}(x,h_n)) T_{h_n}f(x) - h_n^2 \mu_x T_{h_n}f(x)}  > u} \\
				&\leq  \exp \p{-c \frac{u^2}{Nh_n^{d+2} + h_nu}} + \Pro_0 \p{\abs{N P_0(B_{\RR^D}(x,h_n)) - \mu_x}  > u/h_n^2\norm{f}_{\CC^\beta\p{\MM}}}
				\\
				&\leq  \exp \p{-c \frac{u^2}{Nh_n^{d+2} + h_nu}} + \exp \p{-c \frac{u^2}{Nh_n^{d+4} + uh_n^2}}
			\end{align*}
			Therefore, using the fact that $\mu_x \lesssim Nh_n^d$ on $V_N(c)$
			\begin{align*}
				\Pro_0 \p{\abs{\LL f (x) - T_{h_n} f (x)} > u} \leq & \Pro_0 \p{V_N(c)^c} + \Pro_0 \p{\abs{h_n^2 \mu_x \LL f (x) - h_n^2 \mu_x T_{h_n} f (x)} > cNh_n^{d+2} u} \\
				\leq & \Pro_0 \p{V_N(c)^c} + \exp \p{-c \frac{\p{Nh_n^{d+2}u}^2}{Nh_n^{d+2} + Nh_n^{d+3}u}} + \exp \p{-c \frac{\p{Nh_n^{d+2}u}^2}{Nh_n^{d+4} + Nh_n^{d+4}u}} \\
				= & \Pro_0 \p{V_N(c)^c} + \exp \p{-c \frac{N^2 h_n^{2d+4} u^2}{Nh_n^{d+2} + Nh_n^{d+3}u}} + \exp \p{-c \frac{N^2 h_n^{2d+4 }u^2}{Nh_n^{d+4} + Nh_n^{d+4}u}} \\
				= & \Pro_0 \p{V_N(c)^c} + \exp \p{-c \frac{N h_n^{d+2} u^2}{1 + h_nu}} + \exp \p{-c \frac{N h_n^du^2}{1 +u}}
			\end{align*}
			For any $0 < h_n \leq 1, u > 0$ we have $\frac{h_n^2}{1+h_nu} \leq \frac{1}{1+u}$, therefore
			\[
			\Pro_0 \p{\abs{\LL f (x) - T_{h_n} f (x)} > u} \leq \Pro_0 \p{V_N(c)^c} +2 \exp \p{-c \frac{N h_n^{d+2} u^2}{1 + h_nu}}
			\]
			By a union bound this gives
			\[
			\Pro_0 \p{\norm{\LL f - T_{h_n} f}_{L^\infty \p{\nu}} > u} \leq N\Pro_0 \p{V_N(c)^c} + 2N\exp \p{-c \frac{N h_n^{d+2} u^2}{1 + h_nu}}
			\]
			We have also $\Pro_0 \p{V_N(c)^c} \leq e^{-c Nh_n^d}$, so that for any $H>0$, taking $u = M_0 \p{\frac{\ln N}{N}}^{\frac{1}{2}} h_n^{-(1+d/2)}$ with $M_0$ large enough (that depends on $H$) we obtain $h_n u = o(1)$ by assumption \ref{assumption:h} and 
			\[
			\Pro_0 \p{\norm{\LL f - T_{h_n} f}_{L^\infty \p{\nu}} > M_0 \p{\frac{\ln N}{N}}^{\frac{1}{2}} h_n^{-(1+d/2)}} \leq N^{-H}
			\]

	\subsection{Proof of Theorem \ref{corollary:pcr_le}} \label{pr:PCR_LE}
	\begin{proof}
	Recall that $\hat{f} = \sum_{j=1}^J \inner{u_j | Y}_{L^2\p{\nu}} u_j$ so that 
	$$\hat f - f_0 =   \sum_{j=1}^J \inner{u_j | f_0 }_{L^2\p{\nu}} u_j - f_0 +   \sum_{j=1}^J \inner{u_j | \underline \varepsilon }_{L^2\p{\nu}} u_j, \quad \underline \varepsilon \sim \mathcal N( 0, \sigma^2 I_n)$$
	where $I_n$ is the identity matrix in $\mathbb R^n$. 
	Therefore 
	$$\| \hat f - f_0\|_{L^2(\nu)} \leq   \| \sum_{j=1}^J \inner{u_j | f_0 }_{L^2\p{\nu}} u_j - f_0\|_{L^2(\nu)}  + \|  \sum_{j=1}^J \inner{u_j | \underline \varepsilon }_{L^2\p{\nu}} u_j\|_{L^2(\nu)} .$$
	We have by orthonormality of $(u_1, \cdots, u_{J_n})$
	$$ \norm{  \sum_{j=1}^J \inner{u_j | \underline \varepsilon }_{L^2\p{\nu}} u_j}^2_{L^2(\nu)} = \sum_{j=1}^{J_n} \inner{u_j|\underline \varepsilon}_{L^2\p{\nu}}^2$$
	and also that 
		\begin{align*}
		\EE_0 \c{ \left. \inner{u_j|\varepsilon}_{L^2\p{\nu}}^2 \right| x_{1:n}} = & \EE_0 \c{\left. \p{u_j^T Diag\p{\nu} \underline \varepsilon}^2 \right| x_{1:n}} 
		=\sigma^2  u_j^T Diag\p{\nu^2} u_j		=  \sigma^2 \sum_{i=1}^n u_j(x_i)^2 \nu_{x_i}^2
		\end{align*}
		On $A_N$ defined in lemma \ref{thm:volume_regularity}, with  $N=n$ and $ h_- = h_n$  which satisfies the assumption \ref{assumption:h} we have $\nu_{x_i} \lesssim \frac{1}{n}$ so that 
		\[
		\EE_0 \c{\left.  \inner{u_j\right|\underline\varepsilon}_{L^2\p{\nu}}^2 | x_{1:n}} \lesssim \frac{1}{n} \norm{u_j}_{L^2\p{\nu}}^2 = \frac{1}{n}.
		\]
This implies that 
		\[
		\sum_{i=1}^{J_n} \inner{u_j|\underline \varepsilon}_{L^2\p{\nu}}^2 \lesssim \frac{1}{n} \sum_{j=1}^{J_n} Z_j(x_{1:n})^2, \quad \text{where } \quad Z_j(x_{1:n}) = \frac{\inner{u_j|\underline \varepsilon}_{L^2\p{\nu}}}{\EE_0 \c{\left. \inner{u_j|\varepsilon}_{L^2\p{\nu}}^2\right|x_{1:n}}^{1/2}}
		\]
	In particular conditionally on $x_{1:n}$, $Z_j(x_{1:n}) \iid \NN \p{0, 1}$. Hence as in \cite{green_minimax_2021} appendix C1, by a result of \cite{laurent_adaptive_2000} we get for some $c>0$
		\begin{equation}\label{bound:sto}
		\Pro_0 \p{ \norm{\sum_{j=1}^{J_n} \inner{u_j|\varepsilon}_{L^2\p{\nu}}u_j}_{L^2\p{\nu}}^2 > c\frac{J_n}{n} } \leq \Pro_0\p{A_N^c} + \EE_0\left[ \II_{A_N}\Pro_0 \p{\left. \norm{\sum_{j=1}^{J_n} \inner{u_j|\underline\varepsilon}_{L^2\p{\nu}}u_j}_{L^2\p{\nu}}^2 > c\frac{J_n}{n}\right|x_{1:n}}\right] 
		\leq \Pro_0\p{A_N^c} + \exp \p{-J_n}
		\end{equation}
Moreover 
$$\norm{f_{0J_n}-f_0}_{L^2\p{\nu}}^2\leq \norm{p_{J_n} \p{e^{-t\LL}f_t}-f_0}_{L^2\p{\nu}}^2$$
where $p_{J_n} \p{e^{-t\LL}f_t}$ is defined in Theorem \ref{thm:approximation2}. Therefore, using Theorem \ref{thm:approximation2}, we deduce that choosing $t$ accordingly, 
$$\norm{f_{0J_n}-f_0}_{L^2\p{\nu}}^2\leq \norm{p_{J_n} \p{e^{-t\LL}f_t}-f_0}_{L^2\p{\nu}}^2 \leq \tilde \varepsilon_n(J_n)$$
Finally combining this with \eqref{bound:sto}, we obtain that  some $C>0$
		\[
		\Pro_0 \p{ \norm{\hat{f}-f_0}_{L^2\p{\nu}} > C\varepsilon_n } \to 0
		\]
		where
		\[
		\varepsilon_n = 
			\sqrt{\frac{J_n}{n}} +(\ln N)^{\lceil \beta/2 \rceil} \max \left(1, \p{\frac{J_n^{-2/d} \ln N}{h_n^2}}\right)^{\lceil \beta/2 \rceil} \p{ h_n^\beta + \II_{\beta > 1} \p{\frac{\ln N}{Nh_n^d}}^{1/2} h_n } 
		\]
		In order to get a result in expectation, simply notice that
		\[
		\EE_0 \c{\norm{\hat{f}-f_0}_n^q} \leq c \varepsilon_n^q + \EE_0 \c{\norm{\hat{f}-f_0}_n^q \II_{\norm{\hat{f}-f_0}_n > c\varepsilon_n}}
		\]
		Since by definition
		\[
		\norm{\hat{f}-f_0}_n = \min_{f \in span \p{u_1,\ldots,u_J}} \norm{f-f_0}_n \underbrace{\leq}_{f=0} \norm{f_0}_n \leq \norm{f_0}_{L^\infty\p{\MM}}
		\]
		This implies
		\[
		\EE_0 \c{\norm{\hat{f}-f_0}_n^q} \leq c \varepsilon_n^q + \norm{f_0}_{L^\infty \p{\MM}}^q\Pro_0 \c{\norm{\hat{f}-f_0}_n > c\varepsilon_n}
		\]
		By the first part of the proof, for any $H>0$ we have
		\[
		\Pro_0 \c{\norm{\hat{f}-f_0}_n > c\varepsilon_n} \lesssim n^{-H} + e^{-J_n} + \Pro_0 \p{A_N^c}
		\]
		Since $h_n$ satisfies assumption \ref{assumption:h} and $J_n \geq \ln^\kappa n, \kappa > d \geq 1$, for $H>0$ large enough we get
		\[
		\Pro_0 \c{\norm{\hat{f}-f_0}_n > c\varepsilon_n} \lesssim \varepsilon_n^q
		\]
		which concludes the proof.
	\end{proof}
	
	\section{Proof of section \ref{subsection:posterior_contraction_rates}}\label{appendix:proof_posterior_contraction_rates}
	We first recall an adaptation to our setting of the general prior mass and testing approach of \cite{ghosal_convergence_2000,ghosal:vdv:07} that we use to prove posterior contraction rates with respect to the empirical $L^2$ distance $\norm{\cdot}_n$. Note that we apply the approach on a high probability event $A_N$ (which does not change the result), and since in the Gaussian regression setting the Kullback-Leibler divergence and variation are both proportional to the (squared) empirical $L^2$ norm $\|\cdot \|_n^2$, the prior mass condition on the Kullback-Leibler neighbourhood is merely a condition on the prior mass of the $\|\cdot\|_n-$balls centered at $f_0$ (see \cite{ghosal_fundamentals_2017} lemma 2.7, as well as chapter 8 for more details on the prior mass and testing approach).
	\begin{theo}{Contraction rate theorem}\label{theorem:master_theorem}\\
		Assume that $\varepsilon_n \to 0, n\varepsilon_n^2 \to \infty$ and $A_N$ be an event of the form $A_N = \b{x_{1:N} \in \Omega_N}$ for some $\Omega \subset \p{\RR^D}^N$ satisfying $\Pro_0\p{A_N} \to 1$. If there exists fixed constants $C_1>1+\frac{1}{2\sigma^2}, C_2 > 4 \sqrt{1 + 2\sigma^2}, C_3 > 0$ and $\FF_n \subset \RR^V$ such that, on $A_N$
		\begin{itemize}
			\item $\Pi \c{\norm{f-f_0}_n \leq \varepsilon_n} \geq \exp \p{-n\varepsilon_n^2}$
			\item $\Pi \c{\FF_n^c} \leq e^{-C_1n\varepsilon_n^2}$
			\item $\ln N \p{\frac{C_2\varepsilon_n}{2},\FF_n, \norm{\cdot}_n} \leq C_3n\varepsilon_n^2$
		\end{itemize}
		where $N \p{\varepsilon,\FF,\|\cdot\|_n}$ denotes the covering number of the set $\FF$ at scale $\varepsilon$ with respect to the $\|\cdot\|_n-$metric, then for $M>0$ large enough
		\[
		\EE_0 \c{ \Pi \c{\norm{f-f_0}_n > M \varepsilon_n | \mathbb{X}^n}} \to 0
		\]
	\end{theo}
	
	\subsection{Proof of Theorem \ref{theorem:rate_prior1_d_n}}\label{appendix:proof_posterior_contraction_rates:non_adaptive_priors}

As shown by Theorem \ref{theorem:master_theorem}, the proof is based on first bounding from below the prior mass of neighbourhoods of $f_0$, which is done in Lemma \ref{lem:prior_thickness_prior1} below and then to control the entropy by a sieve sequence, which is done in Lemma \ref{lem:entropy}.  Recall that
\begin{equation*}
	\varepsilon_n(J_n, h_n) = 
			\sqrt{\frac{J_n \ln N}{n}} +(\ln N)^{\lceil \beta/2 \rceil} \max \left(1, \p{\frac{J_n^{-2/d} \ln N}{h_n^2}}\right)^{\lceil \beta/2 \rceil} \p{ h_n^\beta + \II_{\beta > 1} \p{\frac{\ln N}{Nh_n^d}}^{1/2} h_n } 
\end{equation*}			
			with $J_n \geq \ln^\kappa N, \kappa > d$ and $h_n$ satisfying assumption \ref{assumption:h}. 
 
	  With $z = z_n = z_1\p{n\varepsilon_n^2}^{1/b_2}$ with $z_1>0$ and $\FF_n = \FF_{z_n,J_n}$. Its $L^\infty\p{\nu}-$ metric entropy is then, according to lemma \ref{lem:entropy},  bounded by 
		\[
		\forall \zeta >0, \quad \ln N \p{ \zeta \varepsilon_n,\FF_n,\norm{\cdot}_{L^\infty \p{\nu}} } \lesssim J_n \ln N \lesssim n\varepsilon_n^2.
		\]
		Moreover by the choice of $z_n$ we have, for some $c>0$
		\[
		\Pi \c{\FF_n^c | J_n, h_n} \leq J_n e^{-b_1z_n^{b_2}} \leq e^{-cz_1^{b_2}n\varepsilon_n^2}
		\]
		
		By taking $z_1 > 0$ large enough we conclude the proof of Theorem \ref{theorem:rate_prior1_d_n}  by applying Lemma \ref{lem:prior_thickness_prior1} together with Theorem \ref{theorem:master_theorem}.

	\subsection{Proof of Theorem \ref{theorem:rate_prior2_d_n}}\label{appendix:proof_posterior_contraction_rates:adaptive_priors}

	The proof is based on that of Theorem \ref{theorem:rate_prior1_d_n}, but showing that $J$ and $h$ can be chosen in a data dependent way. To do that notice that by the non adaptive case, for any $H>0$ there exists a constant $C>0$ such that
	\[
	\Pro_0 \p{ \Pi \c{\norm{f-f_0}_{L^\infty \p{\nu}} > C\varepsilon_n(J_n, h_n) | J_n, h_n} < \exp \p{-nC^2 \varepsilon_n^2}} \leq N^{-H}
	\]
	where $ \varepsilon_n (J_n, h_n)$ is given by \eqref{epsilonn}.
	If for some $J_0>0$, $h_0, h_1>0$
$$
	J_n = J_0 \p{\ln N}^{\frac{2d\p{1+2\tau/d}\lceil \beta/2 \rceil}{2\beta+d}} n^{\frac{d}{2\beta+d}}, \quad h_n \in \left[\frac{h_0J_n^{-1/d}}{2\ln^{\tau/d}N}, \frac{h_1J_n^{-1/d}}{\ln^{\tau/d}N}\right]
	$$ 
	then 
	\[
	\varepsilon_n(J_n, h_n) \lesssim n^{-\frac{\beta}{2\beta+d}} \p{\ln N}^{\frac{\p{2\tau+d}\lceil \beta/2\rceil - \beta \p{\tau + 2\beta/d}}{2\beta+d}}
	\]
	Since
	\begin{align*}
	\Pi \c{\norm{f-f_0}_{L^\infty \p{\nu}} \leq C\varepsilon_n}
	& \geq \pi_J(J_n) \pi_h\p{h_n | J_n} \Pi \c{\norm{f-f_0}_{L^\infty \p{\nu}} > C\varepsilon_n | J_n,h_n}
	\end{align*}
	By assumption we have for some $a_1,b_1,b_2$
	\[
	\pi_J(J_n) \geq e^{-a_1 J_n L_{J_n}} \geq e^{-a_1 J_n \ln N} \text{ and } \pi_h\p{h_n|J_n} \geq b_1 e^{-b_2 h_n^{-d}}
	\]
	Therefore, since $\varepsilon_n \geq \p{nh_n^d}^{-1/2} + \sqrt{\frac{J_n \ln N}{n}} $, for some $C'>0$ we have
	\[
	\Pro_0 \p{ \Pi \c{\norm{f-f_0}_{L^\infty \p{\nu}} \leq C'\varepsilon_n} < e^{-nC'^2\varepsilon_n^2}} \leq N^{-H}
	\]

	To verify the entropy condition, consider for some $k>0$ the set $\FF_n = \cup_{\substack{J \leq k J_n \\ h \in \HH_J}} \FF_{z_n,J,h}$, where $\FF_{z,J,h} $ is defined in  lemma \ref{lem:entropy} and 
	$ z_n = u (n \varepsilon_n^2)^{1/b_2}$ with $u, k$ large enough. Then
			\begin{equation*}
			\begin{split}
			\Pi \c{\FF_n^c} & \leq \Pi_J( J> k J_n) + \sum_{\substack{J\leq k J_n \\ h \in \HH_J}}  J e^{-b_1 z_n^{b_2}} \leq e^{ - C J_n } +(kJ_n)^2 e^{- b_1 u^{b_2} n \varepsilon_n^2} \max_J \# \HH_J  \leq e^{ - C n \varepsilon_n^2 } 
			\end{split}
			\end{equation*}
			for any $C>0$ by choosing $k,u$ large enough, since $\# \HH_J \leq K_1 \exp \p{K_2 J \ln N} \leq K_1 \exp \p{K_2' n \varepsilon_n^2}$ for some $K_1,K_2,K_2'>0$.
	Finally for all $\zeta>0$
	\[
	N( \zeta \varepsilon_n, \FF_n, \norm{\cdot}_{L^\infty \p{\nu}}) \lesssim n \varepsilon_n^2
	\]
	which concludes the proof.

	\subsection{Lemmas \ref{lem:prior_thickness_prior1} and \ref{lem:entropy}} 
	
	\begin{lem}{Prior thickness}\label{lem:prior_thickness_prior1}\\
		Let $J_n \in \b{1,\ldots,N}, J_n \geq \ln^\kappa N, \kappa > d$ and $h_n$ satisfying assumption \ref{assumption:h}. Consider the prior defined by \eqref{prior1}. Then for any $H>0$ there exists $c>0$ such that with the rate $\epsilon_n(J_n, h_n)$ as defined in \eqref{epsilonn}, 
		We have
		\[
		\Pro_0 \p{ \Pi \c{\norm{f-f_0}_{L^\infty\p{\nu}} \leq c\varepsilon_n(J_n,h_n) | J_n, h_n} < \exp \p{-n\varepsilon_n(J_n,h_n)^2} } \leq N^{-H}.
		\]
		In particular, for any $\tau > d/2$ and
		\[
		h_n = n^{-\frac{1}{2\beta+d}} \p{\ln n}^{-\frac{1-\tau-2\p{1+2\tau/d} \lceil \beta/2 \rceil}{2\beta+d}}, J_n = \frac{h_n^{-d}}{\ln^\tau N}
		\]
		then 
		\[
		\varepsilon_n(J_n, h_n) \asymp \p{\ln n}^{\frac{\p{2\tau+d}\lceil \beta/2\rceil + (1-\tau)\beta}{2\beta+d}} n^{-\frac{\beta}{2\beta+d}}.
		\]
	\end{lem}
	\begin{proof}
		Throughout the proof we write $\varepsilon_n = \varepsilon_n(J_n, h_n)$. We use theorem \ref{thm:approximation2} and distinguish the cases : for some $k>0$ large enough, with $t = t_n = k\lambda_{J_n}^{-1} \ln N$
		\begin{enumerate}
			\item If $\beta \leq 2$ we have
			\begin{align*}
			\forall f : V \to \RR, \norm{f-f_0}_{L^\infty\p{\nu}} \leq & \norm{f-p_J \p{e^{-t\LL}f_0}}_{L^\infty\p{\nu}} + \norm{f_0-p_J \p{e^{-t\LL}f_0}}_{L^\infty\p{\nu}} \\
			\lesssim & \norm{f-p_J \p{e^{-t\LL}f_0}}_{L^\infty\p{\nu}} + O_{P_0} \p{ h_n^\beta \p{\lambda_{J_n}^{-1} h_n^{-2} \ln N} }.
			\end{align*}
			Hence for all $c>0$ , there exists $c'>0$, such that 
			\[
			\norm{f-p_J \p{e^{-t\LL}f_0}}_{L^\infty\p{\nu}} \leq c'  h_n^\beta \lambda_{J_n}^{-1} h_n^{-2} \ln N \implies \norm{f-f_0}_{L^\infty\p{\nu}} \leq \varepsilon_n.
			\]
			Moreover if $f = \sum_{j=1}^J z_j u_j$, using proposition \ref{proposition:uniform_norm_u_j} we get 
			\[
			\norm{f-p_J \p{e^{-t\LL}f_0}}_{L^\infty\p{\nu}} = \norm{\sum_{j=1}^J \p{z_j - \inner{u_j | e^{-t\LL} f_0}_{L^2\p{\nu}}} u_j}_{L^\infty\p{\nu}} \leq N \max_{1 \leq j \leq N} \abs{z_j - \inner{u_j | e^{-t\LL} f_0}_{L^2\p{\nu}}}.
			\]
			Hence , 
			\begin{align*}
			\Pi \c{ \left. \norm{f-p_{J_n} \p{e^{-t\LL}f_0}}_{L^\infty\p{\nu}} \leq c'  h_n^\beta \lambda_{J_n}^{-1} h_n^{-2} \ln N \right| J_n, h_n} \geq & \Pi \c{ \max_{1 \leq j \leq N} \abs{z_j - \inner{u_j | e^{-t\LL} f_0}_{L^2\p{\nu}}} \leq \frac{c'  h_n^\beta \lambda_{J_n}^{-1} h_n^{-2} \ln N}{N} } \\
			\geq & \prod_{j=1}^{J_n} \Pi \c{ \abs{z_j - \inner{u_j | e^{-t\LL} f_0}_{L^2\p{\nu}}} \leq \frac{c'  h_n^\beta \lambda_{J_n}^{-1} h_n^{-2} \ln N}{N} } \\
			\geq & \prod_{j=1}^{J_n} \Pi \c{ \abs{z_j - \inner{u_j | e^{-t\LL} f_0}_{L^2\p{\nu}}} \leq \frac{\frac{c'}{2}  h_n^\beta \ln N}{N} } \\
			\geq & \p{\frac{c'  h_n^\beta \ln N}{N} \inf_{[-K,K]} \Psi }^{J_n} \text{vol} ( B_{\RR^{J_n}} ( 0, 1 ) ),
			\end{align*}
			where we have used $\lambda_{J_n} \leq 2h_n^{-2}$, which holds because $\norm{\LL}_{\LL\p{L^\infty\p{\nu}}} \leq 2$, and where $K > 0$ is chosen to be larger than
			\[
			\max_{1 \leq j \leq J_n} \abs{\inner{u_j | e^{-t\LL} f_0}_{L^2\p{\nu}}} + \frac{c' h_n^\beta\ln N}{N} \lesssim 1 + \underbrace{\norm{e^{-t\LL}}_{\LL\p{L^\infty \p{\nu}}}}_{=1}\norm{f_0}_{L^\infty \p{\MM}} < +\infty.
			\]
			Hence
			\[
			-\ln \Pi \c{ \left. \norm{f-p_J \p{ e^{-t\LL}f_0}}_{L^\infty\p{\nu}} \leq c'  h_n^\beta \lambda_{J_n}^{-1} h_n^{-2} \ln N  \right|  J_n, h_n   } \lesssim J_n \ln N
			\]
			Now $\varepsilon_n \geq \sqrt{\frac{J_n \ln N}{n}} \iff J_n \ln N \leq n\varepsilon_n^2$, therefore 
			\begin{equation*}
			\begin{split}
			\mathbb P_0 &\left( \Pi \c{ \left. \norm{f-p_J \p{e^{-t\LL}f_0}}_{L^\infty\p{\nu}} \leq \varepsilon_n \right|  J_n, h_n  } < \exp \p{-n\varepsilon_n^2} \right)  \leq 
			\mathbb P_0 \left(\norm{f_0-p_{J_n} \p{e^{-t\LL}f_0}}_{L^\infty\p{\nu}} > \varepsilon_n /2 \right) \leq n^{-H}
			\end{split}
			\end{equation*}
			for any $H>0$.
			\item The case $\beta >2$ is similar, the only difference being that we use the bound
			\[
			\abs{\inner{u_j| e^{-t\LL}f_t}_{L^2\p{\nu}}} \leq \norm{\sum_{l=0}^k \frac{\p{t\LL}^l e^{-t\LL}}{l!}}_{\LL\p{L^2\p{\nu}}} \norm{f_0}_{L^2\p{\nu}} \leq \norm{f_0}_{L^\infty\p{\MM}} \sup_{s > 0} \sum_{l=0}^k \frac{s^l e^{-s}}{l!} < +\infty
			\]
		\end{enumerate}
	\end{proof}

	\begin{lem}\label{lem:entropy}
		For some $z \geq 1$ and $J \in \b{1,\ldots,N}$ let
		\[
		\FF_{z,J,h} = \b{\sum_{j=1}^J a_j u_j(h) : \abs{a_j} \leq z} \subset \RR^{V}, \quad z>0
		\]
		Then 
		\[
		\Pi \c{\FF_{z,J,h}^c|J,h} \leq J e^{-b_1z^{b_2}}
		\]
		and, there exists $\varepsilon_1, C_1 > 0$ such that for any $\varepsilon \leq \varepsilon_1$
		\[
		N \p{\varepsilon, \FF_{z,J,h}, \norm{\cdot}_{L^\infty \p{\nu}}} \leq \p{\frac{C_1 z N}{\varepsilon}}^{2J}
		\]
	\end{lem}
	\begin{proof}{Proof of lemma \ref{lem:entropy}}\\
		We drope the subscript $h$ for ease of notation. Using the representation $f = \sum_{j=1}^J Z_j u_j$, the upper bound
		\[
		\Pi \c{\FF_{z,J}^c | J,h} \leq \sum_{j=1}^J \Pi \c{\abs{Z_j} > z | J,h} \leq J \Psi \p{[-z,z]^c} \leq J e^{-b_1z^{b_2}}
		\]
		follows by the assumption on the tails of $\Psi$. Moreover by proposition \ref{proposition:uniform_norm_u_j}, if $\mathcal{E}$ is an $\frac{\varepsilon}{JN}$-net of $[-z,z]^J$ then
		$
		\b{\sum_{j=1}^J f_j u_j : f \in \mathcal{E}}
		$
		is an $\varepsilon$-net of $\FF_{z,J}$ (in $L^{\infty}(\nu)$). Since we can always take 
		\[
		\# \mathcal{E} \leq \p{\frac{C_1 z JN}{\varepsilon}}^J \leq \p{\frac{C_1 z N^2}{\varepsilon}}^J \leq \p{\frac{C_1^{1/2} z^{1/2} N}{\varepsilon^{1/2}}}^{2J} \leq \p{\frac{C_1^{1/2} c N}{\varepsilon}}^{2J}
		\] 
		for $\varepsilon \leq \varepsilon_1$ and some $C_1, \varepsilon_1 > 0$ this gives the result by changing $C_1^{1/2}$ into a new $C_1$.
	\end{proof}

\end{document}